\documentclass[12pt]{article}

\usepackage{amsmath}
\usepackage{amsthm}
\usepackage{amssymb}
\usepackage{cite}
\usepackage{url}
\usepackage{footmisc}
\usepackage{graphicx}
\usepackage{epstopdf}
\usepackage{chngcntr}
\usepackage{enumitem}
\usepackage{hhline}
\usepackage{array}
\usepackage{hyperref}
\usepackage{color}
\usepackage{multirow}
\usepackage{tabularx}
\usepackage{tikz}

\hypersetup{colorlinks=false}

\makeatletter
\let\@fnsymbol\@arabic
\makeatother

\tikzset{every picture/.style={line width=0.75pt}} 
\usetikzlibrary{patterns}

\tikzset{
    pattern size/.store in=\mcSize, 
    pattern size = 5pt,
    pattern thickness/.store in=\mcThickness, 
    pattern thickness = 0.3pt,
    pattern radius/.store in=\mcRadius, 
    pattern radius = 1pt
}
\makeatletter
\pgfutil@ifundefined{pgf@pattern@name@_7gn8i7dj0}{
    \pgfdeclarepatternformonly[\mcThickness,\mcSize]{_7gn8i7dj0}
    {\pgfqpoint{0pt}{0pt}}
    {\pgfpoint{\mcSize+\mcThickness}{\mcSize+\mcThickness}}
    {\pgfpoint{\mcSize}{\mcSize}}
    {
    \pgfsetcolor{\tikz@pattern@color}
    \pgfsetlinewidth{\mcThickness}
    \pgfpathmoveto{\pgfqpoint{0pt}{0pt}}
    \pgfpathlineto{\pgfpoint{\mcSize+\mcThickness}{\mcSize+\mcThickness}}
    \pgfusepath{stroke}
    }
}
\makeatother

\if@twoside
\else 
\fi

\numberwithin{equation}{section}
\counterwithin{figure}{section}
\counterwithin{table}{section}

\theoremstyle{plain}% default
\newtheorem{theorem}{Theorem}[section]

\newtheorem{proposition}[theorem]{Proposition}

\newtheorem{algo}{Algorithm}

\theoremstyle{definition}

\theoremstyle{remark}
\newtheorem*{remark}{Remark}

\newcommand{\abs}[1]{\left\vert#1\right\vert}

\newcommand{\kl}[1]{\left(#1\right)}
\newcommand{\Kl}[1]{\left\{#1\right\}}

\newcommand{\R}{\mathbb{R}} 
\newcommand{\C}{\mathbb{C}}

\newcommand{\Z}{\mathbb{Z}}

\newcommand{\F}{\mathcal{F}}
\newcommand{\FI}{\mathcal{F}^{-1}}
\newcommand{\xv}{\boldsymbol{x}}
\newcommand{\xiv}{\boldsymbol{\xi}}

\newcommand{\sv}{\boldsymbol{s}}
\newcommand{\svjk}{\sv_{j,k}}
\newcommand{\svzz}{\sv_{0,0}}

\newcommand{\phijk}{\phi_{j,k}}

\newcommand{\vphi}{\varphi}

\newcommand{\Tjk}{T_{j,k}}

\newcommand{\Ojk}{\Omega_{j,k}}

\newcommand{\cjk}{c_{j,k}}

\renewcommand{\Im}{I^m}

\newcommand{\Imjk}{I^m_{j,k}}

\newcommand{\Em}{E^m}
\newcommand{\Emjk}{E^m_{j,k}}
\newcommand{\Emzz}{E^m_{0,0}}
\newcommand{\Emp}{E^m_p}

\newcommand{\Rt}{{\R^2}}

\newcommand{\phiw}{\phi_{w}}
\newcommand{\phir}{\phi_{r}}
\newcommand{\phiwd}{\phi_{w}^\delta}
\renewcommand{\mod}{\, \operatorname{mod}}
\newcommand{\OTF}{\operatorname{OTF}}

\title{On Phase Unwrapping via Digital Wavefront Sensors}
\author{
Simon Hubmer\footnote{Johannes Kepler University Linz, Institute of Industrial Mathematics, Altenbergerstra{\ss}e~69, 4040 Linz, Austria, (simon.hubmer@jku.at), Corresponding author.} ,
Victoria Laidlaw\footnote{Johannes Kepler University Linz, Institute of Industrial Mathematics, Altenbergerstra{\ss}e~69, 4040 Linz, Austria, (victoria.laidlaw@indmath.uni-linz.ac.at)} ,
Ronny Ramlau\footnote{Johannes Kepler University Linz, Institute of Industrial Mathematics, Altenbergerstra{\ss}e~69, 4040 Linz, Austria, (ronny.ramlau@jku.at)}\,\,\textsuperscript{,}\footnote{Johann Radon Institute Linz, Altenbergerstra{\ss}e~69, 4040 Linz, Austria, (ronny.ramlau@ricam.oeaw.ac.at)} ,
\\
Ekaterina Sherina\footnote{University of Vienna, Faculty of Mathematics, Oskar Morgenstern-Platz 1, 1090 Vienna, Austria (ekaterina.sherina@univie.ac.at)}\,\,\textsuperscript{,}\footnote{Christian Doppler Laboratory for Mathematical Modeling and Simulation of Next Generations of Ultrasound Devices (MaMSi), Oskar Morgenstern-Platz 1, 1090 Vienna, Austria} ,
Bernadett Stadler\footnote{Johann Radon Institute for Computational and Applied Mathematics, Altenbergerstra{\ss}e~69, 4040 Linz, Austria, (bernadett.stadler@indmath.uni-linz.ac.at)}
}

\begin{document}

% Include the title
\maketitle

% Abstract
\begin{abstract}

In this paper, we derive a new class of methods for the classic 2D phase unwrapping problem of recovering a phase function from its wrapped form. For this, we consider the wrapped phase as a wavefront aberration in an optical system, and use reconstruction methods for (digital) wavefront sensors for its recovery. The key idea is that mathematically, common wavefront sensors are insensitive to whether an incoming wavefront is wrapped or not. However, typical reconstructors for these sensors are optimized to compute smooth wavefronts. Thus, digitally ``propagating'' a wrapped phase through such a sensor and then applying one of these reconstructors results in a smooth unwrapped phase. First, we show how this principle can be applied to derive phase unwrapping algorithms based on digital Shack-Hartmann and Fourier-type wavefront sensors. Then, we numerically test our approach on an unwrapping problem appearing in a free-space optical communications project currently under development, and compare the results to those obtained with other state-of-the-art algorithms.


% 65J22 - Numerical Analysis - Inverse Problems
% 68U10 - Computing methodologies for image processing
% 78A10 - Optics, electromagnetic theory - Physical Optics

% 65D18 - Numerical aspects of computer graphics, image analysis, and computational geometry
% 85-08 - Computational methods for problems pertaining to astronomy and astrophysics

\end{abstract}

% % % % % % % % % % % % %
% Start of the sections %
% % % % % % % % % % % % %

% % % % % % % % % % % % % %
% Section - Introduction  %
% % % % % % % % % % % % % %
\section{Introduction}\label{sect_intro}

In this paper, we consider the classic 2D phase unwrapping problem \cite{Ghiglia_1998}, i.e.,
    \begin{equation}\label{prob_unwrap}
        \text{Given} \quad \phiw(\xiv) :=  \phi(\xiv) \mod 2 \pi \,, \quad \text{find} \quad \phi(\xiv) \,,
    \end{equation}
where the wrapped phase $\phiw$ and the unknown phase $\phi$ are real-valued scalar functions defined for all $\xiv \in \Omega \subseteq \Rt$; see Figure~\ref{fig_example_wrapping} for an example on a disc-shaped domain $\Omega$. For this to be a meaningful problem, one typically aims to compute (reconstruct) a \emph{smooth} unwrapped phase $\phir$ from the given wrapped phase $\phiw$ as an approximation to the original non-wrapped phase $\phi$, in particular without any (or only few) $2\pi$ jumps. The phase unwrapping problem \eqref{prob_unwrap} has applications in numerous fields such as astronomy, biology, or medical imaging, where measurements of the phase of an electric field are often only available in wrapped form. This is because phase information commonly appears in the exponential form $e^{i\phi}$, from which $\phi$ is only determined up to $2\pi$ jumps. Hence, the phase unwrapping problem \eqref{prob_unwrap} can also be equivalently rephrased as
    \begin{equation}\label{prob_unwrap_exp}
        \text{Given} \quad e^{i \phi(\xiv)}  \,, \quad \text{find} \quad \phi(\xiv) \,.
    \end{equation}

\begin{figure}
    \centering
    \includegraphics[trim = {4cm 3.5cm 4cm 3cm}, clip = true, width=\textwidth]{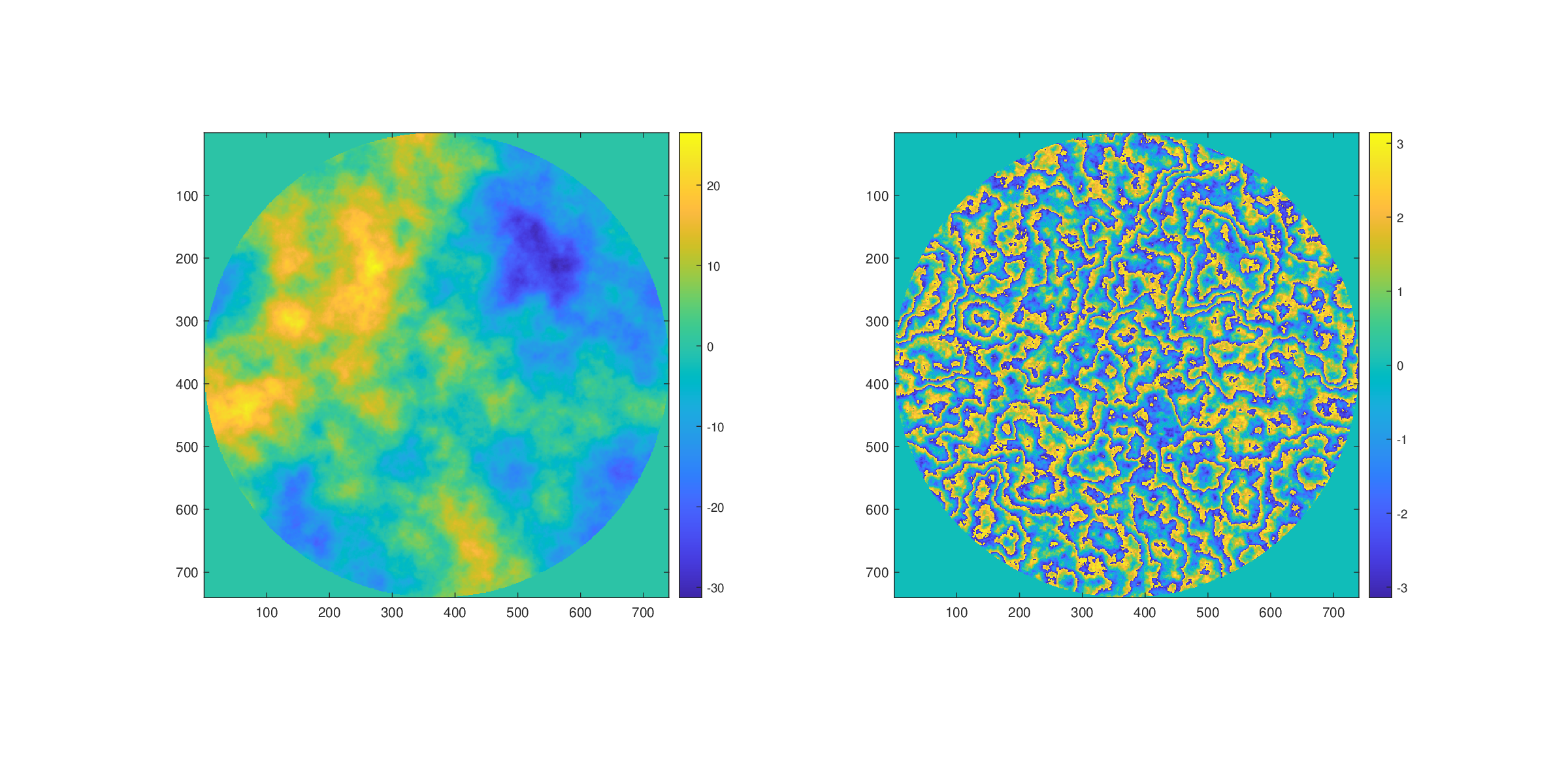}
    \caption{Example of a non-wrapped phase $\phi$ and its wrapped form $\phiw$.}
    \label{fig_example_wrapping}
\end{figure}

Over the years, many different methods have been developed to solve the phase unwrapping problem; see e.g. \cite{Ghiglia_1998,Takeda_1996,Judge_1994,Gens_2003,Kaplan_2007,Huang2022,Wang2022,Zhang2019}. These methods work either in the continuous setting of \eqref{prob_unwrap} and \eqref{prob_unwrap_exp}, or on the actual pixel-discretization of a given wrapped phase, in both cases with the goal to obtain smooth reconstructions without $2\pi$ jumps. In general, one distinguishes between local and global phase unwrapping algorithms. Local methods aim to correct the phase gradient locally, i.e., they follow a certain path and change absolute jumps greater than $\pi$ to their $2\pi$ complement. Such algorithms mainly differ in the choice of the locations where the corrections are applied. A simple strategy is used by the MATLAB built-in function \textit{unwrap()}, which unwraps along the first non-singleton dimension of the phase. More advanced algorithms, such as the one described in \cite{Herraez_2002}, define a reliability function by using a-priori assumptions about the solution, and unwrap the most reliable pixels. Besides local methods, there also exist global phase unwrapping strategies, which aim to reconstruct a phase which satisfies certain a-priori criteria such as smoothness. These are for example based on solving the discrete Poisson equation (PE) \cite{Ghiglia_1989,Ghiglia_1994,Zhao_2020}, or, more recently, on the transport of intensity equation (TIE) \cite{Pandey_2016,Zhao_2019,Martinez_2017}. In recent years, phase unwrapping methods based on deep learning have become popular, see e.g.\ \cite{Huang2022,Wang2022,Zhang2019}. For a more detailed discussion of different phase unwrapping methodologies we refer to Section~\ref{subsect_impl}.

In this paper, we consider a new class of phase unwrapping methods, which are based on mathematical observations of the working principles of certain types of wavefront sensors (WFSs). These sensors are typically used to measure optical wavefront aberrations, and are an integral part of adaptive optics (AO) systems used e.g.\ in astronomy or medical imaging \cite{Roddier_1999,Pircher_Zawadzki_2017}. Within the commonly used framework of Fourier optics, wavefront aberrations $\phi$ usually enter models of imaging systems in exponential factors of the form $e^{i\phi}$ in the Fourier plane of the optical system. Using various optical elements, such as lenslet arrays or pyramidal prisms, wavefront sensors obtain implicit measurements of $\phi$ (e.g.\ its gradient $\nabla_{\xiv} \, \phi$). Using these measurements, the wavefront aberration $\phi$ is then computed via suitable reconstruction algorithms. Crucially, since wavefront aberrations only enter in these sensors via the exponential factor $e^{i\phi}$, the obtained measurements are the same independently of whether $\phi$ contains $2\pi$ jumps or not. This leads directly to the following key idea for phase unwrapping: One can take a wrapped phase $\phiw$ and digitally propagate it through a suitable wavefront sensor, using standard tools of Fourier optics. This provides virtual wavefront measurements, from which one computes $\phi$ using any reconstruction algorithm for the corresponding wavefront sensor. Since these algorithms are designed to produce smooth wavefront reconstructions, the result of this procedure will then be an unwrapped phase $\phi$. The fact that these algorithms are typically optimized with respect to speed and reconstruction quality for aberrations with certain properties then also directly translates to wrapped phases with similar properties.

The aim of this paper is to present the idea of phase unwrapping using digital wavefront sensors outlined above in a detailed and  mathematically rigorous way. In particular, we focus on phase unwrapping using digital Shack-Hartmann and Fourier-type wavefront sensors, which are commonly employed in adaptive optics both in astronomy and ophthalmology. The latter is a very broad class of sensors, and includes for example the well-known 3- and 4-sided pyramid wavefront sensors (PWFS). Both the Shack-Hartmann and Fourier-type wavefront sensors have been successfully used for reconstructing wavefront aberrations resulting from atmospheric turbulence, and thus their digital versions should be well suited for unwrapping atmospheric turbulence profiles. This is relevant e.g.\ in free-space optical communications (FSOC) \cite{Calvo_etal_2019,Chan_2006}, where such turbulence profiles need to be unwrapped repeatedly and accurately. Hence, the numerical experiments presented in this paper have a strong focus on this relevant application problem. In particular, we shall see that our approach matches, and in some cases outperforms, the reconstruction quality of currently used state-of-the-art methods, and this also at a comparable computational cost. 

The outline of this paper is as follows: In Section~\ref{sect_phybackg} we review the necessary physical background on wavefront sensors, which we then use in Section~\ref{sect_digWFS} to introduce a new class of phase unwrapping methods based on digital wavefront sensors. These methods are then numerically tested and compared to state-of-the-art algorithms in Section~\ref{sect_num_exp}, which is followed by a short conclusion in Section~\ref{sect_concl}.

% % % % % % % % % % % % % % % % % % % % % % % % % % % %
% Section - Physical background on wavefront sensors  %
% % % % % % % % % % % % % % % % % % % % % % % % % % % %
\section{Physical background on wavefront sensors}\label{sect_phybackg}

In this section, we review some necessary background on WFSs, which are a key component of current and upcoming AO systems. In general, AO systems consist of one or more WFSs and deformable mirrors (DMs), and are designed to perform real-time corrections of dynamic wavefront aberrations. The aberrations are corrected by converting wavefront measurements obtained by the WFSs into DM actuator commands such that after reflection on the DMs, an incoming wavefront is restored to its non-aberrated (typically planar) form. For the determination of optimal DM actuator commands, precise phase measurements across the entire wavefront are essential. Hence, reconstructing a wavefront from WFS data and calculating DM actuator commands are two crucial inverse problem in AO. 

WFSs have been commonly applied in astronomical AO systems correcting for wavefront aberrations caused by atmospheric turbulence~\cite{Roddier_1999,RoWe96}. For ground-based telescopes, AO systems are employed to perform diffraction-limited observations. With the new generation of Extremely Large Telescopes (ELTs), the development of advanced AO systems has recently gained much attention by researchers. Additionally, AO systems are implemented to enable FSOC, such as optical laser communications between spacecraft and ground-based stations. There, the goal is to equip future space missions with a communication rate orders of magnitude greater than existing channels.
Furthermore, in ophthalmic AO similar instruments are employed for the early detection of anomalies and  diseases in medical diagnostics. Recently, AO systems have also been applied for high- or super-resolution imaging of biological structure, and function in a wide range of microscopic modalities~\cite{AO_micro_2023}.

Depending on the type of AO instrument and its application, reference light sources are typically required for correction. In astronomy, these can either be astronomical or artificial stars, while in other applications lasers are commonly used. Furthermore, the source of aberrations can also differ between applications. For example, while in astronomical imaging aberrations are associated to atmospheric turbulence~\cite{Roddier_1999,RoWe96}, in retinal imaging distortions of the laser beam are caused by imperfect light propagation through the cornea, lens, and vitreous body of the eye~\cite{Liang97}.

% % % % % % % % % % % % % % % % % % % % % % % % % % %
% Subsection - The Shack-Hartmann wavefront sensor  %
% % % % % % % % % % % % % % % % % % % % % % % % % % %
\subsection{The Shack-Hartmann wavefront sensor}\label{subsect_SHWFS}

% Figure - Shack-Hartmann WFS 
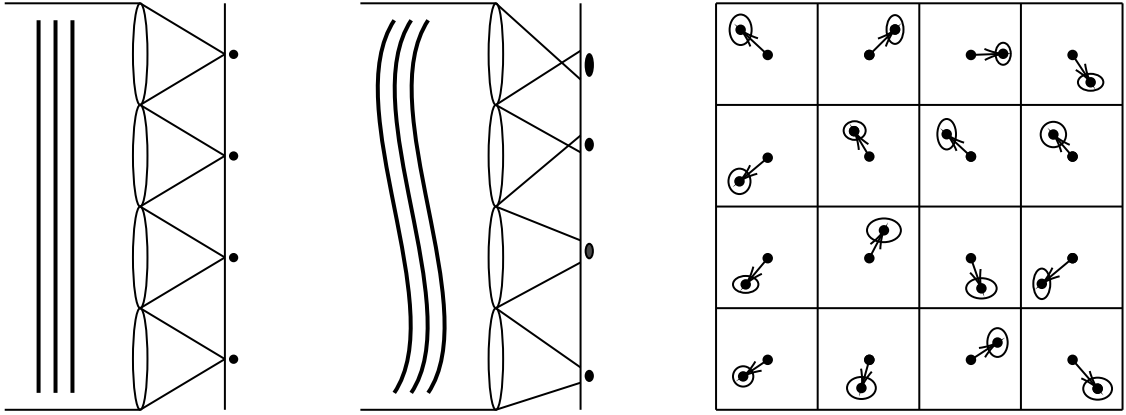
\begin{figure}[ht!]
    \centering
    \begin{tikzpicture}[x=0.75pt,y=0.75pt,yscale=-0.85,xscale=0.85]
    
        % Ellipses and Circles
        \draw  [fill={rgb, 255:red, 0; green, 0; blue, 0 }  ,fill opacity=1 ] (363.01,56.42) .. controls (363.01,52.88) and (363.97,50) .. (365.16,50) .. controls (366.34,50) and (367.3,52.88) .. (367.3,56.42) .. controls (367.3,59.97) and (366.34,62.84) .. (365.16,62.84) .. controls (363.97,62.84) and (363.01,59.97) .. (363.01,56.42) -- cycle ;
        \draw  [fill={rgb, 255:red, 74; green, 74; blue, 74 }  ,fill opacity=1 ] (363,166.28) .. controls (363,163.92) and (363.96,162) .. (365.14,162) .. controls (366.32,162) and (367.28,163.92) .. (367.28,166.28) .. controls (367.28,168.65) and (366.32,170.56) .. (365.14,170.56) .. controls (363.96,170.56) and (363,168.65) .. (363,166.28) -- cycle ;
        \draw  [fill={rgb, 255:red, 0; green, 0; blue, 0 }  ,fill opacity=1 ] (363,240) .. controls (363,238.34) and (363.96,237) .. (365.14,237) .. controls (366.32,237) and (367.28,238.34) .. (367.28,240) .. controls (367.28,241.65) and (366.32,242.99) .. (365.14,242.99) .. controls (363.96,242.99) and (363,241.65) .. (363,240) -- cycle ;
        \draw   (95.71,110) .. controls (95.71,93.43) and (97.63,80) .. (100,80) .. controls (102.37,80) and (104.29,93.43) .. (104.29,110) .. controls (104.29,126.57) and (102.37,140) .. (100,140) .. controls (97.63,140) and (95.71,126.57) .. (95.71,110) -- cycle ;
        \draw   (95.71,170) .. controls (95.71,153.43) and (97.63,140) .. (100,140) .. controls (102.37,140) and (104.29,153.43) .. (104.29,170) .. controls (104.29,186.57) and (102.37,200) .. (100,200) .. controls (97.63,200) and (95.71,186.57) .. (95.71,170) -- cycle ;
        \draw   (95.71,230) .. controls (95.71,213.43) and (97.63,200) .. (100,200) .. controls (102.37,200) and (104.29,213.43) .. (104.29,230) .. controls (104.29,246.57) and (102.37,260) .. (100,260) .. controls (97.63,260) and (95.71,246.57) .. (95.71,230) -- cycle ;
        \draw   (305.71,50) .. controls (305.71,33.43) and (307.63,20) .. (310,20) .. controls (312.37,20) and (314.29,33.43) .. (314.29,50) .. controls (314.29,66.57) and (312.37,80) .. (310,80) .. controls (307.63,80) and (305.71,66.57) .. (305.71,50) -- cycle ;
        \draw   (305.71,110) .. controls (305.71,93.43) and (307.63,80) .. (310,80) .. controls (312.37,80) and (314.29,93.43) .. (314.29,110) .. controls (314.29,126.57) and (312.37,140) .. (310,140) .. controls (307.63,140) and (305.71,126.57) .. (305.71,110) -- cycle ;
        \draw   (305.71,170) .. controls (305.71,153.43) and (307.63,140) .. (310,140) .. controls (312.37,140) and (314.29,153.43) .. (314.29,170) .. controls (314.29,186.57) and (312.37,200) .. (310,200) .. controls (307.63,200) and (305.71,186.57) .. (305.71,170) -- cycle ;
        \draw   (305.71,230) .. controls (305.71,213.43) and (307.63,200) .. (310,200) .. controls (312.37,200) and (314.29,213.43) .. (314.29,230) .. controls (314.29,246.57) and (312.37,260) .. (310,260) .. controls (307.63,260) and (305.71,246.57) .. (305.71,230) -- cycle ;
        \draw  [color={rgb, 255:red, 0; green, 0; blue, 0 }  ,draw opacity=1 ][fill={rgb, 255:red, 0; green, 0; blue, 0 }  ,fill opacity=1 ] (468,50.5) .. controls (468,49.12) and (469.12,48) .. (470.5,48) .. controls (471.88,48) and (473,49.12) .. (473,50.5) .. controls (473,51.88) and (471.88,53) .. (470.5,53) .. controls (469.12,53) and (468,51.88) .. (468,50.5) -- cycle ;
        \draw  [fill={rgb, 255:red, 0; green, 0; blue, 0 }  ,fill opacity=1 ] (528,50.5) .. controls (528,49.12) and (529.12,48) .. (530.5,48) .. controls (531.88,48) and (533,49.12) .. (533,50.5) .. controls (533,51.88) and (531.88,53) .. (530.5,53) .. controls (529.12,53) and (528,51.88) .. (528,50.5) -- cycle ;
        \draw  [fill={rgb, 255:red, 0; green, 0; blue, 0 }  ,fill opacity=1 ] (588,50.5) .. controls (588,49.12) and (589.12,48) .. (590.5,48) .. controls (591.88,48) and (593,49.12) .. (593,50.5) .. controls (593,51.88) and (591.88,53) .. (590.5,53) .. controls (589.12,53) and (588,51.88) .. (588,50.5) -- cycle ;
        \draw  [fill={rgb, 255:red, 0; green, 0; blue, 0 }  ,fill opacity=1 ] (648,50.5) .. controls (648,49.12) and (649.12,48) .. (650.5,48) .. controls (651.88,48) and (653,49.12) .. (653,50.5) .. controls (653,51.88) and (651.88,53) .. (650.5,53) .. controls (649.12,53) and (648,51.88) .. (648,50.5) -- cycle ;
        \draw  [fill={rgb, 255:red, 0; green, 0; blue, 0 }  ,fill opacity=1 ] (468,111.17) .. controls (468,109.79) and (469.12,108.67) .. (470.5,108.67) .. controls (471.88,108.67) and (473,109.79) .. (473,111.17) .. controls (473,112.55) and (471.88,113.67) .. (470.5,113.67) .. controls (469.12,113.67) and (468,112.55) .. (468,111.17) -- cycle ;
        \draw  [fill={rgb, 255:red, 0; green, 0; blue, 0 }  ,fill opacity=1 ] (528,110.5) .. controls (528,109.12) and (529.12,108) .. (530.5,108) .. controls (531.88,108) and (533,109.12) .. (533,110.5) .. controls (533,111.88) and (531.88,113) .. (530.5,113) .. controls (529.12,113) and (528,111.88) .. (528,110.5) -- cycle ;
        \draw  [fill={rgb, 255:red, 0; green, 0; blue, 0 }  ,fill opacity=1 ] (588,110.5) .. controls (588,109.12) and (589.12,108) .. (590.5,108) .. controls (591.88,108) and (593,109.12) .. (593,110.5) .. controls (593,111.88) and (591.88,113) .. (590.5,113) .. controls (589.12,113) and (588,111.88) .. (588,110.5) -- cycle ;
        \draw  [fill={rgb, 255:red, 0; green, 0; blue, 0 }  ,fill opacity=1 ] (648,110.5) .. controls (648,109.12) and (649.12,108) .. (650.5,108) .. controls (651.88,108) and (653,109.12) .. (653,110.5) .. controls (653,111.88) and (651.88,113) .. (650.5,113) .. controls (649.12,113) and (648,111.88) .. (648,110.5) -- cycle ;
        \draw  [fill={rgb, 255:red, 0; green, 0; blue, 0 }  ,fill opacity=1 ] (468,170.5) .. controls (468,169.12) and (469.12,168) .. (470.5,168) .. controls (471.88,168) and (473,169.12) .. (473,170.5) .. controls (473,171.88) and (471.88,173) .. (470.5,173) .. controls (469.12,173) and (468,171.88) .. (468,170.5) -- cycle ;
        \draw  [fill={rgb, 255:red, 0; green, 0; blue, 0 }  ,fill opacity=1 ] (528,170.5) .. controls (528,169.12) and (529.12,168) .. (530.5,168) .. controls (531.88,168) and (533,169.12) .. (533,170.5) .. controls (533,171.88) and (531.88,173) .. (530.5,173) .. controls (529.12,173) and (528,171.88) .. (528,170.5) -- cycle ;
        \draw  [fill={rgb, 255:red, 0; green, 0; blue, 0 }  ,fill opacity=1 ] (588,170.5) .. controls (588,169.12) and (589.12,168) .. (590.5,168) .. controls (591.88,168) and (593,169.12) .. (593,170.5) .. controls (593,171.88) and (591.88,173) .. (590.5,173) .. controls (589.12,173) and (588,171.88) .. (588,170.5) -- cycle ;
        \draw  [fill={rgb, 255:red, 0; green, 0; blue, 0 }  ,fill opacity=1 ] (648,170.5) .. controls (648,169.12) and (649.12,168) .. (650.5,168) .. controls (651.88,168) and (653,169.12) .. (653,170.5) .. controls (653,171.88) and (651.88,173) .. (650.5,173) .. controls (649.12,173) and (648,171.88) .. (648,170.5) -- cycle ;
        \draw  [fill={rgb, 255:red, 0; green, 0; blue, 0 }  ,fill opacity=1 ] (468,230.5) .. controls (468,229.12) and (469.12,228) .. (470.5,228) .. controls (471.88,228) and (473,229.12) .. (473,230.5) .. controls (473,231.88) and (471.88,233) .. (470.5,233) .. controls (469.12,233) and (468,231.88) .. (468,230.5) -- cycle ;
        \draw  [fill={rgb, 255:red, 0; green, 0; blue, 0 }  ,fill opacity=1 ] (528,230.5) .. controls (528,229.12) and (529.12,228) .. (530.5,228) .. controls (531.88,228) and (533,229.12) .. (533,230.5) .. controls (533,231.88) and (531.88,233) .. (530.5,233) .. controls (529.12,233) and (528,231.88) .. (528,230.5) -- cycle ;
        \draw  [fill={rgb, 255:red, 0; green, 0; blue, 0 }  ,fill opacity=1 ] (588,230.5) .. controls (588,229.12) and (589.12,228) .. (590.5,228) .. controls (591.88,228) and (593,229.12) .. (593,230.5) .. controls (593,231.88) and (591.88,233) .. (590.5,233) .. controls (589.12,233) and (588,231.88) .. (588,230.5) -- cycle ;
        \draw  [fill={rgb, 255:red, 0; green, 0; blue, 0 }  ,fill opacity=1 ] (648,230.5) .. controls (648,229.12) and (649.12,228) .. (650.5,228) .. controls (651.88,228) and (653,229.12) .. (653,230.5) .. controls (653,231.88) and (651.88,233) .. (650.5,233) .. controls (649.12,233) and (648,231.88) .. (648,230.5) -- cycle ;
        \draw  [fill={rgb, 255:red, 255; green, 255; blue, 255 }  ,fill opacity=1 ] (448,35.63) .. controls (448,30.66) and (450.91,26.63) .. (454.5,26.63) .. controls (458.09,26.63) and (461,30.66) .. (461,35.63) .. controls (461,40.6) and (458.09,44.63) .. (454.5,44.63) .. controls (450.91,44.63) and (448,40.6) .. (448,35.63) -- cycle ;
        \draw  [fill={rgb, 255:red, 255; green, 255; blue, 255 }  ,fill opacity=1 ] (540.67,35.47) .. controls (540.67,30.77) and (542.91,26.97) .. (545.67,26.97) .. controls (548.43,26.97) and (550.67,30.77) .. (550.67,35.47) .. controls (550.67,40.16) and (548.43,43.97) .. (545.67,43.97) .. controls (542.91,43.97) and (540.67,40.16) .. (540.67,35.47) -- cycle ;
        \draw  [fill={rgb, 255:red, 255; green, 255; blue, 255 }  ,fill opacity=1 ] (605,49.8) .. controls (605,46.21) and (607.01,43.3) .. (609.5,43.3) .. controls (611.99,43.3) and (614,46.21) .. (614,49.8) .. controls (614,53.39) and (611.99,56.3) .. (609.5,56.3) .. controls (607.01,56.3) and (605,53.39) .. (605,49.8) -- cycle ;
        \draw  [fill={rgb, 255:red, 255; green, 255; blue, 255 }  ,fill opacity=1 ] (653.67,66.63) .. controls (653.67,63.87) and (657.02,61.63) .. (661.17,61.63) .. controls (665.31,61.63) and (668.67,63.87) .. (668.67,66.63) .. controls (668.67,69.39) and (665.31,71.63) .. (661.17,71.63) .. controls (657.02,71.63) and (653.67,69.39) .. (653.67,66.63) -- cycle ;
        \draw  [fill={rgb, 255:red, 255; green, 255; blue, 255 }  ,fill opacity=1 ] (447.33,125.13) .. controls (447.33,120.99) and (450.24,117.63) .. (453.83,117.63) .. controls (457.42,117.63) and (460.33,120.99) .. (460.33,125.13) .. controls (460.33,129.28) and (457.42,132.63) .. (453.83,132.63) .. controls (450.24,132.63) and (447.33,129.28) .. (447.33,125.13) -- cycle ;
        \draw  [fill={rgb, 255:red, 255; green, 255; blue, 255 }  ,fill opacity=1 ] (515.33,95.13) .. controls (515.33,92.1) and (518.24,89.63) .. (521.83,89.63) .. controls (525.42,89.63) and (528.33,92.1) .. (528.33,95.13) .. controls (528.33,98.17) and (525.42,100.63) .. (521.83,100.63) .. controls (518.24,100.63) and (515.33,98.17) .. (515.33,95.13) -- cycle ;
        \draw  [fill={rgb, 255:red, 255; green, 255; blue, 255 }  ,fill opacity=1 ] (570.67,97.3) .. controls (570.67,92.33) and (573.13,88.3) .. (576.17,88.3) .. controls (579.2,88.3) and (581.67,92.33) .. (581.67,97.3) .. controls (581.67,102.27) and (579.2,106.3) .. (576.17,106.3) .. controls (573.13,106.3) and (570.67,102.27) .. (570.67,97.3) -- cycle ;
        \draw  [fill={rgb, 255:red, 255; green, 255; blue, 255 }  ,fill opacity=1 ] (631.67,97.47) .. controls (631.67,93.32) and (635.02,89.97) .. (639.17,89.97) .. controls (643.31,89.97) and (646.67,93.32) .. (646.67,97.47) .. controls (646.67,101.61) and (643.31,104.97) .. (639.17,104.97) .. controls (635.02,104.97) and (631.67,101.61) .. (631.67,97.47) -- cycle ;
        \draw  [fill={rgb, 255:red, 255; green, 255; blue, 255 }  ,fill opacity=1 ] (450,185.97) .. controls (450,183.21) and (453.36,180.97) .. (457.5,180.97) .. controls (461.64,180.97) and (465,183.21) .. (465,185.97) .. controls (465,188.73) and (461.64,190.97) .. (457.5,190.97) .. controls (453.36,190.97) and (450,188.73) .. (450,185.97) -- cycle ;
        \draw  [fill={rgb, 255:red, 255; green, 255; blue, 255 }  ,fill opacity=1 ] (529.13,154.03) .. controls (529.13,150.17) and (533.61,147.03) .. (539.13,147.03) .. controls (544.66,147.03) and (549.13,150.17) .. (549.13,154.03) .. controls (549.13,157.9) and (544.66,161.03) .. (539.13,161.03) .. controls (533.61,161.03) and (529.13,157.9) .. (529.13,154.03) -- cycle ;
        \draw  [fill={rgb, 255:red, 255; green, 255; blue, 255 }  ,fill opacity=1 ] (587.67,188.3) .. controls (587.67,184.99) and (591.7,182.3) .. (596.67,182.3) .. controls (601.64,182.3) and (605.67,184.99) .. (605.67,188.3) .. controls (605.67,191.61) and (601.64,194.3) .. (596.67,194.3) .. controls (591.7,194.3) and (587.67,191.61) .. (587.67,188.3) -- cycle ;
        \draw  [fill={rgb, 255:red, 255; green, 255; blue, 255 }  ,fill opacity=1 ] (627.33,185.63) .. controls (627.33,180.66) and (629.57,176.63) .. (632.33,176.63) .. controls (635.09,176.63) and (637.33,180.66) .. (637.33,185.63) .. controls (637.33,190.6) and (635.09,194.63) .. (632.33,194.63) .. controls (629.57,194.63) and (627.33,190.6) .. (627.33,185.63) -- cycle ;
        \draw  [fill={rgb, 255:red, 255; green, 255; blue, 255 }  ,fill opacity=1 ] (450,240.3) .. controls (450,236.99) and (452.69,234.3) .. (456,234.3) .. controls (459.31,234.3) and (462,236.99) .. (462,240.3) .. controls (462,243.61) and (459.31,246.3) .. (456,246.3) .. controls (452.69,246.3) and (450,243.61) .. (450,240.3) -- cycle ;
        \draw  [fill={rgb, 255:red, 255; green, 255; blue, 255 }  ,fill opacity=1 ] (517.33,247.13) .. controls (517.33,243.54) and (521.14,240.63) .. (525.83,240.63) .. controls (530.53,240.63) and (534.33,243.54) .. (534.33,247.13) .. controls (534.33,250.72) and (530.53,253.63) .. (525.83,253.63) .. controls (521.14,253.63) and (517.33,250.72) .. (517.33,247.13) -- cycle ;
        \draw  [fill={rgb, 255:red, 255; green, 255; blue, 255 }  ,fill opacity=1 ] (600.17,220.3) .. controls (600.17,215.61) and (602.85,211.8) .. (606.17,211.8) .. controls (609.48,211.8) and (612.17,215.61) .. (612.17,220.3) .. controls (612.17,224.99) and (609.48,228.8) .. (606.17,228.8) .. controls (602.85,228.8) and (600.17,224.99) .. (600.17,220.3) -- cycle ;
        \draw  [fill={rgb, 255:red, 255; green, 255; blue, 255 }  ,fill opacity=1 ] (656.8,247.53) .. controls (656.8,243.94) and (660.61,241.03) .. (665.3,241.03) .. controls (669.99,241.03) and (673.8,243.94) .. (673.8,247.53) .. controls (673.8,251.12) and (669.99,254.03) .. (665.3,254.03) .. controls (660.61,254.03) and (656.8,251.12) .. (656.8,247.53) -- cycle ;
        \draw  [fill={rgb, 255:red, 0; green, 0; blue, 0 }  ,fill opacity=1 ] (519,95.5) .. controls (519,94.12) and (520.12,93) .. (521.5,93) .. controls (522.88,93) and (524,94.12) .. (524,95.5) .. controls (524,96.88) and (522.88,98) .. (521.5,98) .. controls (520.12,98) and (519,96.88) .. (519,95.5) -- cycle ;
        \draw  [color={rgb, 255:red, 0; green, 0; blue, 0 }  ,draw opacity=1 ][fill={rgb, 255:red, 0; green, 0; blue, 0 }  ,fill opacity=1 ] (452,35.63) .. controls (452,34.25) and (453.12,33.13) .. (454.5,33.13) .. controls (455.88,33.13) and (457,34.25) .. (457,35.63) .. controls (457,37.01) and (455.88,38.13) .. (454.5,38.13) .. controls (453.12,38.13) and (452,37.01) .. (452,35.63) -- cycle ;
        \draw  [color={rgb, 255:red, 0; green, 0; blue, 0 }  ,draw opacity=1 ][fill={rgb, 255:red, 0; green, 0; blue, 0 }  ,fill opacity=1 ] (451.33,125.13) .. controls (451.33,123.75) and (452.45,122.63) .. (453.83,122.63) .. controls (455.21,122.63) and (456.33,123.75) .. (456.33,125.13) .. controls (456.33,126.51) and (455.21,127.63) .. (453.83,127.63) .. controls (452.45,127.63) and (451.33,126.51) .. (451.33,125.13) -- cycle ;
        \draw  [color={rgb, 255:red, 0; green, 0; blue, 0 }  ,draw opacity=1 ][fill={rgb, 255:red, 0; green, 0; blue, 0 }  ,fill opacity=1 ] (455,185.97) .. controls (455,184.59) and (456.12,183.47) .. (457.5,183.47) .. controls (458.88,183.47) and (460,184.59) .. (460,185.97) .. controls (460,187.35) and (458.88,188.47) .. (457.5,188.47) .. controls (456.12,188.47) and (455,187.35) .. (455,185.97) -- cycle ;
        \draw  [color={rgb, 255:red, 0; green, 0; blue, 0 }  ,draw opacity=1 ][fill={rgb, 255:red, 0; green, 0; blue, 0 }  ,fill opacity=1 ] (543.17,35.47) .. controls (543.17,34.09) and (544.29,32.97) .. (545.67,32.97) .. controls (547.05,32.97) and (548.17,34.09) .. (548.17,35.47) .. controls (548.17,36.85) and (547.05,37.97) .. (545.67,37.97) .. controls (544.29,37.97) and (543.17,36.85) .. (543.17,35.47) -- cycle ;
        \draw  [color={rgb, 255:red, 0; green, 0; blue, 0 }  ,draw opacity=1 ][fill={rgb, 255:red, 0; green, 0; blue, 0 }  ,fill opacity=1 ] (607,49.8) .. controls (607,48.42) and (608.12,47.3) .. (609.5,47.3) .. controls (610.88,47.3) and (612,48.42) .. (612,49.8) .. controls (612,51.18) and (610.88,52.3) .. (609.5,52.3) .. controls (608.12,52.3) and (607,51.18) .. (607,49.8) -- cycle ;
        \draw  [color={rgb, 255:red, 0; green, 0; blue, 0 }  ,draw opacity=1 ][fill={rgb, 255:red, 0; green, 0; blue, 0 }  ,fill opacity=1 ] (629.83,185.63) .. controls (629.83,184.25) and (630.95,183.13) .. (632.33,183.13) .. controls (633.71,183.13) and (634.83,184.25) .. (634.83,185.63) .. controls (634.83,187.01) and (633.71,188.13) .. (632.33,188.13) .. controls (630.95,188.13) and (629.83,187.01) .. (629.83,185.63) -- cycle ;
        \draw  [color={rgb, 255:red, 0; green, 0; blue, 0 }  ,draw opacity=1 ][fill={rgb, 255:red, 0; green, 0; blue, 0 }  ,fill opacity=1 ] (536.63,154.03) .. controls (536.63,152.65) and (537.75,151.53) .. (539.13,151.53) .. controls (540.51,151.53) and (541.63,152.65) .. (541.63,154.03) .. controls (541.63,155.41) and (540.51,156.53) .. (539.13,156.53) .. controls (537.75,156.53) and (536.63,155.41) .. (536.63,154.03) -- cycle ;
        \draw  [color={rgb, 255:red, 0; green, 0; blue, 0 }  ,draw opacity=1 ][fill={rgb, 255:red, 0; green, 0; blue, 0 }  ,fill opacity=1 ] (573.67,97.3) .. controls (573.67,95.92) and (574.79,94.8) .. (576.17,94.8) .. controls (577.55,94.8) and (578.67,95.92) .. (578.67,97.3) .. controls (578.67,98.68) and (577.55,99.8) .. (576.17,99.8) .. controls (574.79,99.8) and (573.67,98.68) .. (573.67,97.3) -- cycle ;
        \draw  [color={rgb, 255:red, 0; green, 0; blue, 0 }  ,draw opacity=1 ][fill={rgb, 255:red, 0; green, 0; blue, 0 }  ,fill opacity=1 ] (523.33,247.13) .. controls (523.33,245.75) and (524.45,244.63) .. (525.83,244.63) .. controls (527.21,244.63) and (528.33,245.75) .. (528.33,247.13) .. controls (528.33,248.51) and (527.21,249.63) .. (525.83,249.63) .. controls (524.45,249.63) and (523.33,248.51) .. (523.33,247.13) -- cycle ;
        \draw  [color={rgb, 255:red, 0; green, 0; blue, 0 }  ,draw opacity=1 ][fill={rgb, 255:red, 0; green, 0; blue, 0 }  ,fill opacity=1 ] (453.5,240.3) .. controls (453.5,238.92) and (454.62,237.8) .. (456,237.8) .. controls (457.38,237.8) and (458.5,238.92) .. (458.5,240.3) .. controls (458.5,241.68) and (457.38,242.8) .. (456,242.8) .. controls (454.62,242.8) and (453.5,241.68) .. (453.5,240.3) -- cycle ;
        \draw  [color={rgb, 255:red, 0; green, 0; blue, 0 }  ,draw opacity=1 ][fill={rgb, 255:red, 0; green, 0; blue, 0 }  ,fill opacity=1 ] (658.67,66.63) .. controls (658.67,65.25) and (659.79,64.13) .. (661.17,64.13) .. controls (662.55,64.13) and (663.67,65.25) .. (663.67,66.63) .. controls (663.67,68.01) and (662.55,69.13) .. (661.17,69.13) .. controls (659.79,69.13) and (658.67,68.01) .. (658.67,66.63) -- cycle ;
        \draw  [color={rgb, 255:red, 0; green, 0; blue, 0 }  ,draw opacity=1 ][fill={rgb, 255:red, 0; green, 0; blue, 0 }  ,fill opacity=1 ] (636.67,97.47) .. controls (636.67,96.09) and (637.79,94.97) .. (639.17,94.97) .. controls (640.55,94.97) and (641.67,96.09) .. (641.67,97.47) .. controls (641.67,98.85) and (640.55,99.97) .. (639.17,99.97) .. controls (637.79,99.97) and (636.67,98.85) .. (636.67,97.47) -- cycle ;
        \draw  [color={rgb, 255:red, 0; green, 0; blue, 0 }  ,draw opacity=1 ][fill={rgb, 255:red, 0; green, 0; blue, 0 }  ,fill opacity=1 ] (594.17,188.3) .. controls (594.17,186.92) and (595.29,185.8) .. (596.67,185.8) .. controls (598.05,185.8) and (599.17,186.92) .. (599.17,188.3) .. controls (599.17,189.68) and (598.05,190.8) .. (596.67,190.8) .. controls (595.29,190.8) and (594.17,189.68) .. (594.17,188.3) -- cycle ;
        \draw  [color={rgb, 255:red, 0; green, 0; blue, 0 }  ,draw opacity=1 ][fill={rgb, 255:red, 0; green, 0; blue, 0 }  ,fill opacity=1 ] (603.67,220.3) .. controls (603.67,218.92) and (604.79,217.8) .. (606.17,217.8) .. controls (607.55,217.8) and (608.67,218.92) .. (608.67,220.3) .. controls (608.67,221.68) and (607.55,222.8) .. (606.17,222.8) .. controls (604.79,222.8) and (603.67,221.68) .. (603.67,220.3) -- cycle ;
        \draw  [color={rgb, 255:red, 0; green, 0; blue, 0 }  ,draw opacity=1 ][fill={rgb, 255:red, 0; green, 0; blue, 0 }  ,fill opacity=1 ] (662.8,247.53) .. controls (662.8,246.15) and (663.92,245.03) .. (665.3,245.03) .. controls (666.68,245.03) and (667.8,246.15) .. (667.8,247.53) .. controls (667.8,248.91) and (666.68,250.03) .. (665.3,250.03) .. controls (663.92,250.03) and (662.8,248.91) .. (662.8,247.53) -- cycle ;
        \draw  [fill={rgb, 255:red, 0; green, 0; blue, 0 }  ,fill opacity=1 ] (153,50.15) .. controls (153,48.96) and (153.96,48) .. (155.15,48) .. controls (156.33,48) and (157.29,48.96) .. (157.29,50.15) .. controls (157.29,51.33) and (156.33,52.29) .. (155.15,52.29) .. controls (153.96,52.29) and (153,51.33) .. (153,50.15) -- cycle ;
        \draw  [fill={rgb, 255:red, 0; green, 0; blue, 0 }  ,fill opacity=1 ] (363,103.43) .. controls (363,101.53) and (363.96,100) .. (365.14,100) .. controls (366.32,100) and (367.28,101.53) .. (367.28,103.43) .. controls (367.28,105.32) and (366.32,106.85) .. (365.14,106.85) .. controls (363.96,106.85) and (363,105.32) .. (363,103.43) -- cycle ;
        \draw   (95.71,50) .. controls (95.71,33.43) and (97.63,20) .. (100,20) .. controls (102.37,20) and (104.29,33.43) .. (104.29,50) .. controls (104.29,66.57) and (102.37,80) .. (100,80) .. controls (97.63,80) and (95.71,66.57) .. (95.71,50) -- cycle ;
        \draw  [fill={rgb, 255:red, 0; green, 0; blue, 0 }  ,fill opacity=1 ] (153,110.15) .. controls (153,108.96) and (153.96,108) .. (155.15,108) .. controls (156.33,108) and (157.29,108.96) .. (157.29,110.15) .. controls (157.29,111.33) and (156.33,112.29) .. (155.15,112.29) .. controls (153.96,112.29) and (153,111.33) .. (153,110.15) -- cycle ;
        \draw  [fill={rgb, 255:red, 0; green, 0; blue, 0 }  ,fill opacity=1 ] (153,170.15) .. controls (153,168.96) and (153.96,168) .. (155.15,168) .. controls (156.33,168) and (157.29,168.96) .. (157.29,170.15) .. controls (157.29,171.33) and (156.33,172.29) .. (155.15,172.29) .. controls (153.96,172.29) and (153,171.33) .. (153,170.15) -- cycle ;
        \draw  [fill={rgb, 255:red, 0; green, 0; blue, 0 }  ,fill opacity=1 ] (153,230.15) .. controls (153,228.96) and (153.96,228) .. (155.15,228) .. controls (156.33,228) and (157.29,228.96) .. (157.29,230.15) .. controls (157.29,231.33) and (156.33,232.29) .. (155.15,232.29) .. controls (153.96,232.29) and (153,231.33) .. (153,230.15) -- cycle ;
        \draw  [fill={rgb, 255:red, 0; green, 0; blue, 0 }  ,fill opacity=1 ] (528,50.5) .. controls (528,49.12) and (529.12,48) .. (530.5,48) .. controls (531.88,48) and (533,49.12) .. (533,50.5) .. controls (533,51.88) and (531.88,53) .. (530.5,53) .. controls (529.12,53) and (528,51.88) .. (528,50.5) -- cycle ;
        \draw  [fill={rgb, 255:red, 0; green, 0; blue, 0 }  ,fill opacity=1 ] (588,110.5) .. controls (588,109.12) and (589.12,108) .. (590.5,108) .. controls (591.88,108) and (593,109.12) .. (593,110.5) .. controls (593,111.88) and (591.88,113) .. (590.5,113) .. controls (589.12,113) and (588,111.88) .. (588,110.5) -- cycle ;
        \draw  [fill={rgb, 255:red, 0; green, 0; blue, 0 }  ,fill opacity=1 ] (648,110.5) .. controls (648,109.12) and (649.12,108) .. (650.5,108) .. controls (651.88,108) and (653,109.12) .. (653,110.5) .. controls (653,111.88) and (651.88,113) .. (650.5,113) .. controls (649.12,113) and (648,111.88) .. (648,110.5) -- cycle ;
        \draw  [fill={rgb, 255:red, 0; green, 0; blue, 0 }  ,fill opacity=1 ] (648,170.5) .. controls (648,169.12) and (649.12,168) .. (650.5,168) .. controls (651.88,168) and (653,169.12) .. (653,170.5) .. controls (653,171.88) and (651.88,173) .. (650.5,173) .. controls (649.12,173) and (648,171.88) .. (648,170.5) -- cycle ;
        \draw  [fill={rgb, 255:red, 0; green, 0; blue, 0 }  ,fill opacity=1 ] (528,230.5) .. controls (528,229.12) and (529.12,228) .. (530.5,228) .. controls (531.88,228) and (533,229.12) .. (533,230.5) .. controls (533,231.88) and (531.88,233) .. (530.5,233) .. controls (529.12,233) and (528,231.88) .. (528,230.5) -- cycle ;
        \draw  [fill={rgb, 255:red, 0; green, 0; blue, 0 }  ,fill opacity=1 ] (519,95.5) .. controls (519,94.12) and (520.12,93) .. (521.5,93) .. controls (522.88,93) and (524,94.12) .. (524,95.5) .. controls (524,96.88) and (522.88,98) .. (521.5,98) .. controls (520.12,98) and (519,96.88) .. (519,95.5) -- cycle ;
        \draw  [color={rgb, 255:red, 0; green, 0; blue, 0 }  ,draw opacity=1 ][fill={rgb, 255:red, 0; green, 0; blue, 0 }  ,fill opacity=1 ] (543.17,35.47) .. controls (543.17,34.09) and (544.29,32.97) .. (545.67,32.97) .. controls (547.05,32.97) and (548.17,34.09) .. (548.17,35.47) .. controls (548.17,36.85) and (547.05,37.97) .. (545.67,37.97) .. controls (544.29,37.97) and (543.17,36.85) .. (543.17,35.47) -- cycle ;
        \draw  [color={rgb, 255:red, 0; green, 0; blue, 0 }  ,draw opacity=1 ][fill={rgb, 255:red, 0; green, 0; blue, 0 }  ,fill opacity=1 ] (662.8,247.53) .. controls (662.8,246.15) and (663.92,245.03) .. (665.3,245.03) .. controls (666.68,245.03) and (667.8,246.15) .. (667.8,247.53) .. controls (667.8,248.91) and (666.68,250.03) .. (665.3,250.03) .. controls (663.92,250.03) and (662.8,248.91) .. (662.8,247.53) -- cycle ;

        % Straight lines
        \draw (100,260) -- (20,260);
        \draw (150,20) -- (150,260);
        \draw [line width=1.5] (50,30) -- (50,250);
        \draw (310,20) -- (360,65);
        \draw (310,80) -- (360,108);
        \draw (310,140) -- (360,160);
        \draw (310,200) -- (360,235);
        \draw (310,80) -- (360,48);
        \draw [line width=0.75] (310,140) -- (360,98.04);
        \draw (310,200) -- (360,173);
        \draw (360,244) -- (310,260);
        \draw (440,20) -- (440,260);
        \draw (440,20) -- (680,20);
        \draw (680,20) -- (680,260);
        \draw (440,260) -- (680,260);
        \draw (500,20) -- (500,260);
        \draw (560,20) -- (560,260);
        \draw (620,20) -- (620,260);
        \draw (440,80) -- (680,80);
        \draw (440,140) -- (680,140);
        \draw (440,200) -- (680,200);
        \draw (100,80) -- (150,110.05);
        \draw (100,140) -- (150,170.05);
        \draw (100,200) -- (150,230.05);
        \draw (100,80) -- (150,50.05);
        \draw (100,140) -- (150,110.05);
        \draw (100,200) -- (150,170.05);
        \draw (100,260) -- (150,230.05);
        \draw [line width=0.75] (100,20) -- (20,20);
        \draw (100,20) -- (150,50.05);  
        \draw [line width=1.5] (40,30) -- (40,250);
        \draw [line width=1.5] (60,30) -- (60,250);
        \draw (360,20) -- (360,260);
        \draw [line width=0.75] (310,20) -- (230,20);
        \draw [line width=0.75] (310,260) -- (230,260);

        % Curved lines
        \draw [line width=1.5] (270,250) .. controls (303.66,199.65) and (236.3,81.28) .. (270,30);
        \draw [line width=1.5] (260,250) .. controls (293.66,199.65) and (226.3,81.28) .. (260,30);
        \draw [line width=1.5] (250,250) .. controls (283.66,199.65) and (216.3,81.28) .. (250,30);

        % Arrows 
        \draw (650.5,170.5) -- (633.87,184.35) ;
        \draw [shift={(632.33,185.63)}, rotate = 320.2] [color={rgb, 255:red, 0; green, 0; blue, 0 }  ][line width=0.75]    (10.93,-3.29) .. controls (6.95,-1.4) and (3.31,-0.3) .. (0,0) .. controls (3.31,0.3) and (6.95,1.4) .. (10.93,3.29)   ;
        \draw (530.5,230.5) -- (526.37,245.21) ;
        \draw [shift={(525.83,247.13)}, rotate = 285.67] [color={rgb, 255:red, 0; green, 0; blue, 0 }  ][line width=0.75]    (10.93,-3.29) .. controls (6.95,-1.4) and (3.31,-0.3) .. (0,0) .. controls (3.31,0.3) and (6.95,1.4) .. (10.93,3.29)   ;
        \draw (530.5,110.5) -- (522.53,97.21) ;
        \draw [shift={(521.5,95.5)}, rotate = 59.04] [color={rgb, 255:red, 0; green, 0; blue, 0 }  ][line width=0.75]    (10.93,-3.29) .. controls (6.95,-1.4) and (3.31,-0.3) .. (0,0) .. controls (3.31,0.3) and (6.95,1.4) .. (10.93,3.29)   ;
        \draw (470.5,50.5) -- (455.97,36.99) ;
        \draw [shift={(454.5,35.63)}, rotate = 42.9] [color={rgb, 255:red, 0; green, 0; blue, 0 }  ][line width=0.75]    (10.93,-3.29) .. controls (6.95,-1.4) and (3.31,-0.3) .. (0,0) .. controls (3.31,0.3) and (6.95,1.4) .. (10.93,3.29)   ;
        \draw    (530.5,50.5) -- (544.25,36.87) ;
        \draw [shift={(545.67,35.47)}, rotate = 135.25] [color={rgb, 255:red, 0; green, 0; blue, 0 }  ][line width=0.75]    (10.93,-3.29) .. controls (6.95,-1.4) and (3.31,-0.3) .. (0,0) .. controls (3.31,0.3) and (6.95,1.4) .. (10.93,3.29)   ;
        \draw    (590.5,50.5) -- (607.5,49.87) ;
        \draw [shift={(609.5,49.8)}, rotate = 177.89] [color={rgb, 255:red, 0; green, 0; blue, 0 }  ][line width=0.75]    (10.93,-3.29) .. controls (6.95,-1.4) and (3.31,-0.3) .. (0,0) .. controls (3.31,0.3) and (6.95,1.4) .. (10.93,3.29)   ;
        \draw    (470.5,111.17) -- (455.37,123.85) ;
        \draw [shift={(453.83,125.13)}, rotate = 320.04] [color={rgb, 255:red, 0; green, 0; blue, 0 }  ][line width=0.75]    (10.93,-3.29) .. controls (6.95,-1.4) and (3.31,-0.3) .. (0,0) .. controls (3.31,0.3) and (6.95,1.4) .. (10.93,3.29)   ;
        \draw    (650.5,50.5) -- (660.06,64.97) ;
        \draw [shift={(661.17,66.63)}, rotate = 236.53] [color={rgb, 255:red, 0; green, 0; blue, 0 }  ][line width=0.75]    (10.93,-3.29) .. controls (6.95,-1.4) and (3.31,-0.3) .. (0,0) .. controls (3.31,0.3) and (6.95,1.4) .. (10.93,3.29)   ;
        \draw    (470.5,170.5) -- (458.79,184.44) ;
        \draw [shift={(457.5,185.97)}, rotate = 310.05] [color={rgb, 255:red, 0; green, 0; blue, 0 }  ][line width=0.75]    (10.93,-3.29) .. controls (6.95,-1.4) and (3.31,-0.3) .. (0,0) .. controls (3.31,0.3) and (6.95,1.4) .. (10.93,3.29)   ;
        \draw    (530.5,170.5) -- (538.2,155.8) ;
        \draw [shift={(539.13,154.03)}, rotate = 117.67] [color={rgb, 255:red, 0; green, 0; blue, 0 }  ][line width=0.75]    (10.93,-3.29) .. controls (6.95,-1.4) and (3.31,-0.3) .. (0,0) .. controls (3.31,0.3) and (6.95,1.4) .. (10.93,3.29)   ;
        \draw    (590.5,170.5) -- (596.01,186.41) ;
        \draw [shift={(596.67,188.3)}, rotate = 250.89] [color={rgb, 255:red, 0; green, 0; blue, 0 }  ][line width=0.75]    (10.93,-3.29) .. controls (6.95,-1.4) and (3.31,-0.3) .. (0,0) .. controls (3.31,0.3) and (6.95,1.4) .. (10.93,3.29)   ;
        \draw    (590.5,110.5) -- (577.64,98.65) ;
        \draw [shift={(576.17,97.3)}, rotate = 42.64] [color={rgb, 255:red, 0; green, 0; blue, 0 }  ][line width=0.75]    (10.93,-3.29) .. controls (6.95,-1.4) and (3.31,-0.3) .. (0,0) .. controls (3.31,0.3) and (6.95,1.4) .. (10.93,3.29)   ;
        \draw    (650.5,110.5) -- (640.48,98.98) ;
        \draw [shift={(639.17,97.47)}, rotate = 48.99] [color={rgb, 255:red, 0; green, 0; blue, 0 }  ][line width=0.75]    (10.93,-3.29) .. controls (6.95,-1.4) and (3.31,-0.3) .. (0,0) .. controls (3.31,0.3) and (6.95,1.4) .. (10.93,3.29)   ;
        \draw    (470.5,230.5) -- (457.66,239.18) ;
        \draw [shift={(456,240.3)}, rotate = 325.95] [color={rgb, 255:red, 0; green, 0; blue, 0 }  ][line width=0.75]    (10.93,-3.29) .. controls (6.95,-1.4) and (3.31,-0.3) .. (0,0) .. controls (3.31,0.3) and (6.95,1.4) .. (10.93,3.29)   ;
        \draw    (590.5,230.5) -- (604.49,221.39) ;
        \draw [shift={(606.17,220.3)}, rotate = 146.93] [color={rgb, 255:red, 0; green, 0; blue, 0 }  ][line width=0.75]    (10.93,-3.29) .. controls (6.95,-1.4) and (3.31,-0.3) .. (0,0) .. controls (3.31,0.3) and (6.95,1.4) .. (10.93,3.29)   ;
        \draw    (650.5,230.5) -- (663.99,246.02) ;
        \draw [shift={(665.3,247.53)}, rotate = 229.01] [color={rgb, 255:red, 0; green, 0; blue, 0 }  ][line width=0.75]    (10.93,-3.29) .. controls (6.95,-1.4) and (3.31,-0.3) .. (0,0) .. controls (3.31,0.3) and (6.95,1.4) .. (10.93,3.29)   ;
        
    \end{tikzpicture}
    \caption{Schematic depiction of the physical principle of a SH-WFS. Left and center: Cross cuts parallel to the optical axis for unperturbed and perturbed wavefronts. Right: Front view onto the sensor array and blurred/shifted images of the distant point source. See Section~\ref{subsect_SHWFS} for details. Image taken from \cite{Hubmer_Sherina_Ramlau_Pircher_Leitgeb_2023}.}
    \label{fig_Shack_Hartmann_WFS}
\end{figure}

The Shack-Hartmann wavefront sensor (SH-WFS) is one of the most commonly used WFSs in AO \cite{Ellerbroek_Vogel_2009,Platt_Shack_2001,Primot_2003}. It consists of a quadratic array of small lenses, also called lenslets. Through these, incoming light has to pass before reaching a CCD photon detector located in their focal plane. The detector itself is subdivided into a regular grid of so-called subapertures, and each lenslet focuses a part of the incoming light onto one of the subapertures; cf.~Figure~\ref{fig_Shack_Hartmann_WFS}. Now assume that the incoming light originates from a distant point source such as an astronomical or artificial star, or a distant laser, and that no aberrations are present. In this case, the incoming wavefront of the light source is plane, and thus each lenslet focuses the light onto a (diffraction-limited) point onto its corresponding subaperture on the CCD detector; cf.~Figure~\ref{fig_Shack_Hartmann_WFS}(left). On the other hand, if due to aberrations the wavefront is not plane, then these points are in fact shifted away from the center, and their shape is blurred; cf.~Figure~\ref{fig_Shack_Hartmann_WFS}(middle). The SH-WFS effectively measures the change in position of the centers of mass of these blurred points with respect to the centers of the subapertures; cf.~Figure~\ref{fig_Shack_Hartmann_WFS}(right). These relative shifts can mathematically be connected to the average gradient (slope) of the wavefront aberration on each subaperture, from which an approximation of the incoming wavefront can then be computed numerically \cite{Roddier_1999}. This is a non-trivial inverse problem, for which many different reconstruction approaches have been proposed; see e.g.\ \cite{Zhariy_Neubauer_Rosensteiner_Ramlau_2011,Rosensteiner_2011_01,Rosensteiner_2011_02}. These typically leverage physical priors such as smoothness or statistical assumptions on the wavefront aberrations, and are usually optimized in terms of numerical efficiency to satisfy stringent real-time requirements \cite{Stadler2022,Tallon_Bechet_TallonBosc_LeLouran_Thiebaut_Marchetti_2011}.

% % % % % % % % % % % % % % % % % % % % % % % % 
% Subsection - Fourier-type wavefront sensors %
% % % % % % % % % % % % % % % % % % % % % % % %
\subsection{Fourier-type wavefront sensors}\label{subsect_FWFS}

% Figure - Image Formation Model 
\begin{figure}[ht!]
    \centering
    \begin{tikzpicture}[x=0.75pt,y=0.75pt,yscale=-0.85,xscale=0.85]
    
        % Ellipses representing lenses
        \draw (176.25,150) .. controls (176.25,94.77) and (182.41,50) .. (190,50) .. controls (197.59,50) and (203.75,94.77) .. (203.75,150) .. controls (203.75,205.23) and (197.59,250) .. (190,250) .. controls (182.41,250) and (176.25,205.23) .. (176.25,150) -- cycle;
        \draw (476.25,150) .. controls (476.25,94.77) and (482.41,50) .. (490,50) .. controls (497.59,50) and (503.75,94.77) .. (503.75,150) .. controls (503.75,205.23) and (497.59,250) .. (490,250) .. controls (482.41,250) and (476.25,205.23) .. (476.25,150) -- cycle;
    
        % Straight lines representing Lenses and Fourier Plane
        \draw (340,90) -- (340,120);
        \draw (340,160) -- (340,190);
        \draw (190,50) -- (190,250);
        \draw (490,50) -- (490,250);
        \draw [ultra thick] (650,50) -- (650,250);
        
        % Curved lines representing light waves
        \draw (40,130) .. controls (60.8,100.3) and (19.8,79.3) .. (40,50);
        \draw (40,250) .. controls (60.8,220.3) and (19.8,199.3) .. (40,170);
        \draw (640,130) .. controls (660.8,100.3) and (619.8,79.3) .. (640,50);
        \draw (640,250) .. controls (660.8,220.3) and (619.8,199.3) .. (640,170);

        % Arrows representing focal lengths
        \draw (100,270) -- (42,270);
        \draw [shift={(40,270)}, rotate = 360][color={rgb, 255:red, 0; green, 0; blue, 0}][line width=0.75] (10.93,-3.29) .. controls (6.95,-1.4) and (3.31,-0.3) .. (0,0) .. controls (3.31,0.3) and (6.95,1.4) .. (10.93,3.29);
        \draw (250,270) -- (192,270);
        \draw [shift={(190,270)}, rotate = 360][color={rgb, 255:red, 0; green, 0; blue, 0}][line width=0.75] (10.93,-3.29) .. controls (6.95,-1.4) and (3.31,-0.3) .. (0,0) .. controls (3.31,0.3) and (6.95,1.4) .. (10.93,3.29);
        \draw (400,270) -- (342,270);
        \draw [shift={(340,270)}, rotate = 360][color={rgb, 255:red, 0; green, 0; blue, 0}][line width=0.75] (10.93,-3.29) .. controls (6.95,-1.4) and (3.31,-0.3) .. (0,0) .. controls (3.31,0.3) and (6.95,1.4) .. (10.93,3.29);
        \draw (550,270) -- (492,270);
        \draw [shift={(490,270)}, rotate = 360][color={rgb, 255:red, 0; green, 0; blue, 0}][line width=0.75] (10.93,-3.29) .. controls (6.95,-1.4) and (3.31,-0.3) .. (0,0) .. controls (3.31,0.3) and (6.95,1.4) .. (10.93,3.29);
        \draw (130,270) -- (188,270);
        \draw [shift={(190,270)}, rotate = 180] [color={rgb, 255:red, 0; green, 0; blue, 0}][line width=0.75] (10.93,-3.29) .. controls (6.95,-1.4) and (3.31,-0.3) .. (0,0) .. controls (3.31,0.3) and (6.95,1.4) .. (10.93,3.29);
        \draw (280,270) -- (338,270);
        \draw [shift={(340,270)}, rotate = 180] [color={rgb, 255:red, 0; green, 0; blue, 0}][line width=0.75] (10.93,-3.29) .. controls (6.95,-1.4) and (3.31,-0.3) .. (0,0) .. controls (3.31,0.3) and (6.95,1.4) .. (10.93,3.29) ;
        \draw (430,270) -- (488,270);
        \draw [shift={(490,270.3)}, rotate = 180.29] [color={rgb, 255:red, 0; green, 0; blue, 0}][line width=0.75] (10.93,-3.29) .. controls (6.95,-1.4) and (3.31,-0.3) .. (0,0) .. controls (3.31,0.3) and (6.95,1.4) .. (10.93,3.29);
        \draw (580,270) -- (638,270);
        \draw [shift={(640,270)}, rotate = 180] [color={rgb, 255:red, 0; green, 0; blue, 0}][line width=0.75] (10.93,-3.29) .. controls (6.95,-1.4) and (3.31,-0.3) .. (0,0) .. controls (3.31,0.3) and (6.95,1.4) .. (10.93,3.29);

        % Text description of focal length
        \draw (110,258.4) node [anchor=north west][inner sep=0.75pt]{$f$};
        \draw (260,258.4) node [anchor=north west][inner sep=0.75pt]{$f$};
        \draw (410,258.4) node [anchor=north west][inner sep=0.75pt]{$f$};
        \draw (560,258.4) node [anchor=north west][inner sep=0.75pt]{$f$};
    
        % Text description of Image and Fourier planes   
        \draw (172,23) node [anchor=north west][inner sep=0.75pt][align=left]{Lens};
        \draw (472,23) node [anchor=north west][inner sep=0.75pt][align=left]{Lens};
        \draw (290,23) node [anchor=north west][inner sep=0.75pt][align=left]{Fourier plane};
        \draw (2,23) node [anchor=north west][inner sep=0.75pt][align=left]{Image plane};
        \draw (562,23) node [anchor=north west][inner sep=0.75pt][align=left]{Conjugate plane};
        \draw (280,63.4) node [anchor=north west][inner sep=0.75pt]{Optical filtering};
        \draw (323,125) node [anchor=north west][inner sep=0.75pt]{\Huge $\Delta$};
        \draw (305,220) node [anchor=north west][inner sep=0.75pt]{$\triangleq$ OTF};
        \draw (275,195) node [anchor=north west][inner sep=0.75pt][align=left]{Optical element};
        \draw (2,138.4) node [anchor=north west][inner sep=0.75pt]{Incident field};
        \draw (540,140) node [anchor=north west][inner sep=0.75pt]{Filtered field};
        \draw (680,80) node [anchor=north west][inner sep=0.75pt][rotate=270]{Camera/Detector};

    \end{tikzpicture}
    \caption{Schematic depiction of a Fourier-type WFS measurement system; cf.~\cite{HuNeuSha_2023}.}
    \label{fig_fwfs}
\end{figure}
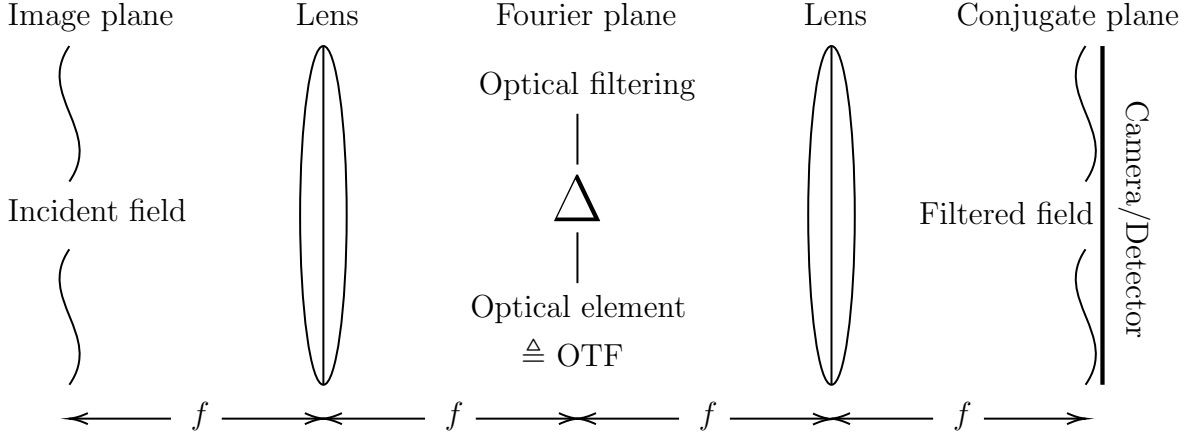

Next, we consider the large class of nonlinear Fourier-type WFSs~\cite{Fauv16,Potiron2019}, which are included in the instrument design of many current and upcoming AO systems. Fourier-type wavefront sensing is based on optical Fourier filtering with a suitable optical element, such as a multi-facet glass prism, located in the focal plane; cf.~Figure~\ref{fig_fwfs}. The optical element splits the electromagnetic field into several (potentially overlapping) beams, each of which representing a differently filtered image of the entrance pupil. The number of beams after the optical element corresponds to the number of its facets dividing the light. The intensity patterns are then measured by a camera in the conjugate pupil plane. Especially in astronomical AO systems, an additional modulation component is often added to the Fourier-type WFS setup due to nonlinearity issues. This means that the electromagnetic field is moved around the apex of the optical element in order to improve the light distribution to all faces of the optical element \cite{ShaHuRa20}, influencing system complexity, cost-effectiveness, and sensitivity~\cite{RaFa99,Veri04}. 

\begin{figure}[ht!]
    \centering
    \includegraphics[width=\textwidth]{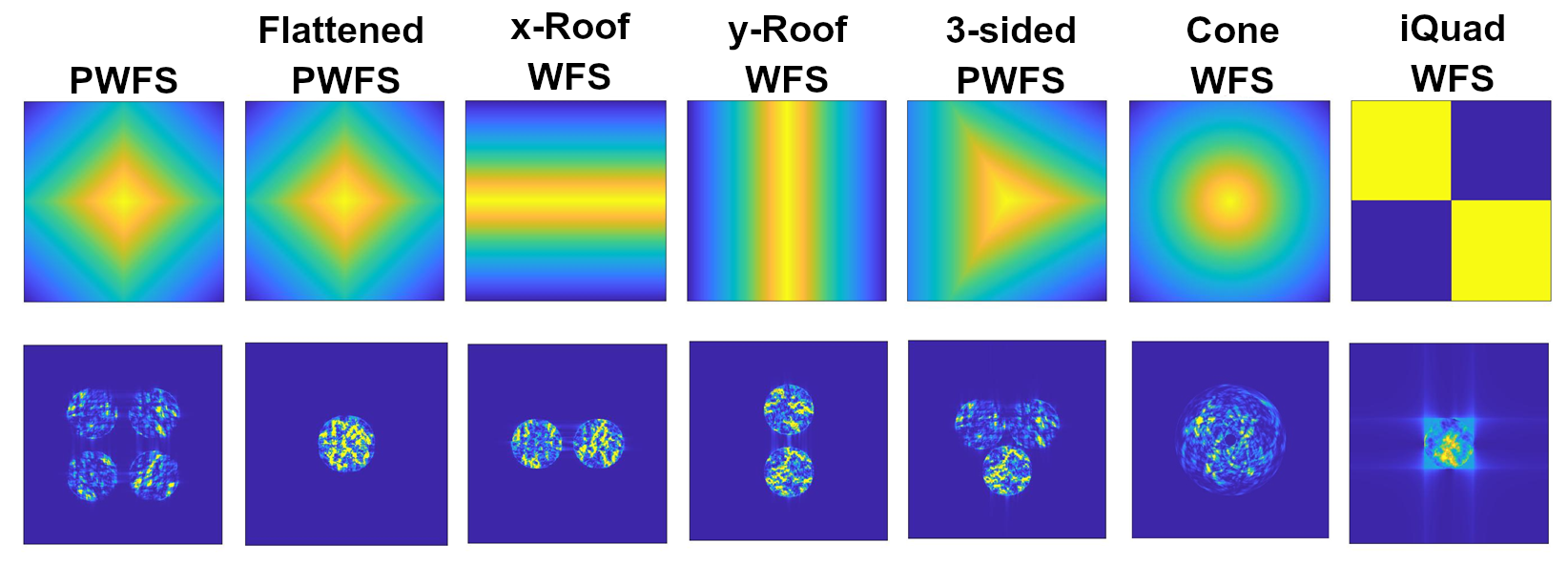}
    \caption{Optical transfer function (top) and corresponding intensity image (bottom) of several Fourier-type WFSs with the phase $\phi$ from Figure~\ref{fig_example_wrapping} as the incoming wavefront aberration. The first two columns refer to PWFSs with different apex angles.}
  \label{fig_otfs}
\end{figure}

The key component of a Fourier-type WFS is the optical element in its focal/Fourier plane. Different realisations of Fourier-type WFSs correspond to different optical elements which split the light in the focal plane, e.g.\ an $n$-sided pyramidal prism. The most common example in AO instruments is the $4$-sided glass pyramidal prism, typically referred to as the pyramid WFS (PWFS)~\cite{Raga96,RaDi02}. The prism splits the electromagnetic field into $4$ beams that are imaged in the conjugate pupil plane. Moreover, $2$-sided (roof WFS), $3$-sided, and $6$-sided prisms, as well as the extension of the idea to an infinite number of faces (cone WFS) have been investigated in more detail for astronomical AO~\cite{Clare_ao4elt5,ClareCone2020,engler2017,Veri04}. 
A Fourier-type WFSs for which the optical element does not have the shape of a prism is the Zernike WFS. It has a small circular hole in the center of the mask of the optical element~\cite{Zernike1934}. The iQuad WFS is another Fourier-type WFS. Its focal plane is divided into $4$-quadrants around the origin, i.e., it has a Cartesian structure where each quadrant is phase shifted with its two neighbours~\cite{FaHuShaRaLAM19_AO4ELTproc}. As we will discuss in Section~\ref{subsect_digFWFS}, the different optical elements can mathematically be characterized via their optical transfer functions (OTFs) \cite{Fauv16}, examples of which are given in Table~\ref{table_psi} below. Visualizations of the OTFs corresponding to the 3-sided WFS, cone WFS, iQuad WFS, and different PWFSs and roof WFSs are depicted in Figure~\ref{fig_otfs}. Note that the apex angle of the prism influences the separation distance between the intensity patterns in the conjugate pupil plane (see Figure~\ref{fig_otfs}, first three columns from the left).

% % % % % % % % % % % % % % % % % % % % % % % %
% Section - Digital WFSs for phase unwrapping %
% % % % % % % % % % % % % % % % % % % % % % % %
\section{Digital wavefront sensors for phase unwrapping}\label{sect_digWFS}

In this section, we present the general concept of phase unwrapping using digital WFSs, and introduce two classes of unwrapping algorithms based on (digital) Shack-Hartmann and Fourier-type WFSs. For this, we first have to recall some results from Fourier optics.

% % % % % % % % % % % % % % % % % % % % % % % % % % % %
% Subsection - Fourier transforms and Fourier optics  %
% % % % % % % % % % % % % % % % % % % % % % % % % % % %
\subsection{Fourier transforms and Fourier optics}\label{subsect_Fourier}

In its most general form, the propagation of an electric field (light) through an optical system is described by Maxwell's equations. In many situations, this propagation can be approximated to a high degree of accuracy using the tools of Fourier optics, a detailed overview of which can e.g.\ be found in the seminal work \cite{Goodman_2005}. Here, we only require the Fourier optics description of the optical setup depicted in Figure~\ref{fig_image_formation_model}, which consists of two properly aligned lenses and serves as a generalized model of an imaging system.

% Figure - Image Formation Model 
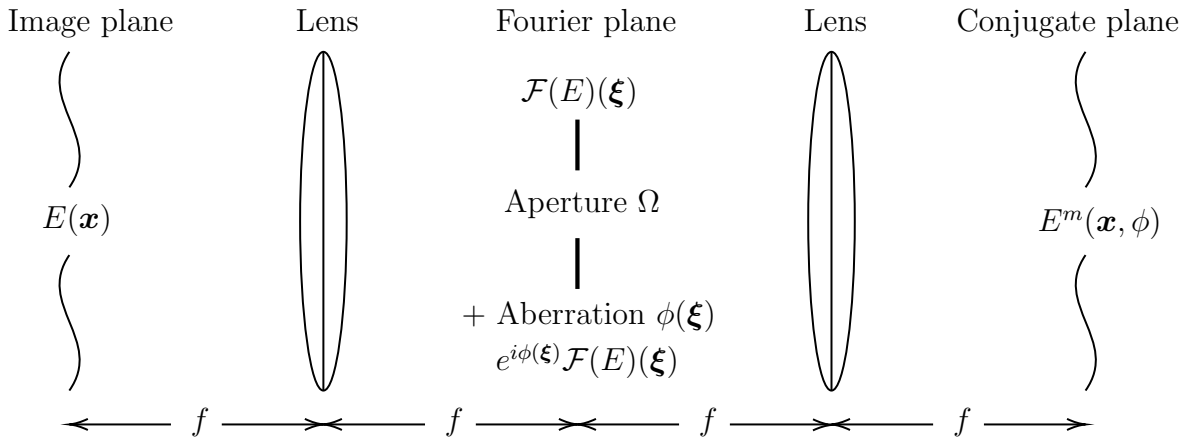
\begin{figure}[ht!]
    \centering
    \begin{tikzpicture}[x=0.75pt,y=0.75pt,yscale=-0.85,xscale=0.85]
    
        % Ellipses representing lenses
        \draw (176.25,150) .. controls (176.25,94.77) and (182.41,50) .. (190,50) .. controls (197.59,50) and (203.75,94.77) .. (203.75,150) .. controls (203.75,205.23) and (197.59,250) .. (190,250) .. controls (182.41,250) and (176.25,205.23) .. (176.25,150) -- cycle;
        \draw (476.25,150) .. controls (476.25,94.77) and (482.41,50) .. (490,50) .. controls (497.59,50) and (503.75,94.77) .. (503.75,150) .. controls (503.75,205.23) and (497.59,250) .. (490,250) .. controls (482.41,250) and (476.25,205.23) .. (476.25,150) -- cycle;
    
        % Straight lines representing Lenses and Fourier Plane
        \draw [ultra thick](340,90) -- (340,120);
        \draw [ultra thick](340,160) -- (340,190);
        %\draw (340,90) -- (340,190);
        \draw (190,50) -- (190,250);
        \draw (490,50) -- (490,250);
    
        % Curved lines representing light waves
        \draw (40,130) .. controls (60.8,100.3) and (19.8,79.3) .. (40,50);
        \draw (40,250) .. controls (60.8,220.3) and (19.8,199.3) .. (40,170);
        \draw (640,130) .. controls (660.8,100.3) and (619.8,79.3) .. (640,50);
        \draw (640,250) .. controls (660.8,220.3) and (619.8,199.3) .. (640,170);

        % Arrows representing focal lengths
        \draw (100,270) -- (42,270);
        \draw [shift={(40,270)}, rotate = 360][color={rgb, 255:red, 0; green, 0; blue, 0}][line width=0.75] (10.93,-3.29) .. controls (6.95,-1.4) and (3.31,-0.3) .. (0,0) .. controls (3.31,0.3) and (6.95,1.4) .. (10.93,3.29);
        \draw (250,270) -- (192,270);
        \draw [shift={(190,270)}, rotate = 360][color={rgb, 255:red, 0; green, 0; blue, 0}][line width=0.75] (10.93,-3.29) .. controls (6.95,-1.4) and (3.31,-0.3) .. (0,0) .. controls (3.31,0.3) and (6.95,1.4) .. (10.93,3.29);
        \draw (400,270) -- (342,270);
        \draw [shift={(340,270)}, rotate = 360][color={rgb, 255:red, 0; green, 0; blue, 0}][line width=0.75] (10.93,-3.29) .. controls (6.95,-1.4) and (3.31,-0.3) .. (0,0) .. controls (3.31,0.3) and (6.95,1.4) .. (10.93,3.29);
        \draw (550,270) -- (492,270);
        \draw [shift={(490,270)}, rotate = 360][color={rgb, 255:red, 0; green, 0; blue, 0}][line width=0.75] (10.93,-3.29) .. controls (6.95,-1.4) and (3.31,-0.3) .. (0,0) .. controls (3.31,0.3) and (6.95,1.4) .. (10.93,3.29);
        \draw (130,270) -- (188,270);
        \draw [shift={(190,270)}, rotate = 180] [color={rgb, 255:red, 0; green, 0; blue, 0}][line width=0.75] (10.93,-3.29) .. controls (6.95,-1.4) and (3.31,-0.3) .. (0,0) .. controls (3.31,0.3) and (6.95,1.4) .. (10.93,3.29);
        \draw (280,270) -- (338,270);
        \draw [shift={(340,270)}, rotate = 180] [color={rgb, 255:red, 0; green, 0; blue, 0}][line width=0.75] (10.93,-3.29) .. controls (6.95,-1.4) and (3.31,-0.3) .. (0,0) .. controls (3.31,0.3) and (6.95,1.4) .. (10.93,3.29) ;
        \draw (430,270) -- (488,270);
        \draw [shift={(490,270.3)}, rotate = 180.29] [color={rgb, 255:red, 0; green, 0; blue, 0}][line width=0.75] (10.93,-3.29) .. controls (6.95,-1.4) and (3.31,-0.3) .. (0,0) .. controls (3.31,0.3) and (6.95,1.4) .. (10.93,3.29);
        \draw (580,270) -- (638,270);
        \draw [shift={(640,270)}, rotate = 180] [color={rgb, 255:red, 0; green, 0; blue, 0}][line width=0.75] (10.93,-3.29) .. controls (6.95,-1.4) and (3.31,-0.3) .. (0,0) .. controls (3.31,0.3) and (6.95,1.4) .. (10.93,3.29);

        % Text description of focal length
        \draw (110,258.4) node [anchor=north west][inner sep=0.75pt]{$f$};
        \draw (260,258.4) node [anchor=north west][inner sep=0.75pt]{$f$};
        \draw (410,258.4) node [anchor=north west][inner sep=0.75pt]{$f$};
        \draw (560,258.4) node [anchor=north west][inner sep=0.75pt]{$f$};
    
        % Text description of Image and Fourier planes   
        \draw (172,23) node [anchor=north west][inner sep=0.75pt][align=left]{Lens};
        \draw (472,23) node [anchor=north west][inner sep=0.75pt][align=left]{Lens};
        \draw (290,23) node [anchor=north west][inner sep=0.75pt][align=left]{Fourier plane};
        \draw (2,23) node [anchor=north west][inner sep=0.75pt][align=left]{Image plane};
        \draw (562,23) node [anchor=north west][inner sep=0.75pt][align=left]{Conjugate plane};
        \draw (305,63.4) node [anchor=north west][inner sep=0.75pt]{$\F( E)(\xiv)$};
        \draw (290,220) node [anchor=north west][inner sep=0.75pt]{$e^{i \phi(\xiv)}\F(E)(\xiv)$};
        \draw (295,130) node [anchor=north west][inner sep=0.75pt]{Aperture $\Omega$};
        \draw (270,195) node [anchor=north west][inner sep=0.75pt][align=left]{+ Aberration $\phi(\xiv)$};
        \draw (22,138.4) node [anchor=north west][inner sep=0.75pt]{$E(\xv)$};
        \draw (610,140) node [anchor=north west][inner sep=0.75pt]{$E^{m}(\xv,\phi)$};

    \end{tikzpicture}
    \caption{Schematic depiction of two-lens imaging system corresponding to \eqref{model_aperture}.}
    \label{fig_image_formation_model}
\end{figure}

First, we recall the Fourier transform $\F$ and its inverse $\FI$, defined by
    \begin{equation*}
        \F(f)(\xiv) := \int_{\Rt} f(\xv) e^{-2\pi i \xv \cdot \xiv } \, d \xv \,,
        \qquad
        \text{and}
        \qquad 
        \FI( g )(\xv) := \int_{\Rt} g(\xiv) e^{2\pi i \xiv \cdot \xv } \, d \xiv \,,
    \end{equation*}
for all (square) integrable functions $f : \Rt \to \C$. Next, assume that the complex valued electric field $E : \Rt \to \C$, initially located at the focal length $f$ from the first lens, enters the optical system from the left as depicted in Figure~\ref{fig_image_formation_model}. Furthermore, assume that in the Fourier plane the resulting field is modulated by the multiplicative factor $e^{i \phi}$, where the function $\phi : \Rt \to \R$ represents either some wavefront aberration or any other type of phase modulation (e.g., induced by an optical element or a spatial light modulator). According to Fourier optics, the electric field $\Em : \Rt \to \C$ exiting the optical system at the focal length $f$ after the second lens can be computed by
    \begin{equation}\label{model}
        \Em(\xv,\phi) = \FI\kl{ e^{i \phi(\xiv)} \F(E)(\xiv) }(\xv) \,, 
        \qquad
        \forall \, \xv \in \R^2 \,.
    \end{equation}
If an aperture is present in the Fourier plane, \eqref{model} becomes
    \begin{equation}\label{model_aperture}
        \Em(\xv,\phi) = \FI\kl{ \chi_\Omega(\xiv) e^{i \phi(\xiv)} \F(E)(\xiv) }(\xv) \,, 
        \qquad
        \forall \, \xv \in \R^2 \,.
    \end{equation}
where $\chi_\Omega$ is the characteristic function of the domain $\Omega \subset \Rt$ representing the shape of the aperture. The intensity corresponding to the electric field $\Em$ is defined by
    \begin{equation*}
        \Im(\xv,\phi) := \abs{ \Em(\xv,\phi) }^2 \,.
    \end{equation*}
In the idealized case, where the incoming light $E$ stems from an infinitely distant point source, i.e., $E \equiv \delta_0$ or $\F(E) \equiv 1$, we thus obtain that
    \begin{equation}\label{intensity_point_source}
        \Im(\xv,\phi) = \abs{ \Em(\xv,\phi) }^2 = \abs{ \FI\kl{ \chi_\Omega(\xiv) e^{i \phi(\xiv)} }(\xv) }^2 \,, 
        \qquad
        \forall \, \xv \in \R^2 \,.
    \end{equation}
Hence, in case of no aberrations, i.e., $\phi \equiv 0$, the intensity $\Im = \abs{ \FI\kl{ \chi_\Omega } }^2$ of the point source is exactly the diffraction limited point-spread function (PSF) of the system.

% % % % % % % % % % % % % % % % % % % % % % % % % % % % % %
% Subsection - Key ideas and generic unwrapping algorithm %
% % % % % % % % % % % % % % % % % % % % % % % % % % % % % %
\subsection{Key ideas and generic unwrapping algorithm}\label{subsect_keyidea}

We now return to the phase unwrapping problem \eqref{prob_unwrap} and derive a generic unwrapping algorithm based on the concept of digital WFSs. The key observation here is that
    \begin{equation*}
        e^{i \phiw(\xiv)} \overset{\eqref{prob_unwrap}}{=} e^{i \kl{\phi(\xiv) \mod 2 \pi}} = e^{i\phi(\xiv)} \,,
        \qquad
        \forall \, \xiv \in \Omega \subseteq \R^2 \,,
    \end{equation*}
and thus, by its definition \eqref{model_aperture}, there follows the equivalence of the electric fields
    \begin{equation*}
        \Em(\xv,\phi) = \Em(\xv,\phiw) \,, 
        \qquad
        \forall \, \xv \in \R^2 \,.
    \end{equation*}
This means that when the phase $\phi$ and its wrapped form $\phiw$ are considered as wavefront aberrations in the optical system corresponding to Figure~\ref{fig_image_formation_model}, their effect on the outgoing electric field $\Em$ is equivalent. Consequently, any WFS using the electric field $\Em$ as an input for measuring the wavefront aberration is unable to distinguish between the non-wrapped and wrapped phases $\phi$ and $\phiw$, respectively. Crucially, though, most wavefront reconstructors are optimized for reconstructing aberrations with certain smoothness properties, in particular those designed for the SH and Fourier-type WFSs in astronomical applications. Hence, the result of applying such a reconstructor to $\Em(\xv,\phiw)$ essentially yields a smooth and unwrapped phase, denoted by $\phir$. 

From here, only two further observations are necessary. The first one is that WFSs are typically based on the assumption that the incoming light stems from a sufficiently distant point source, which corresponds to $E \equiv \delta_0$ or $\F(E) \equiv 1$. This is necessary, since otherwise in \eqref{model_aperture} the (non-zero) phase of $\F(E)$ mixes with the aberration $\phi$, making a subsequent separation of the two difficult \cite{Hubmer_Sherina_Ramlau_Pircher_Leitgeb_2023}. Secondly, an electric field $\Em$ aberrated by a given wrapped phase $\phiw$ can be obtained \emph{completely digitally}. For example, in the case of a distant point source one can simply define, cf.\ \eqref{intensity_point_source},
    \begin{equation}\label{def_Emp}
        \Emp(\xv,\phiw) :=  \FI\kl{ \chi_\Omega e^{i \phiw} }(\xv) \,, 
        \qquad
        \forall \, \xv \in \Rt \,.
    \end{equation}
In order to digitally obtain WFS measurements, one can then simply apply the WFS forward model on which the subsequently employed reconstruction algorithm is based.  

Combining these key ideas and observations, we can thus formulate the following generic phase unwrapping algorithm, which is the basis for our further considerations.

\begin{algo}[Generic digital WFS phase unwrapping algorithm]\label{algo_generic}
Given the wrapped phase $\phiw(\xiv)$ for all $\xiv \in \Omega \subset \Rt$, compute the unwrapped phase $\phir(\xiv)$ via the steps:
\begin{enumerate}
    \item Compute $\Emp(x,\phiw)$ as in \eqref{def_Emp}, i.e., $\Emp(\xv,\phiw) =  \FI\kl{ \chi_\Omega e^{i \phiw} }(\xv)$.
    \item Choose a WFS and create corresponding sensor measurements using $\Emp(x,\phiw)$.
    \item Choose a WFS reconstructor and run it on the measurements to obtain $\phir$.
\end{enumerate}
\end{algo}

As the name suggests, Algorithm~\ref{algo_generic} is more of a general concept than a directly implementable algorithm. In particular, the concrete realization of Steps 2 and 3 will strongly impact the overall complexity and performance of the resulting method. Two possible choices, based on both SH- and Fourier-type WFSs, are considered below.

\begin{remark}
The phase unwrapping problem can be considered both in a continuous and a discrete setting. Our generic Algorithm~\ref{algo_generic} and its specific instances discussed below are operating in the continuous setting corresponding of problem \eqref{prob_unwrap}. Besides the greater generality of the continuous setting, a main reason for this is that the underlying physical models of the considered (digital) wavefront sensors operate in the continuous setting as well. When our proposed algorithms are applied in the discrete setting of numerical computations, the involved operations need to be replaced by their discrete counterparts. This essentially concerns the continuous Fourier transform, which has to be replaced by the (fast) discrete Fourier transform, as is done in the numerical experiments presented in Section~\ref{sect_num_exp}. Note that this situation is the same for the global comparison algorithms based on the PE and the TIE, which were introduced above and are discussed in detail in Section~\ref{subsect_impl}. These methods are formulated in terms of (continuous) partial differential equations, which are discretized in numerical applications.  
\end{remark}

% % % % % % % % % % % % % % % % % % % % % % % % % % % % % % %
% Subsection - The digital Shack-Hartmann wavefront sensor  %
% % % % % % % % % % % % % % % % % % % % % % % % % % % % % % %
\subsection{The digital Shack-Hartmann wavefront sensor}\label{subsect_digSHWFS}

In this section, we consider the specific choice of the SH-WFS in Step~2 of Algorithm~\ref{algo_generic}. In particular, we discuss how corresponding WFS measurements can be calculated completely digitally, following the mathematical framework derived in \cite{Hubmer_Sherina_Ramlau_Pircher_Leitgeb_2023}.

As discussed in Section~\ref{subsect_SHWFS}, the SH-WFS works via a lenslet array, each lenslet of which focuses the incoming light of a distant point source onto a corresponding subaperture of detector. Mathematically, this can be modelled by splitting the aperture domain $\Omega$ in \eqref{model_aperture} into a number of smaller subdomains $\Ojk \subset \Omega$, defined by
    \begin{equation}\label{def_Ojk}
        \Ojk := 
        \Kl{\xiv \in \R^2 \, \vert \, \Tjk^{-1}(\xiv) \in [ -J/2, J/2 ) \times [-K/2, K/2) } \,, 
    \end{equation}
for indices $(j,k) \in \Lambda \subset \Z^2$, where $J, K \in \R^+$ and the translation operator $\Tjk$ is defined by
    \begin{equation*}
        \Tjk(\xiv) := (\xi_1 + jJ, \xi_2 + k K) \,,
        \qquad
        \forall \, \xiv \in \R^2 \,.
    \end{equation*}
Note that we implicitly assume that the domain $\Omega$ can be split evenly into the subdomains $\Ojk$ for a certain, potentially infinite, set of indices $\Lambda$. Next, let $\chi_{j,k}$ denote the characteristic functions of $\Ojk$ and define
    \begin{equation}\label{def_Emjk}
		\Emjk(\xv,\phi) := \FI\kl{ \chi_{0,0}(\xiv) \F(\Em(\cdot,\phi))(\Tjk(\xiv))}(\xv)\,,
		\qquad
		\forall \xv \in \R^2 \,.
	\end{equation}
This definition of $\Emjk$ corresponds to the Fourier optics description of the electric fields arriving at each of the subapertures of the SH-WFS, and is illustrated in Figure~\ref{fig_subimage_definition}.

% % % % % % % % % % % % % % % % %
% Figure - Subimage Definition  %
% % % % % % % % % % % % % % % % %
\begin{figure}
    \centering
    \begin{tikzpicture}[x=0.75pt,y=0.75pt,yscale=-0.85,xscale=0.85]

        % % % % % %
        % Image 1 %
        % % % % % %
    
        % Axis - Coordinate System
        \draw [line width=1.5] (10,190) -- (130,190)(51,90) -- (51,220) (123,185) -- (130,190) -- (123,195) (46,97) -- (51,90) -- (56,97);
    
        % Image Background Pattern
        \draw [draw opacity=0][pattern=_7gn8i7dj0,pattern size=15pt,pattern thickness=0.75pt,pattern radius=0pt, pattern color={rgb, 255:red, 0; green, 0; blue, 0}] (20,100) -- (120,100) -- (120,210) -- (20,210) -- cycle;
    
        % Background for Image Description
        \draw [color={rgb, 255:red, 255; green, 255; blue, 255}, draw opacity=1][fill={rgb, 255:red, 255; green, 255; blue, 255},fill opacity=1] (100,100) -- (130,100) -- (130,130) -- (100,130) -- cycle;

        % Text Nodes - Image Descriptions
        \draw (107,101.4) node [anchor=north west][inner sep=0.75pt]{$\Em$};
        \draw (22,223) node [anchor=north west][inner sep=0.75pt][align=left]{Image plane};
    
        % % % % % % %
        % I1 to I2  %
        % % % % % % %
    
        \draw (152,133.4) node [anchor=north west][inner sep=0.75pt]{$\F$};
        \draw (140,160) -- (178,160);
        \draw [shift={(180,160)}, rotate = 180][color={rgb, 255:red, 0; green, 0; blue, 0}][line width=0.75] (10.93,-3.29) .. controls (6.95,-1.4) and (3.31,-0.3) .. (0,0) .. controls (3.31,0.3) and (6.95,1.4) .. (10.93,3.29);

        % % % % % %
        % Image 2 %
        % % % % % %
    
        % Axis - Coordinate System 
        \draw [line width=1.5] (190,190) -- (310,190)(231,90) -- (231,220) (303,185) -- (310,190) -- (303,195) (226,97) -- (231,90) -- (236,97);
    
        % Image Background Pattern
        \draw [draw opacity=0][pattern=_7gn8i7dj0,pattern size=15pt,pattern thickness=0.75pt,pattern radius=0pt, pattern color={rgb, 255:red, 0; green, 0; blue, 0}][dash pattern={on 4.5pt off 4.5pt}] (200,100) -- (300,100) -- (300,210) -- (200,210) -- cycle;
    
        % Subdomain Rectangles
        \draw [color={rgb, 255:red, 0; green, 0; blue, 0}, draw opacity=1][dash pattern={on 0.84pt off 2.51pt}] (276,115) -- (306,115) -- (306,145) -- (276,145) -- cycle;
        \draw [color={rgb, 255:red, 0; green, 0; blue, 0}, draw opacity=1][dash pattern={on 0.84pt off 2.51pt}] (246,115) -- (276,115) -- (276,145) -- (246,145) -- cycle;
        \draw [color={rgb, 255:red, 0; green, 0; blue, 0}, draw opacity=1][dash pattern={on 0.84pt off 2.51pt}] (216,115) -- (246,115) -- (246,145) -- (216,145) -- cycle;
        \draw [color={rgb, 255:red, 0; green, 0; blue, 0}, draw opacity=1][dash pattern={on 0.84pt off 2.51pt}] (216,145) -- (246,145) -- (246,175) -- (216,175) -- cycle;
        \draw [color={rgb, 255:red, 0; green, 0; blue, 0}, draw opacity=1][dash pattern={on 0.84pt off 2.51pt}] (246,145) -- (276,145) -- (276,175) -- (246,175) -- cycle;
        \draw [color={rgb, 255:red, 0; green, 0; blue, 0}, draw opacity=1][dash pattern={on 0.84pt off 2.51pt}] (276,145) -- (306,145) -- (306,175) -- (276,175) -- cycle;
        \draw [color={rgb, 255:red, 0; green, 0; blue, 0}, draw opacity=1][dash pattern={on 0.84pt off 2.51pt}] (216,175) -- (246,175) -- (246,205) -- (216,205) -- cycle;
        \draw [color={rgb, 255:red, 0; green, 0; blue, 0}, draw opacity=1][dash pattern={on 0.84pt off 2.51pt}] (246,175) -- (276,175) -- (276,205) -- (246,205) -- cycle;
        \draw [color={rgb, 255:red, 0; green, 0; blue, 0}, draw opacity=1][dash pattern={on 0.84pt off 2.51pt}] (276,175) -- (306,175) -- (306,205) -- (276,205) -- cycle;

        % Shift arrow
        \draw (450,150) -- (431.41,168.59);
        \draw [shift={(430,170)}, rotate = 315][color={rgb, 255:red, 0; green, 0; blue, 0}][line width=0.75] (10.93,-3.29) .. controls (6.95,-1.4) and (3.31,-0.3) .. (0,0) .. controls (3.31,0.3) and (6.95,1.4) .. (10.93,3.29);
    
        % Background for Image Description 
        \draw [color={rgb, 255:red, 255; green, 255; blue, 255}, draw opacity=1][fill={rgb, 255:red, 255; green, 255; blue, 255}, fill opacity=1] (270,100) -- (300,100) -- (300,130) -- (270,130) -- cycle;

        % Text Nodes - Image Descriptions
        \draw (271,102.4) node [anchor=north west][inner sep=0.75pt]{$\F(\Em)$};
        \draw (202,223) node [anchor=north west][inner sep=0.75pt][align=left]{Fourier plane};
    
        % % % % % % % 
        % I2 to I3  %
        % % % % % % %
    
        \draw (330,133) node [anchor=north west][inner sep=0.75pt][align=left]{cutout};
        \draw (320,130) .. controls (359.6,100.3) and (399.2,158.81) .. (438.8,130.87);
        \draw [shift={(440,130)}, rotate = 503.13] [color={rgb, 255:red, 0; green, 0; blue, 0}][line width=0.75] (10.93,-3.29) .. controls (6.95,-1.4) and (3.31,-0.3) .. (0,0) .. controls (3.31,0.3) and (6.95,1.4) .. (10.93,3.29);

        % % % % % %
        % Image 3 %
        % % % % % %
    
        % Axis - Coordinate System 
        \draw [line width=1.5] (370,190) -- (490,190)(411,90) -- (411,220) (483,185) -- (490,190) -- (483,195) (406,97) -- (411,90) -- (416,97);
    
        % Text Nodes - Image Descriptions
        \draw (425,88.4) node [anchor=north west][inner sep=0.75pt][font=\small]{$\chi_{j,k} \F(\Em)$};
        \draw (452,153) node [anchor=north west][inner sep=0.75pt][align=left]{shift};
        \draw (382,223) node [anchor=north west][inner sep=0.75pt][align=left]{Fourier plane};
    
        % Shifted Rectangle
        \draw [color={rgb, 255:red, 0; green, 0; blue, 0}, draw opacity=1][pattern=_7gn8i7dj0,pattern size=15pt,pattern thickness=0.75pt,pattern radius=0pt, pattern color={rgb, 255:red, 0; green, 0; blue, 0}][dash pattern={on 4.5pt off 4.5pt}] (396,175) -- (426,175) -- (426,205) -- (396,205) -- cycle;
    
        % Original Rectangle 
        \draw [color={rgb, 255:red, 0; green, 0; blue, 0}, draw opacity=1][pattern=_7gn8i7dj0,pattern size=15pt,pattern thickness=0.75pt,pattern radius=0pt, pattern color={rgb, 255:red, 0; green, 0; blue, 0}] (456,115) -- (486,115) -- (486,145) -- (456,145) -- cycle;

        % % % % % % %
        % I3 to I4  %
        % % % % % % %
    
        \draw (515,133.4) node [anchor=north west][inner sep=0.75pt]{$\F$};
        \draw (510,160) -- (548,160);
        \draw [shift={(550,160)}, rotate = 180] [color={rgb, 255:red, 0; green, 0; blue, 0}][line width=0.75] (10.93,-3.29) .. controls (6.95,-1.4) and (3.31,-0.3) .. (0,0) .. controls (3.31,0.3) and (6.95,1.4) .. (10.93,3.29);
    
        % % % % % %
        % Image 4 %
        % % % % % %
    
        % Axis - Coordinate System
        \draw [line width=1.5] (560,190) -- (680,190)(601,90) -- (601,220) (673,185) -- (680,190) -- (673,195) (596,97) -- (601,90) -- (606,97);

        % Image Background Pattern
        \draw [draw opacity=0][pattern=_7gn8i7dj0,pattern size=15pt,pattern thickness=0.75pt,pattern radius=0pt, pattern color={rgb, 255:red, 0; green, 0; blue, 0}] (573,100) -- (673,100) -- (673,210) -- (573,210) -- cycle;

        % Background for Image Description 
        \draw [color={rgb, 255:red, 255; green, 255; blue, 255},draw opacity=1][fill={rgb, 255:red, 255; green, 255; blue, 255 },fill opacity=1] (643,100) -- (673,100) -- (673,130) -- (643,130) -- cycle;

        % Text Nodes - Image Descriptions
        \draw (647,101.4) node [anchor=north west][inner sep=0.75pt]{$\Emjk$};
        \draw (572,223) node [anchor=north west][inner sep=0.75pt][align=left] {Image plane};

    \end{tikzpicture}
    \caption{Illustration of the definition of the electric fields $\Emjk$ given in \eqref{def_Emjk}.}
    \label{fig_subimage_definition}
\end{figure}
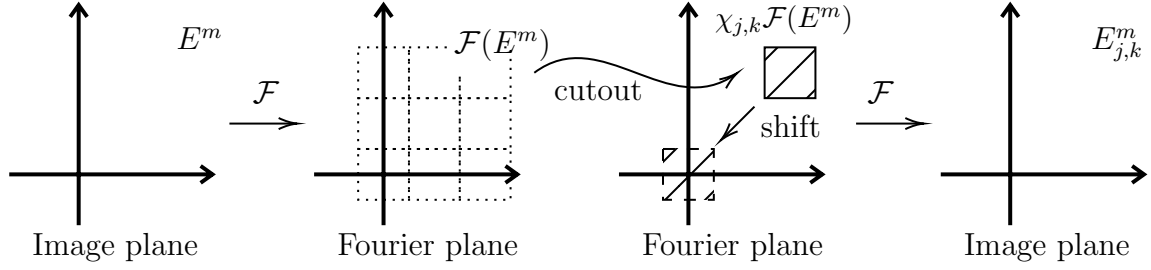

Next, note that from the defining model equation \eqref{model_aperture} for $\Em$ it follows that if $E$ is a point source, i.e., if it satisfies $\F(E) \equiv 1$, then there holds
    \begin{equation*}
    \begin{split}
        \Emjk(\xv,\phi) & = 
        \FI\kl{\chi_{0,0}(\xiv) e^{i \phi(\Tjk(\xiv))}\F(E)(\Tjk(\xiv))}(\xv) 
        =
        \FI\kl{\chi_{0,0}(\xiv) e^{i \phi(\Tjk(\xiv))}}(\xv) \,.
    \end{split}
    \end{equation*}
Comparing this with \eqref{intensity_point_source}, we observe that the corresponding intensity $\Imjk := \abs{\Emjk}^2$ essentially amounts to a distorted version of $\abs{\FI\kl{\chi_{0,0}}}^2$, the diffraction-limited PSF. Furthermore, we see that the distortion itself is proportional to the local aberration $\phi$ on the subdomain $\Ojk$, or in a first-order approximation to its average gradient. In the following proposition, we see that this average gradient is responsible for the shifts of the focal spots on the subapertures of the SH-WFS, as discussed in Section~\ref{subsect_SHWFS}.

\begin{proposition}
Let $\Emjk$ be defined as in \eqref{def_Emjk}, let $\Em$ satisfy \eqref{model_aperture} and assume that $\F(E) \equiv 1$. Furthermore, assume that on each subdomain $\Ojk$ the wavefront aberration $\phi$ can be reasonably well approximated by a linear function, i.e., there holds
    \begin{equation}
        \phijk(\xiv) := \phi(\Tjk(\xiv)) \approx 2\pi \kl{ \cjk +  \svjk \cdot \xiv } \,,
        \qquad
        \forall \, \xiv \in \Omega_{0,0} \,. 
    \end{equation}
Then for all indices $j,k$ and all $\xv \in \Rt$ there holds
    \begin{equation}\label{subimages_shifted}
        \abs{\Emjk(\xv  - (\svjk - \svzz),\phi) }
        \approx
        \abs{  \Em_{0,0}(\xv ,\phi) } \,.
    \end{equation}
\end{proposition}
\begin{proof}
This assertion follows directly from \cite[Proposition~3.8]{Hubmer_Sherina_Ramlau_Pircher_Leitgeb_2023} using $\F(E) \equiv 1$.    
\end{proof}

From \eqref{subimages_shifted} one can see that the intensity images on each subaperture of the CCD detector of the SH-WFS are essentially shifted versions of each other, and that the relative shifts $(\svjk - \svzz)$ correspond to the average slopes of the local wavefront aberration. These relative shifts can e.g.\ be estimated via matching the subimages $\abs{\Emjk}$ and $\abs{\Em_{0,0}}$; see e.g.\ \cite{Kumar_Drexler_Leitgeb_2013,Kumar_Wurster_Salas_Ginner_Drexler_Leitgeb_2017,Hubmer_Sherina_Ramlau_Pircher_Leitgeb_2023}. Alternatively, as suggested in \cite{Hubmer_Sherina_Ramlau_Pircher_Leitgeb_2023} they can also be calculated via
    \begin{equation}\label{subimages_slopes}
        \svjk - \svzz
        \, \approx \,
        \frac{\int_\Rt \xv \abs{\Emzz(\xv,\phi) } \, d\xv}{\int_\Rt \abs{\Emzz(\xv,\phi) } \, d\xv}
        -
        \frac{\int_\Rt \xv \abs{  \Emjk(\xv,\phi) } \, d\xv}{\int_\Rt \abs{\Emjk(\xv,\phi) } \, d\xv} \,,
    \end{equation}
i.e., via comparing the centers of mass of the subimages $\abs{\Emjk}$, as is commonly done for the SH-WFS \cite{Gilles_Ellerbroek_2006}. Once the relative shifts $(\svjk - \svzz)$ have been found, what remains is to compute an approximation $\phi$ from these measurements, i.e., Step~3 in Algorithm~\ref{algo_generic}. For this, any wavefront reconstruction method for the SH-WFS can be used \cite{Roddier_1999}. In the numerical experiments presented below, the fast and stable Cumulative Reconstructor with Domain Decomposition (CuReD) 
\cite{Zhariy_Neubauer_Rosensteiner_Ramlau_2011,Rosensteiner_2011_01,Rosensteiner_2011_02} is used. 

Summarizing the above considerations, we thus arrive at the following digital SH-WFS phase unwrapping algorithm as a special case of the generic Algorithm~\ref{algo_generic}.

\begin{algo}[Digital SH-WFS phase unwrapping algorithm]\label{algo_dSH_WFS}
Given the wrapped phase $\phiw(\xiv)$ for all $\xiv \in \Omega \subset \Rt$, compute the unwrapped phase $\phir(\xiv)$ via the following steps:
\begin{enumerate}
    %\item Compute $\Emp(x,\phiw)$ as in \eqref{def_Emp}, i.e., $\Emp(\xv,\phiw) =  \FI\kl{ \chi_\Omega e^{i \phiw} }(\xv)$.
    \item Select $J,K \in \R^+$ and subdomains $\Ojk \subset \Omega$ with $\bigcup_{j,k} \Ojk = \Omega$ as in \eqref{def_Ojk}.
    \item For each $j,k$ compute the fields $\Emjk$ via \eqref{def_Emjk} with $\Em(\xv,\phi) = \Emp(\xv,\phiw)$, i.e.,
        \begin{equation*}
        \begin{split}
            \Emjk(\xv,\phiw) = \FI\kl{ \chi_{0,0}(\xiv) e^{i \phiw(\Tjk(\xiv))} }(\xv) \,.
        \end{split}
        \end{equation*}
    \item Estimate the relative slopes $(\svjk-\svzz)$ from the fields $\Emjk$ either through image fitting using \eqref{subimages_shifted}, or compute them directly via the center of mass relation \eqref{subimages_slopes}.
    \item Reconstruct the wavefront aberration $\phi$ from the relative slopes $(\svjk-\svzz)$ using any wavefront reconstruction algorithm for the SH-WFS such as CuReD.
    \item Select the reconstructed wavefront aberration as the unwrapped phase estimate $\phir$.
\end{enumerate}
\end{algo}

\begin{remark}
As with the actual SH-WFS, the number and thus the size of the subapertures $\Ojk$ has a direct impact on the overall reconstruction quality in the digital case. On the one hand, if the number of subapertures is too small, and thus their relative size too large, then the reconstructed wavefront will be comparatively smoother than the original wavefront, leading to insufficient resolution of finer details. On the other hand, if too many and thus too small apertures are chosen, then the reconstruction becomes susceptible to noise. See the results presented in Figure~\ref{fig_OCTOPUS_results_SH} for an illustrative example. Hence, a proper balance between the number and size of the subapertures needs to be struck to address this inherent ill-posedness effect. Doing so in an algorithmic way is rather difficult, requiring estimates of both the amount of noise and the expected smoothness of the original wavefront aberration. While not entirely circumventing these issues, the use of additional regularization in the wavefront reconstruction step, or alternatively the use of an ``averaged'' digital SH-WFS, can be beneficial here \cite{Hubmer_Sherina_Ramlau_Pircher_Leitgeb_2023}.
\end{remark}

% % % % % % % % % % % % % % % % % % % % % % % % % % % %
% Subsection - Digital Fourier-type wavefront sensors %
% % % % % % % % % % % % % % % % % % % % % % % % % % % %
\subsection{Digital Fourier-type wavefront sensors}\label{subsect_digFWFS}

In this section, we consider the choice of Fourier-type WFSs in Step~2 of Algorithm~\ref{algo_generic}. As we shall see below, these WFSs have a simple mathematical description in terms of Fourier optics, and thus the computation of digital WFS measurements required in this step is even more straightforward than for the digital SH-WFS considered above.

First, compare the two-lens imaging system depicted in Figure~\ref{fig_image_formation_model} with the schematic depiction of a Fourier-type WFS in Figure~\ref{fig_fwfs}. The main difference is that instead of the wavefront aberration $\phi$ and the aperture $\Omega$, in Figure~\ref{fig_fwfs} there is an optical element characterized by an OTF. However, these two settings are closely related: As in Figure~\ref{fig_image_formation_model} let $E$ denote the incident field entering the Fourier-type WFS, and let $\Em$ denote the filtered field at the conjugate plane. Then, following the Fourier-optics description of this measurement system setting (cf.~\cite{HuNeuSha_2023,Fauv16,Goodman_2005}), we obtain that
    \begin{equation}\label{model_pyramid_OTF}
        \Em(\xv) =  \FI\kl{ \OTF_\psi(\xiv) \F(E)(\xiv) }(\xv) \,,
    \end{equation}
where, depending on the shape $\psi$ of the chosen optical element, the OTF is given by
    \begin{equation}\label{eq_helper_1}
        \OTF_\psi(\xiv) := e^{i\psi(\xiv)} \,.
    \end{equation}
Comparing this with \eqref{model}, we can see that the shape function $\psi$ now takes the role of the wavefront aberration $\phi$. For the different optical elements shown in Figure~\ref{fig_otfs}, the corresponding shape functions $\psi$ are listed in Table~\ref{table_psi}.

\begin{table}[ht!]
\centering
    \begin{tabular}{|c|c|c|c|}
        \hline  
        & 4-sided PWFS
        & x-Roof WFS
        & 3-sided PWFS 
        \\
		\hline 
        $\psi(\xi_1,\xi_2)=$ 
        & $c(|\xi_1|+|\xi_2|)$ 
        & $c|\xi_1|$ 
        & $\begin{cases} 
            -2c\xi_1 \,, & \frac{-\pi}{3}< \arctan{\frac{\xi_1}{\xi_2}} < \frac{\pi}{3} \\ c(\xi_1-\sqrt{3}\xi_2)\,, 
            & \frac{\pi}{3}< \arctan{\frac{\xi_1}{\xi_2}} < \pi \\ c(\xi_1+\sqrt{3}\xi_2) \,, & \text{otherwise}
            \end{cases}$ 
        \\
        \hline 
        & Cone WFS 
        & y-Roof WFS
        & iQuad WFS 
        \\
		\hline 
        $\psi(\xi_1,\xi_2)=$ 
        & $c\sqrt{|\xi_1|^2+|\xi_2|^2}$ 
        & $c|\xi_2|$ 
        &  $\begin{cases} 
            \frac{\pi}{2}\,, & \xi_1 \xi_2 < 0 \,, \\ 0 \,, & \text{otherwise} 
            \end{cases}$ \\
        \hline 
	\end{tabular}
    \caption{Shape functions $\psi(\xiv)$, with $\xiv = (\xi_1,\xi_2)$, corresponding to the Fourier-type WFSs depicted in Figure~\ref{fig_fwfs}. The constant $c>0$ relates to the apex angle of the pyramidal prisms. The typical choice $c \approx 1$ is used in the numerical examples presented in Section~\ref{sect_num_exp}.}
    \label{table_psi}
\end{table}

For Fourier-type WFSs in astronomical settings, one typically assumes \cite{Fauv16} that
    \begin{equation}\label{eq_helper_2}
        E(\xv) = E(\xv,\phi) := \chi_\Omega(\xv) e^{-i \phi(\xv)} \,, 
    \end{equation}
where $\phi$ is the wavefront aberration to be reconstructed, and the domain $\Omega$ corresponds to the aperture of the imaging system. This is physically realized by an additional lens placed in front of the imaging system in Figure~\ref{fig_fwfs} in case the incoming light stems from a sufficiently distant point source. With this, \eqref{model_pyramid_OTF} becomes
    \begin{equation}\label{model_pyramid}
        \Em_\psi(\xv,\phi) := \Em(\xv) = \FI\kl{ e^{i \psi(\xiv)} \F\kl{\chi_\Omega e^{-i \phi}}(\xiv) }(\xv) \,,
    \end{equation}
where we have used \eqref{eq_helper_1} and \eqref{eq_helper_2}. Finally, the camera/detector in the conjugate plane of the imaging system in Figure~\ref{fig_fwfs} measures the intensity of the field $\Em$, i.e.,
    \begin{equation}\label{pyramid_intensity}
        \Im_\psi(\xv,\phi) := \abs{ \Em_\psi(\xv,\phi) }^2
        =
        \abs{\FI\kl{ e^{i \psi(\xiv)} \F\kl{\chi_\Omega e^{-i \phi}}(\xiv) }(\xv)}^2 \,.
    \end{equation}

For the reconstruction of the aberration $\phi$ from the intensity measurements $I_\psi$, a number of different reconstruction algorithms have been proposed. Most of these depend heavily on the shape of the chosen optical element, i.e., on the concrete form of $\psi$, and are often based on linearizations of \eqref{pyramid_intensity}. An overview of reconstruction methods for the most commonly used 4-sided PWFS can, e.g., be found in \cite{ShaHuRa20}. Additionally, AI-based methods such as Deep Optics PWFS have recently been proposed \cite{Vera_AO4ELT7proc}. For arbitrary choices of $\psi$, including those which do not correspond to a pyramidal prism but can, e.g., be realized by a spatial light modulator, the NOnlinear Pyramid Extension (NOPE)~\cite{HuNeuSha_2023} has been proposed.   

Summarizing the above considerations the following digital Fourier-type WFS phase unwrapping algorithm is proposed as a special case of the generic Algorithm~\ref{algo_generic}.

\begin{algo}[Digital Fourier-type WFS phase unwrapping algorithm]\label{algo_dFb_WFS}
Given the wrapped phase $\phiw(\xiv)$ for all $\xiv \in \Omega \subset \Rt$, compute the unwrapped phase $\phir(\xiv)$ via:
\begin{enumerate}
    \item Compute the intensity measurements $\Im_\psi$ via \eqref{pyramid_intensity} with $\phi = \phiw$, i.e., 
    \begin{equation*}
        \Im_\psi(\xv,\phiw) = \abs{\FI\kl{ e^{i \psi(\xiv)} \F\kl{\chi_\Omega e^{-i \phiw}}(\xiv) }(\xv)}^2 \,.
    \end{equation*}
    \item Reconstruct the wavefront aberration $\phi$ from the intensity measurements $\Im_\psi$ using any wavefront reconstruction algorithm for Fourier-type WFSs such as the NOPE.
    \item Select the reconstructed wavefront aberration as the unwrapped phase estimate $\phir$.
\end{enumerate}
\end{algo}

As noted above, the 4-sided PWFS is the most commonly used Fourier-type WFS in practice. In \cite{Shat13}, the Preprocessed Cumulative Reconstructor with Domain Decomposition (PCuReD), based on CuReD for the SH-WFS discussed above, is suggested as a suitable and highly efficient reconstructor, which gives rise to the following algorithm.

\begin{algo}[Digital 4-sided PWFS phase unwrapping algorithm]\label{algo_dP_WFS}
Given the wrapped phase $\phiw(\xiv)$ for all $\xiv =(\xi_1,\xi_2)\in \Omega \subset \Rt$, compute the unwrapped phase $\phir(\xiv)$ via:
\begin{enumerate}
    \item Compute the intensity measurements $\Im_\psi$ via \eqref{pyramid_intensity} with $\phi = \phiw$ and choosing $\psi$ corresponding to the 4-sided PWFS in Table~\ref{table_psi} for some $c>0$, i.e., 
    \begin{equation*}
        \Im_\psi(\xv,\phiw) = \abs{\FI\kl{ e^{i c(|\xi_1| + |\xi_2|)} \F\kl{\chi_\Omega e^{-i \phiw}}(\xiv) }(\xv)}^2 \,.
    \end{equation*}
    \item Reconstruct the aberration $\phi$ from the intensity measurements $\Im_\psi$ using PCuReD.
    \item Select the reconstructed wavefront aberration as the unwrapped phase estimate $\phir$.
\end{enumerate}
\end{algo}

\begin{remark}
Additional extensions of the Fourier-type WFS unwrapping algorithms are possible by incorporating the physical concept of modulation. Especially in astronomical AO systems, modulation of the electric field is often added to the Fourier-type WFS setup. E.g.\ by moving around the apex of the optical element, the light distribution to all faces of the optical element is improved, helping with certain nonlinearity issues \cite{ShaHuRa20}.

Mathematically, this modulation is included in the Fourier-optics model \eqref{model_pyramid} via adding a time-dependent modulation function $\vphi^{\mod}_{t}(\xv)$ in the image plane, i.e.,
    \begin{equation}\label{model_pyramid_modulated}
        \Em_\psi(\xv,\phi,t) = \FI\kl{ e^{i \psi(\xiv)} \F\kl{\chi_\Omega e^{-i (\phi + \vphi^{\mod}_{t})}}(\xiv) }(\xv) \,.
    \end{equation}
In the case of a round aperture $\Omega$ with diameter $D$, a circular modulation by a radius $r$ around the apex is e.g.\ described by the modulation function
    \begin{equation}\label{i3}
        \vphi^{\mod}_{t}(\xv) = \frac{2\pi r}{D} \kl{ \xv_1 \cos(2\pi t) + \xv_2 \sin(2\pi t)} \,.
    \end{equation}
The resulting intensity on the detector is then modelled by the time average
    \begin{equation}\label{i4}
        \Tilde{I}^m_\psi(\xv,\phi) := \frac{1}{T}\int_0^T \abs{\Em_\psi(\xv,\phi,t)}^2 \,dt  \,,
    \end{equation}
where $T$ denotes the duration of one full modulation period. Other modulation scenarios are also possible~\cite{ShaHuRa20}. Hence, one can again follow the concept for phase unwrapping discussed above, i.e., computing $\Tilde{I}^m_\psi(\xv,\phiw)$ and using an appropriate wavefront reconstructor to obtain an unwrapped phase $\phir$. This then leads to a complete new class of modulated Fourier-type WFS phase unwrapping algorithms.
\end{remark}

% % % % % % % % % % % % % % % % % %
% Section - Numerical experiments %
% % % % % % % % % % % % % % % % % %
\section{Numerical experiments}\label{sect_num_exp}
In this section, we numerically evaluate our algorithms introduced in Section~\ref{sect_digWFS} on a real-world phase unwrapping problem appearing in FSOC, and compare the results to those obtained with state-of-the-art algorithms.

In FSOC, laser beams are used for the wireless transmission of data in, e.g., optical laser communications between satellites and ground-based stations. Thanks to high bandwidths, laser communication provides large and fast data transfer rates. However, when the laser signal travels through the Earth's atmosphere, it gets disrupted by turbulence. Recent research in FSOC focuses on using AO to overcome this problem; see e.g.\ \cite{Osborn_2021,Farley_2022,Stotts_2021,Martinez_2019,Torres_2022}. In the design phase of AO systems for FSOC, simulations of laser light propagation including diffraction and interference are required. To this end, physical optics propagation is commonly employed \cite{Basden_2018}, which uses Fourier transforms to compute phase angles. As these phase angles contain $2\pi$ ambiguities, fast and accurate phase unwrapping algorithms are of great interest for solving this problem.

The digital WFS phase unwrapping algorithms described in Section~\ref{sect_digWFS} were all implemented in MATLAB. As mentioned above, for the digital SH-WFS based unwrapping the CuReD method is used in Step~4 of Algorithm~\ref{algo_dSH_WFS}. This algorithm requires neither parameter tuning nor an initial guess \cite{Zhariy_Neubauer_Rosensteiner_Ramlau_2011,Rosensteiner_2011_01,Rosensteiner_2011_02}. Similarly, for Step~2 in Algorithm~\ref{algo_dP_WFS} the PCuReD algorithm is used. Finally, in Step~2 of Algorithm~\ref{algo_dFb_WFS}, the NOPE algorithm is used, which does require an initial guess, for which either zero, or an approximation computed via PCuReD has been used \cite{HuNeuSha_2023}. This is denoted below by ``zero starting'' and ``linear starting'', respectively. Furthermore, the NOPE also has an optional smoothness parameter $s$, which, unless noted otherwise, is set to the canonical choice $s=11/6$ as suggested in \cite{HuNeuSha_2023}; see also \cite{Kolmogorov_1991}.

% % % % % % % % % % % % % % % % % % % % % % % % %
% Subsection - Comparison with existing methods %
% % % % % % % % % % % % % % % % % % % % % % % % %
\subsection{Comparison with existing methods}\label{subsect_impl}

To validate the performance of our proposed algorithms, we compare them to a number of commonly used local and global phase unwrapping approaches, which are based on line integration of an estimate of the phase gradient. While the former evaluate the integral along a single path, the latter average integrals along different paths.

Among the local methods, we use the built-in MATLAB unwrapping function, as well as the algorithm proposed in \cite{Herraez_2002}. The MATLAB built-in unwrapping works in one-dimension. If the input is a matrix, the function operates column-wise. The phase is unwrapped by adding multiples of $2\pi$ to the phase whenever the jump between two consecutive pixels is greater than or equal to $\pi$ radians. The method proposed in \cite{Herraez_2002} is more advanced, and is based on unwrapping most reliable pixels (MRP). Its MATLAB implementation is available in \cite{Firman_2023}. As a first step, the algorithm calculates the reliability of each pixel using second order derivatives. Then, it follows a non-continuous path to perform the unwrapping process, in which pixels with higher reliability are unwrapped first. This algorithm is robust, fast, and is used inside AO simulation tools. However, in situations where the noise level is high the quality of the unwrapping suffers.

In contrast to local methods, global methods utilize the wrapped values together with some additional (global) assumptions, such as the probability distribution of the noise or the regularity of the phase field. Among those methods, we select one algorithm based on the PE \cite{Zhao_2020} and one based on the TIE \cite{Zhao_2019}. For both methods, MATLAB codes are available in \cite{Zhao_2023_1} and \cite{Zhao_2023_2}, respectively. In these approaches, first an auxiliary complex field is generated using the wrapped phase. Then, a two-dimensional second-order elliptic partial differential equation (PDE) is solved, which uses the longitudinal intensity derivative of this field as an input. As the names suggest, in the case of the PE algorithm the Poisson equation is solved, while in the case of the TIE the transport of intensity equation is solved. The resulting phase profile is then automatically in unwrapped form, since it has been obtained as the solution of a PDE rather than as the argument of a complex-valued function. Those approaches have already been successfully applied in, e.g., interferometry and digital holography applications \cite{Pandey_2016}.

In recent years, deep learning methods have become more and more popular for phase unwrapping. Promising results have been shown in the literature; see e.g. \cite{Zhang2019,Wang2022,Huang2022}. However, the training phase of neuronal networks requires a lot of data samples, experience, and time. Moreover, it is not so straight forward to include the neuronal network within the AO simulation tool we are using, and would also require a new network training process adapted to our discussed application. Hence, we do not further consider those approaches in the framework of this paper.

% % % % % % % % % % % % % % % % % % 
% Subsection - Numerical results  %
% % % % % % % % % % % % % % % % % %
\subsection{Numerical results}\label{subsect_num_res}

In this section, we present the results of our digital WFS phase unwrapping approaches applied to two sets of data, the first with a known ground truth and the second coming from a real-world example without ground truth.

% Numerical experiment 1: Unwrapping with known ground truth
\subsubsection*{Numerical experiment 1: Unwrapping with known ground truth}

\begin{figure}[ht!]
    \centering
    \includegraphics[width=0.3\textwidth, trim = {12cm 0.5cm 11cm 1cm}, clip=true]{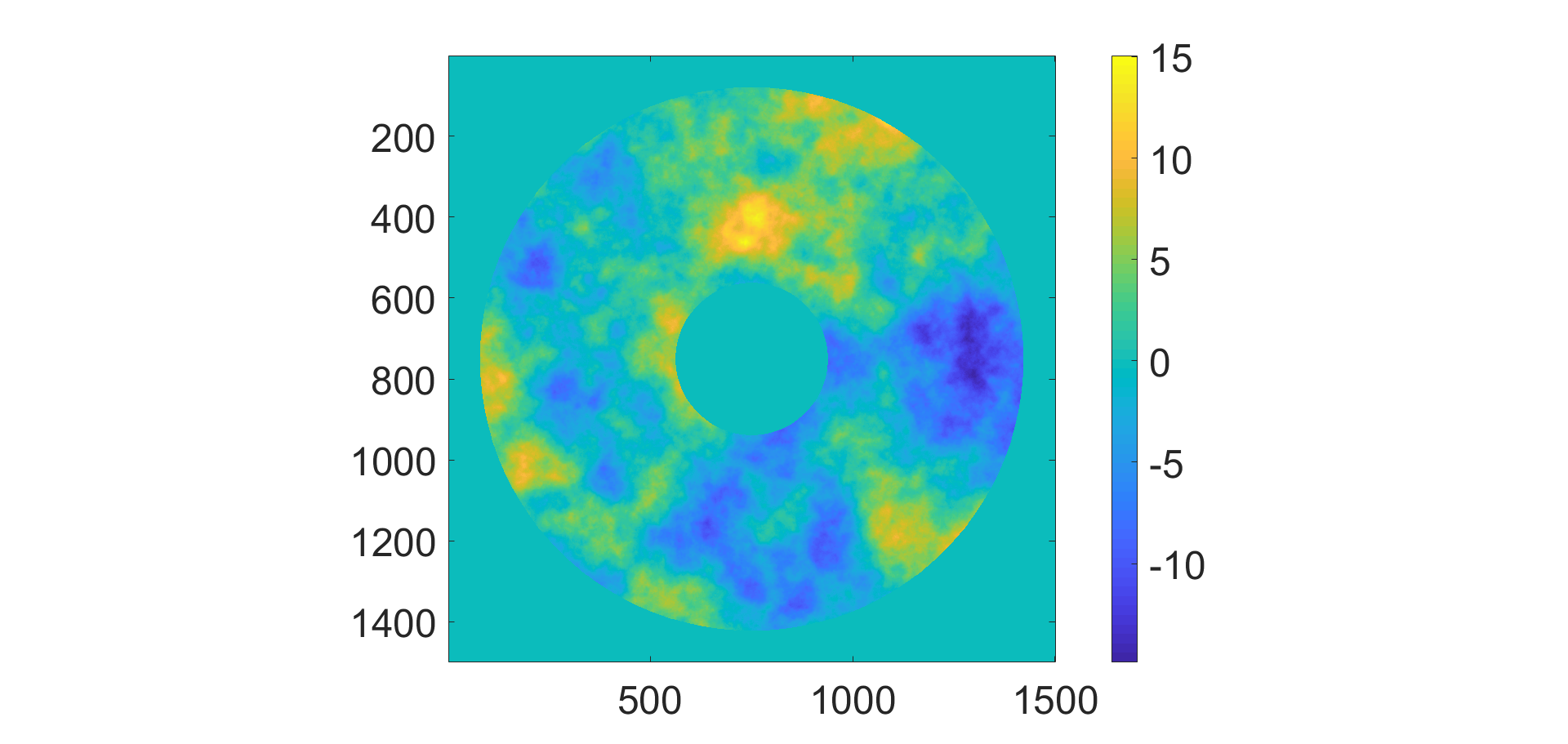}
    \quad
    \includegraphics[width=0.3\textwidth, trim = {12cm 0.5cm 11cm 1cm}, clip=true]{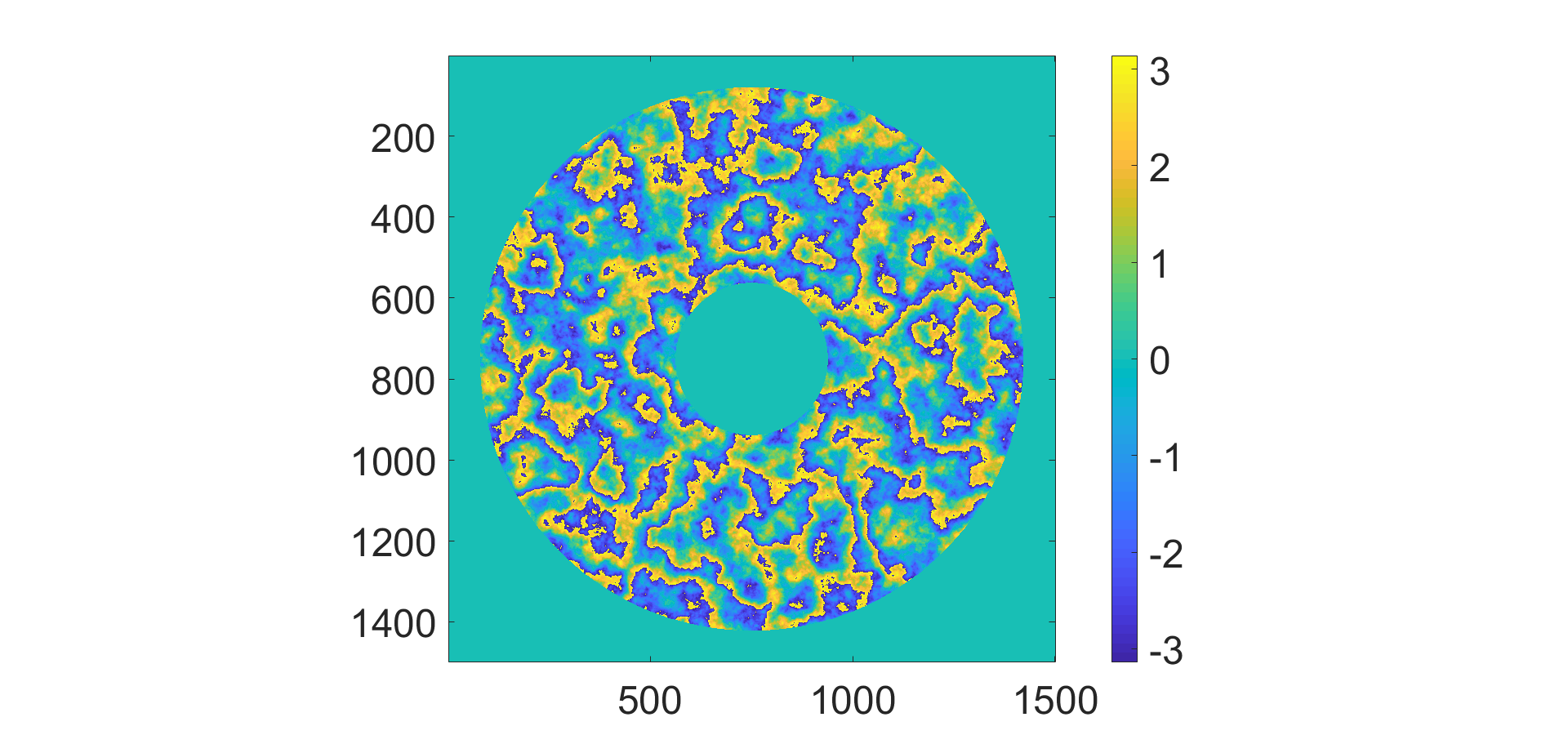}
    \quad
    \includegraphics[width=0.3\textwidth, trim = {12cm 0.5cm 11cm 1cm}, clip=true]{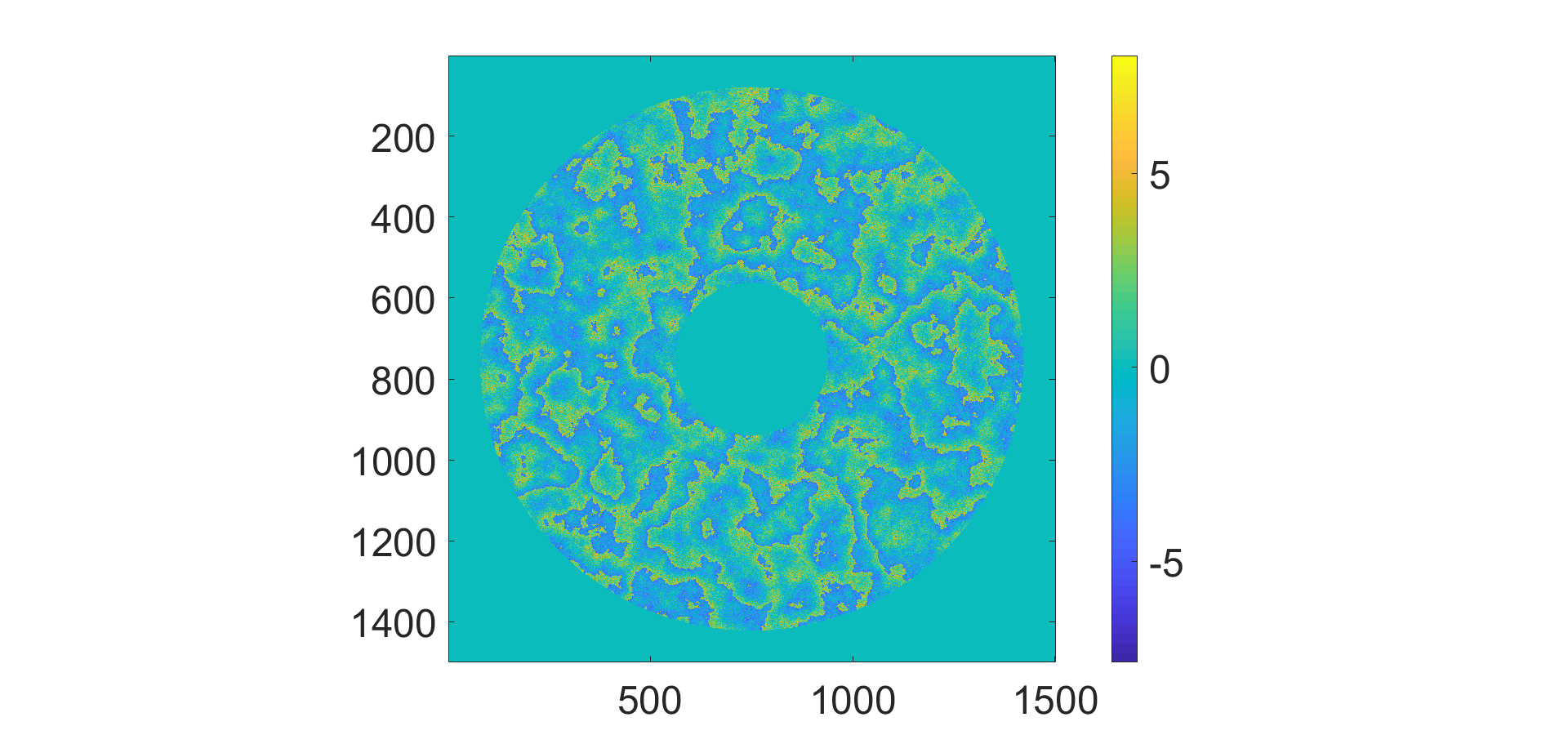}
    \caption{Non-wrapped phase $\phi$ created by the OCTOPUS simulation tool (left), wrapped phase $\phiw$ (middle), and noisy wrapped phase $\phiwd$ including $20\%$ noise (right). }
    \label{fig_OCTOPUS_data}
\end{figure}

For the first set of numerical tests, we consider a random phase $\phi$ created by the OCTOPUS simulation tool (popular in the AO community) of the European Southern Observatory \cite{octopus06}, corresponding to the wavefront aberration caused by some atmospheric turbulence profile. Figure~\ref{fig_OCTOPUS_data} contains both the actual phase $\phi$, the manually wrapped phase $\phiw$ created using the MATLAB built-in function \textit{wrapToPi()}, and its noisy version $\phiwd$ containing $20\%$ relative uniform random noise. The noise was added after the wrapping, which accounts for the larger range of values of $\phiwd$ in comparison with $\phiw$. This noisy wrapped phase $\phiwd$ was then used as input for the unwrapping algorithms.

\begin{figure}[ht!]
\centering
\small
\resizebox{\columnwidth}{!}{%
    \begin{tabular}{ccc}
        \textbf{\footnotesize 30 $\times$ 30} & \textbf{\footnotesize 100 $\times$ 100} & \textbf{\footnotesize 500 $\times$ 500} \\
        \textbf{\footnotesize subapertures} & \textbf{\footnotesize subapertures} & \textbf{\footnotesize subapertures} \\
        \includegraphics[width=0.3\textwidth, trim = {12cm 0.5cm 11cm 1cm}, clip=true]{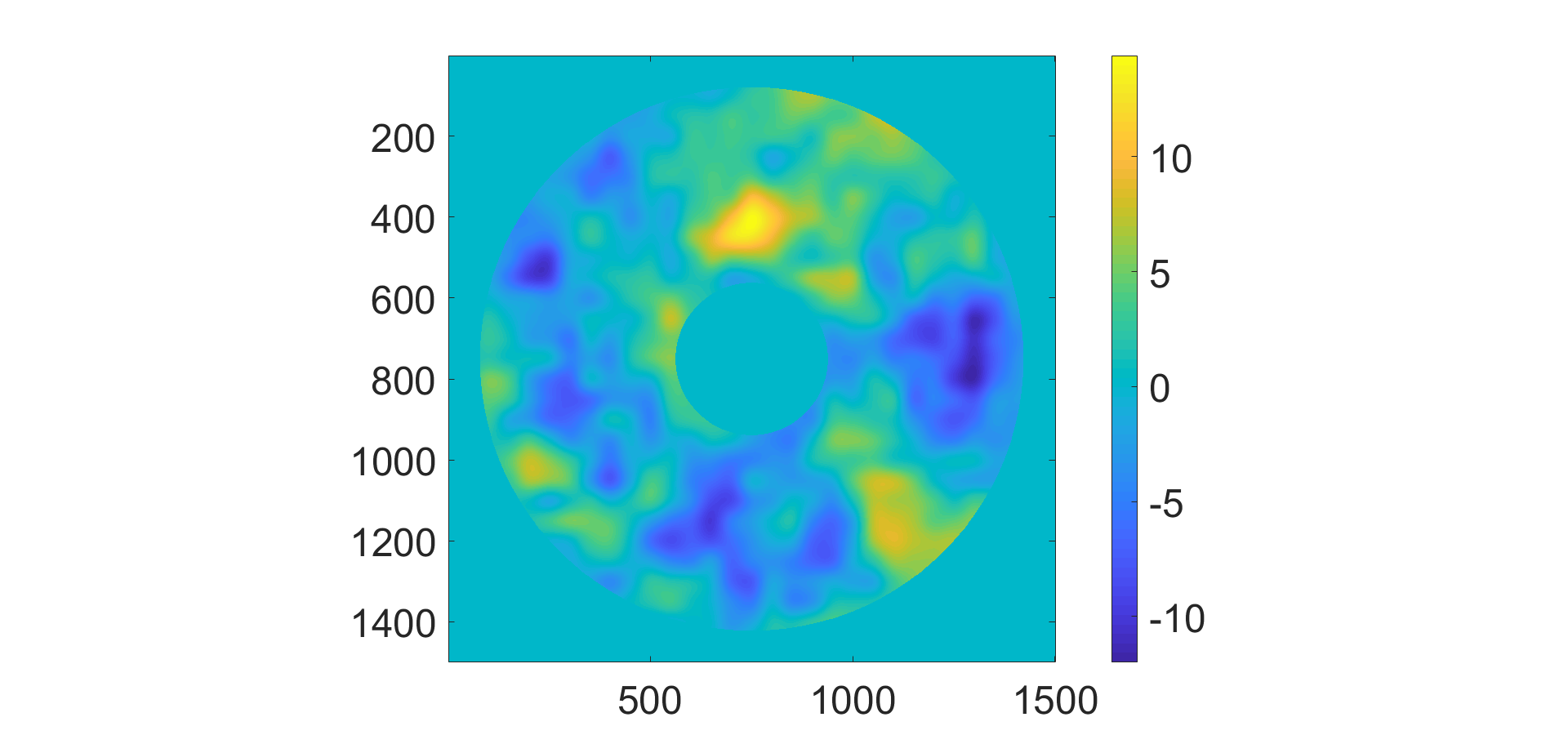} &
        \includegraphics[width=0.3\textwidth, trim = {12cm 0.5cm 11cm 1cm}, clip=true]{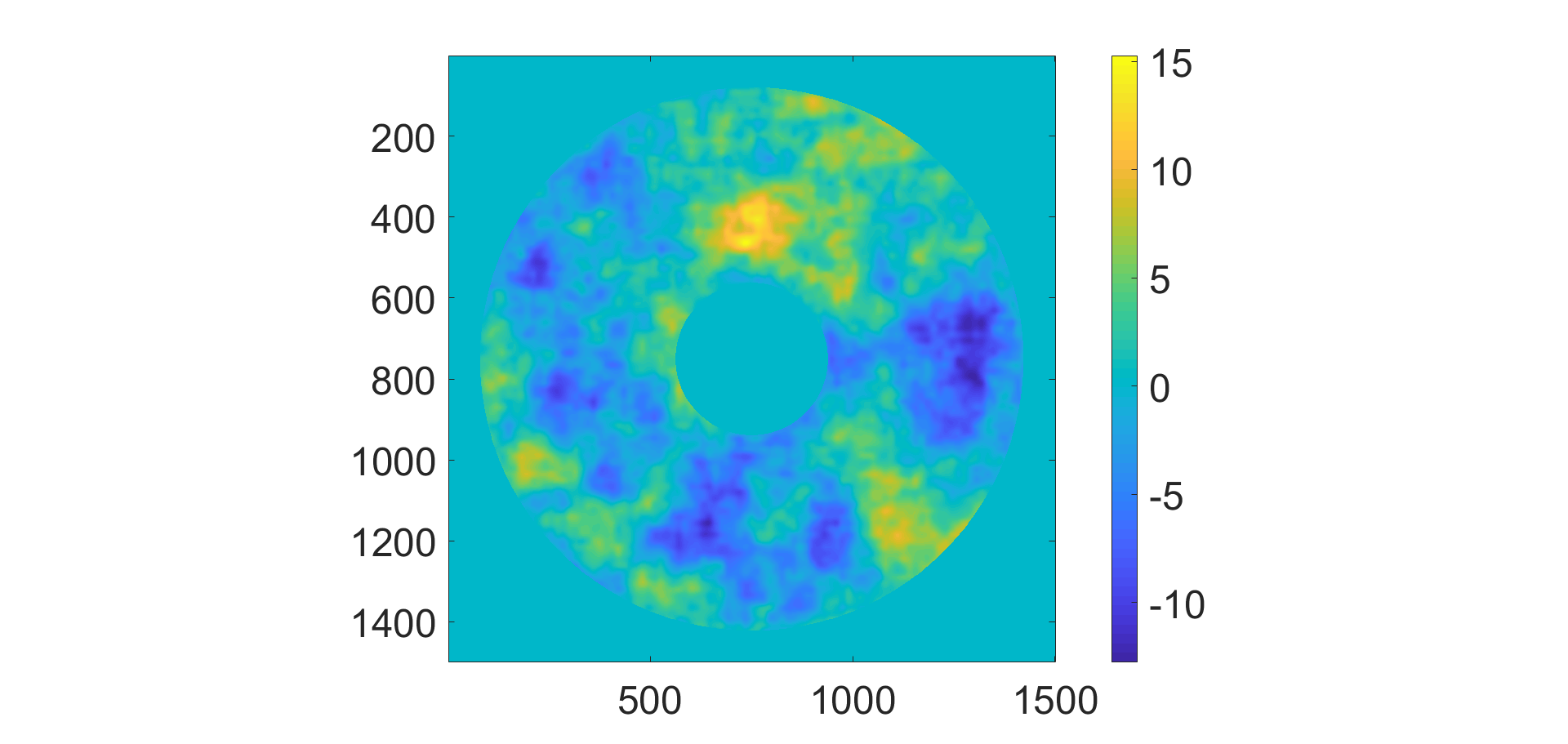} &
        \includegraphics[width=0.3\textwidth, trim = {12cm 0.5cm 11cm 1cm}, clip=true]{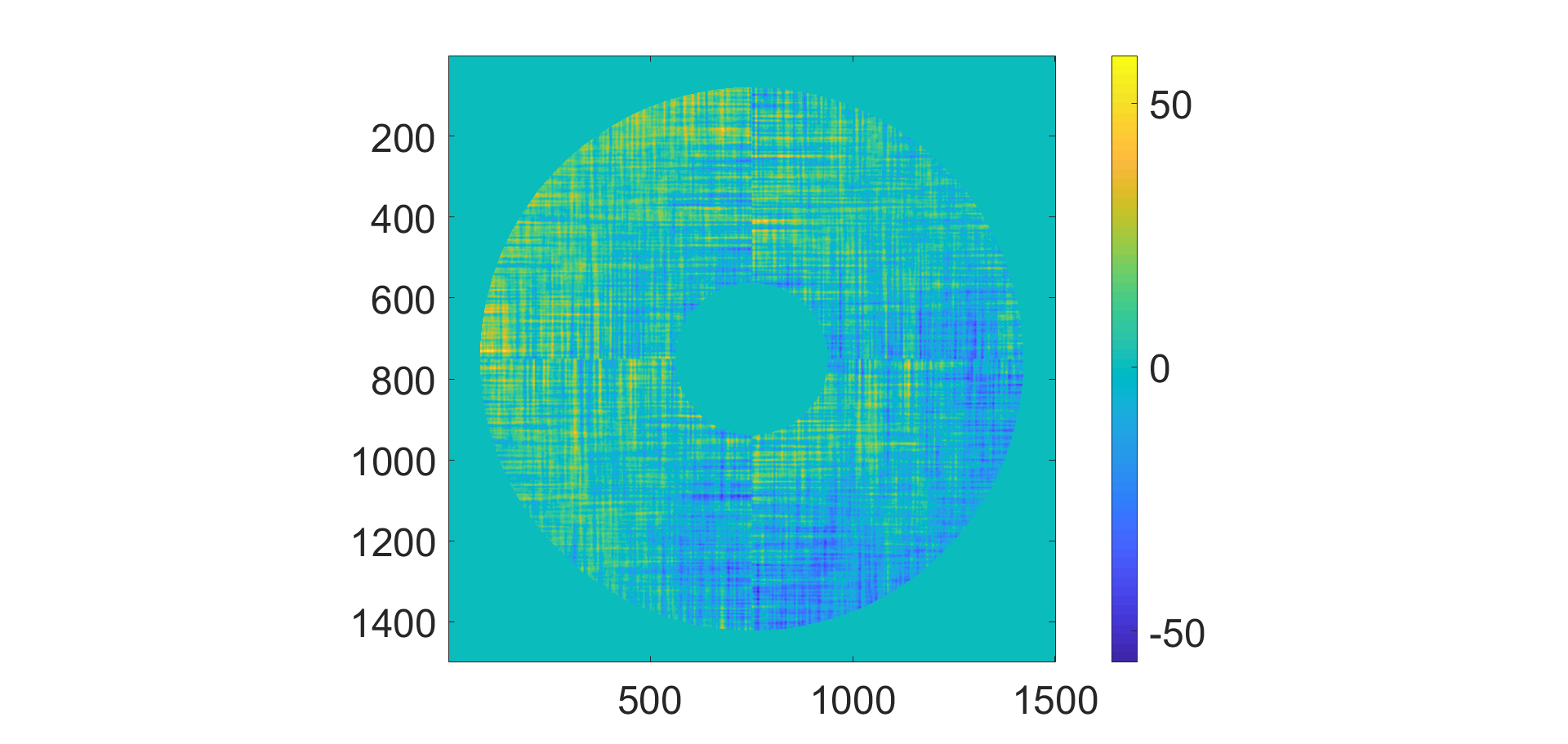} \\
        \includegraphics[width=0.3\textwidth,trim = {12cm 0.5cm 11cm 1cm}, clip=true]{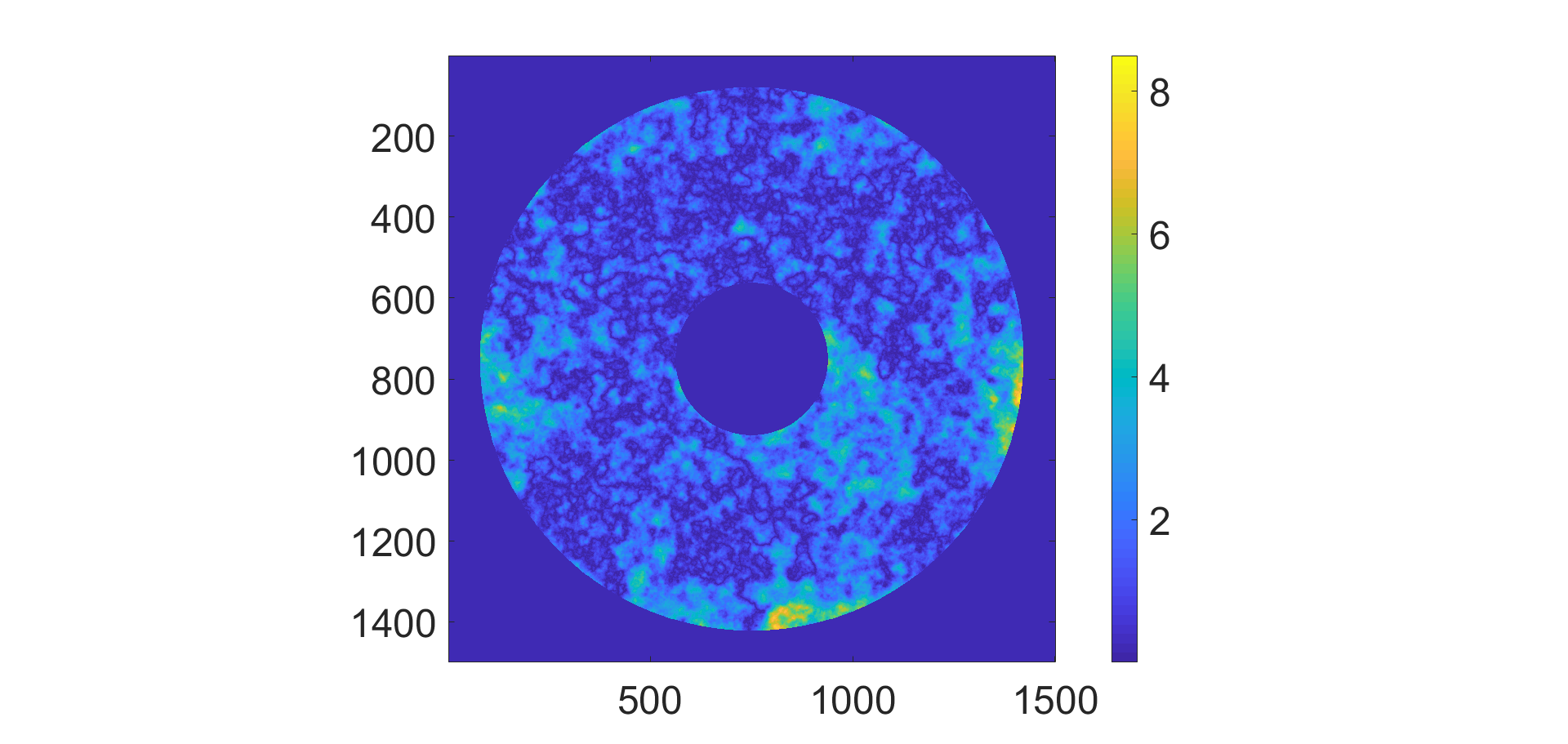} &
        \includegraphics[width=0.3\textwidth,trim = {12cm 0.5cm 11cm 1cm}, clip=true]{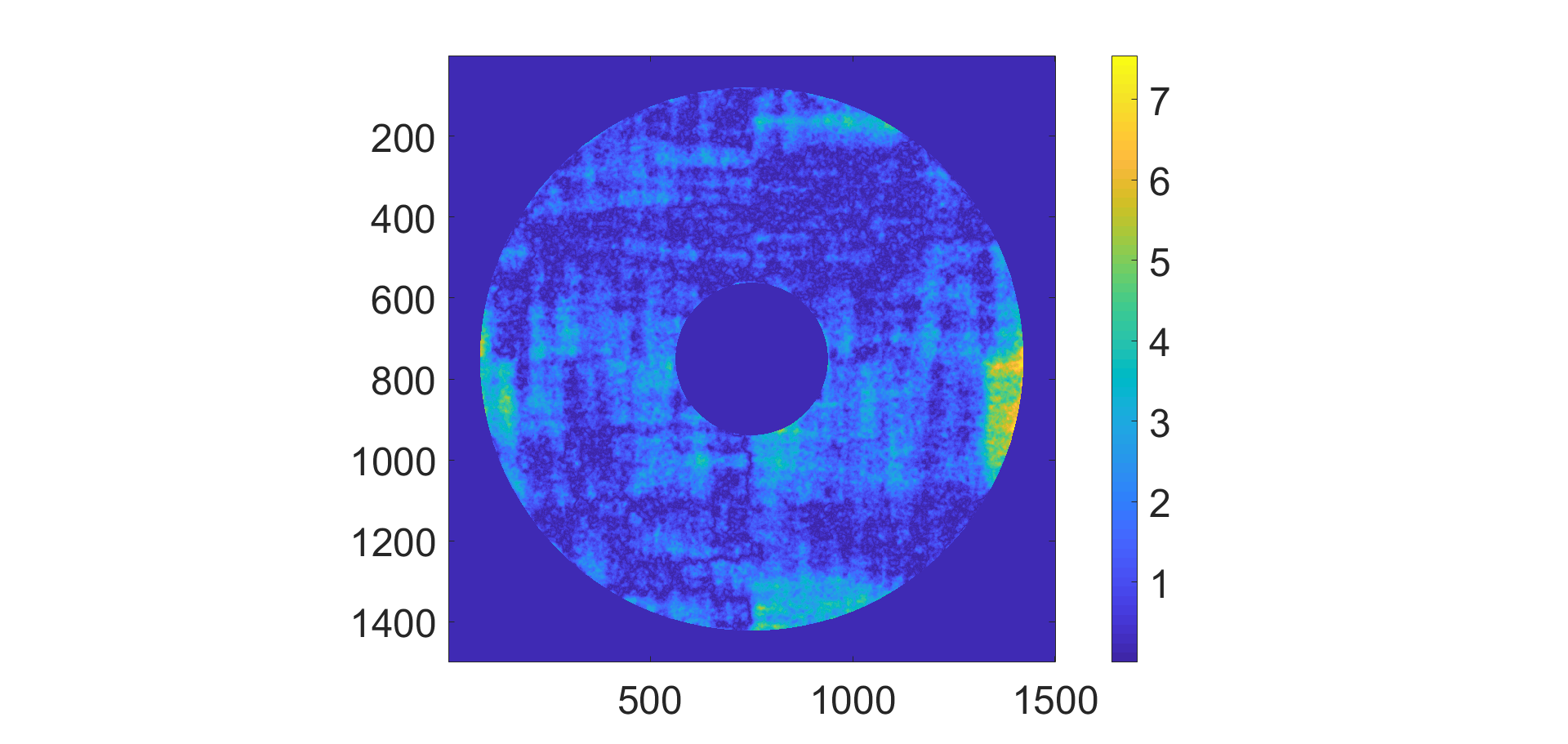} &
        \includegraphics[width=0.3\textwidth,trim = {12cm 0.5cm 11cm 1cm}, clip=true]{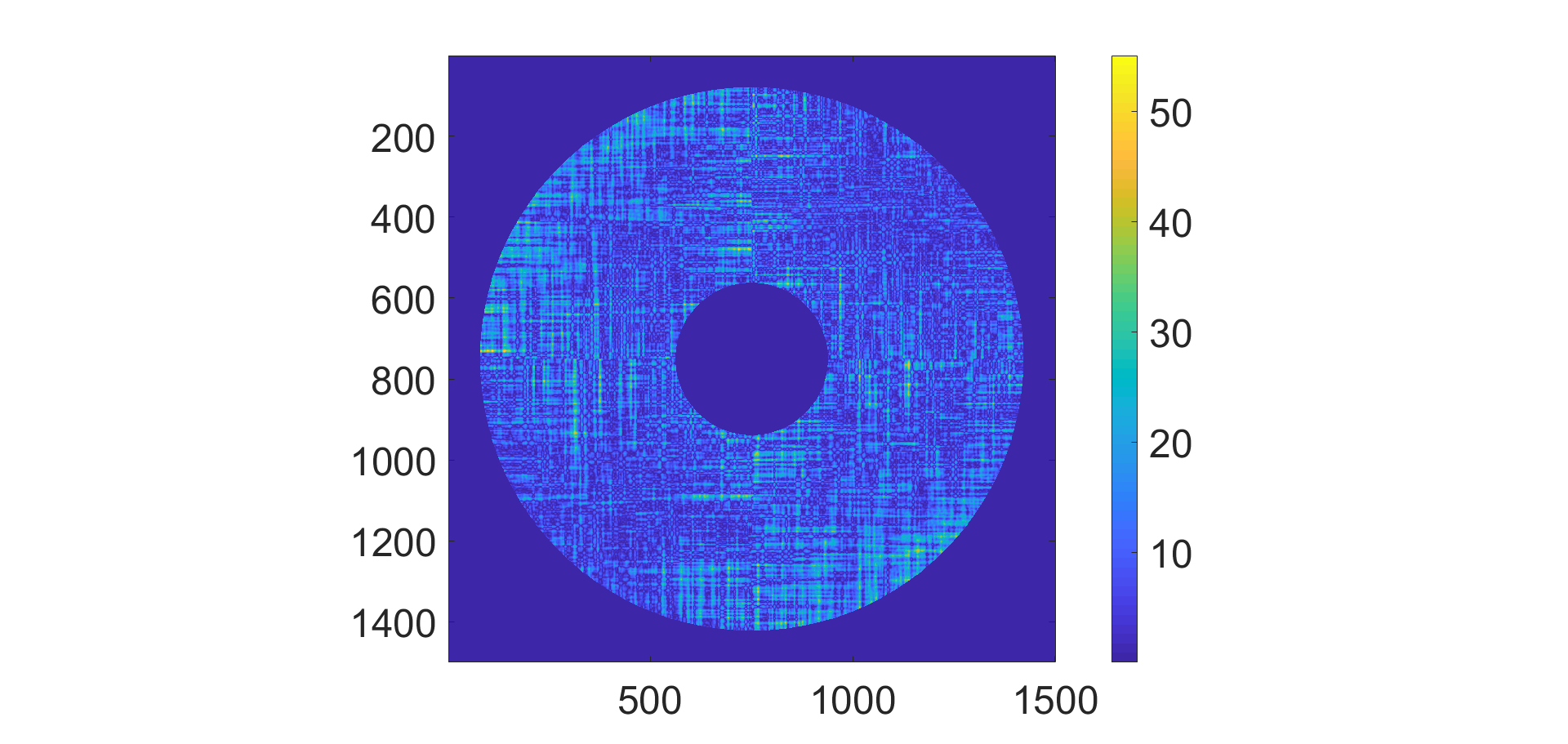} 
    \end{tabular}
    }
    \caption{Unwrapped phases $\phir$ obtained via the digital SH-WFS phase unwrapping approach, i.e., Algorithm~\ref{algo_dSH_WFS} (top) applied to the noisy wrapped phase $\phiwd$ depicted in Figure~\ref{fig_OCTOPUS_data}, and the corresponding absolute difference to the non-wrapped phase $\phi$ (bottom). The number of subapertures corresponds to the number of subdomains $\Ojk$.}
    \label{fig_OCTOPUS_results_SH}
\end{figure}

Figure~\ref{fig_OCTOPUS_results_SH} shows the unwrapped phase $\phir$ obtained via the digital SH-WFS phase unwrapping approach, i.e., Algorithm~\ref{algo_dSH_WFS}, for three different choices of the number of subapertures/subdomains $\Ojk$, as well as the corresponding reconstruction error. Compared to the ground truth phase $\phi$ depicted in Figure~\ref{fig_OCTOPUS_data}, the best results are obtained with $100 \times 100$ subapertures. In the case of $30 \times 30$ subapertures, the structure of the true phase $\phi$ is also still clearly visible, but the resolution is limited due to the comparatively low number of subdomains. However, in the case of $500 \times 500$ subapertures the reconstruction is almost entirely corrupted by noise. This is because the number of subdomains implicitly acts as a regularization parameter in this approach. Thus it has to be chosen appropriately to strike a proper balance between resolution and stability. Additional regularization may also be used, either by adapting Step~4 of Algorithm~\ref{algo_dSH_WFS}, or for example via the averaged digital SH-WFS approach introduced in \cite{Hubmer_Sherina_Ramlau_Pircher_Leitgeb_2023}.

\begin{figure}[ht!]
    \centering
    \small
\resizebox{\columnwidth}{!}{%
    \begin{tabular}{cccc}
        \textbf{\Large PCuReD} & \textbf{\Large NOPE} & \textbf{\Large NOPE} & \textbf{\Large NOPE} \\
        \textbf{\Large Linear} & \textbf{\Large Nonlinear} & \textbf{\Large Nonlinear} & \textbf{\Large Nonlinear} \\
         & \textbf{\Large $s=11/6$} & \textbf{\Large $s=11/6$} & \textbf{\Large $s=1/2$} \\
        \textbf{\Large zero starting} & \textbf{\Large zero starting} & \textbf{\Large linear starting} & \textbf{\Large linear starting} \\
        \includegraphics[width=0.4\textwidth, trim = {12cm 0.5cm 11cm 1cm}, clip=true]{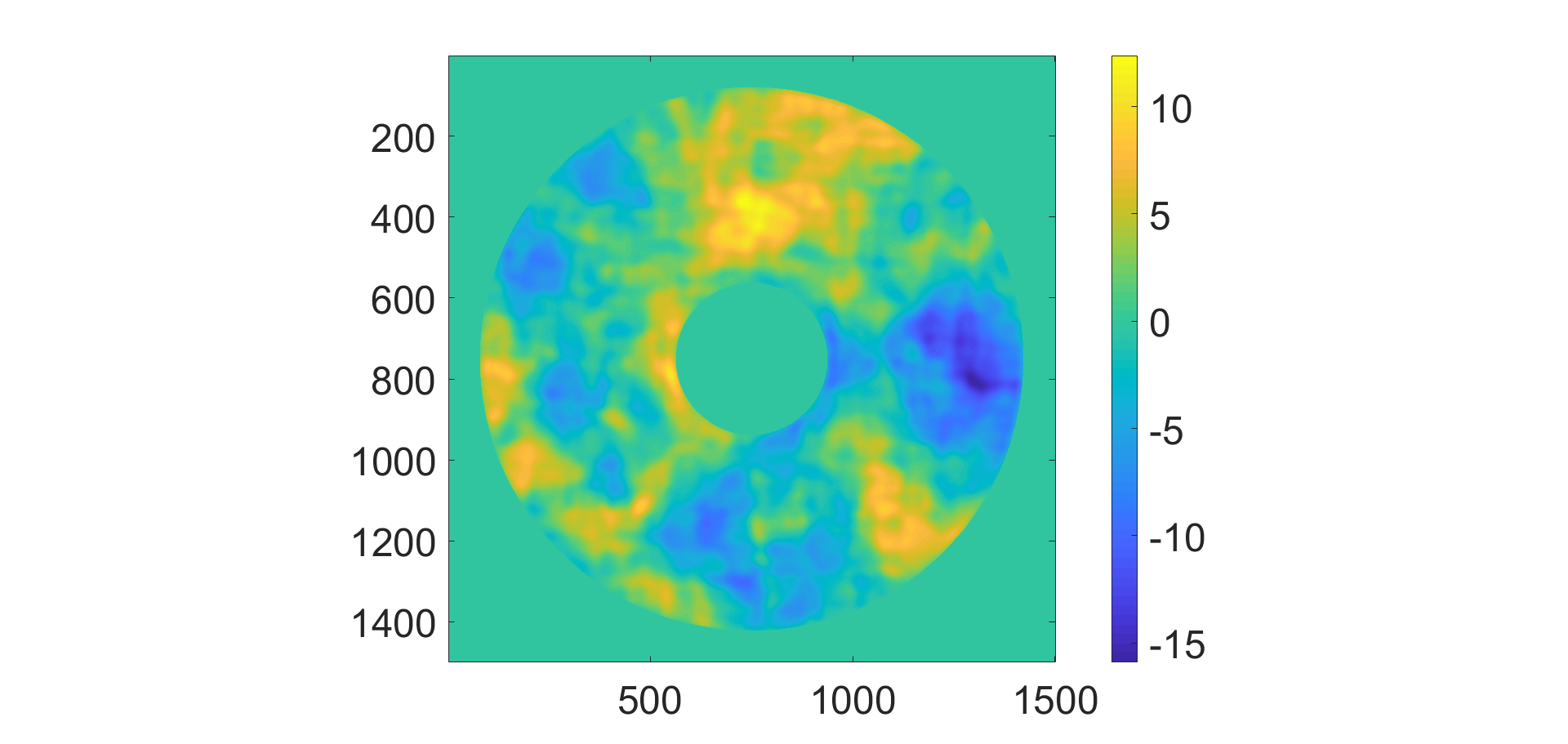} &
        \includegraphics[width=0.4\textwidth, trim = {12cm 0.5cm 11cm 1cm}, clip=true]{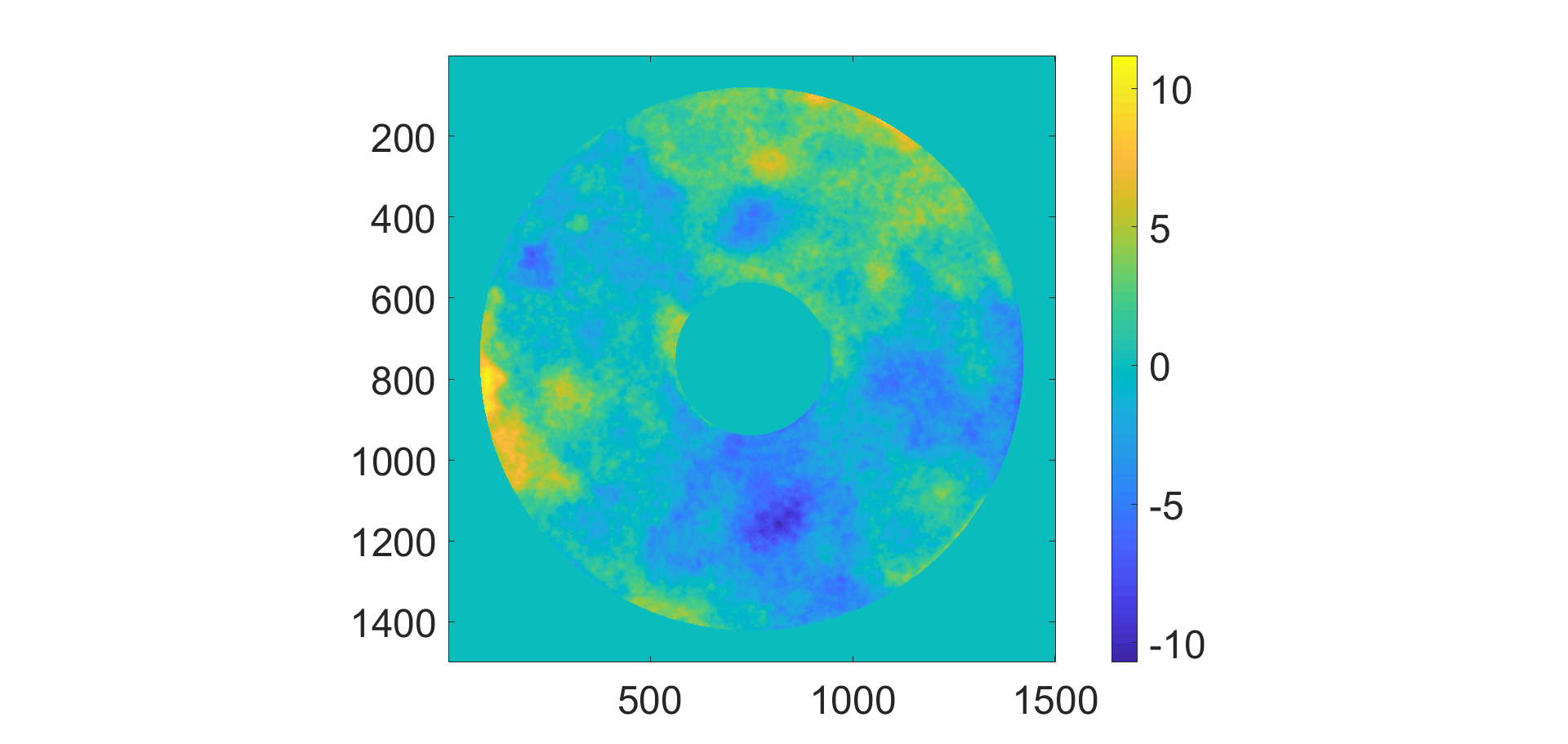} &
        \includegraphics[width=0.4\textwidth, trim = {12cm 0.5cm 11cm 1cm}, clip=true]{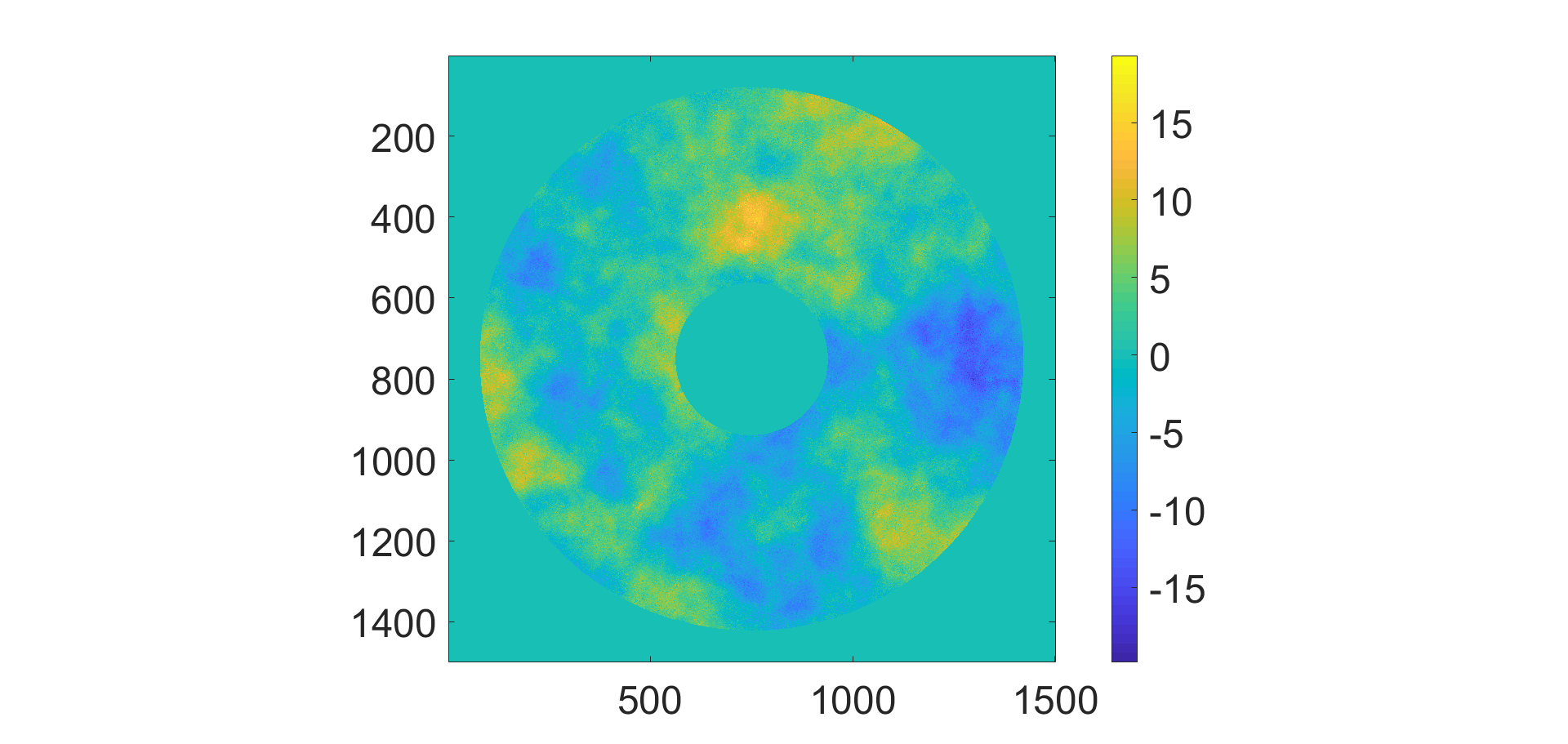} &
        \includegraphics[width=0.4\textwidth, trim = {12cm 0.5cm 11cm 1cm}, clip=true]{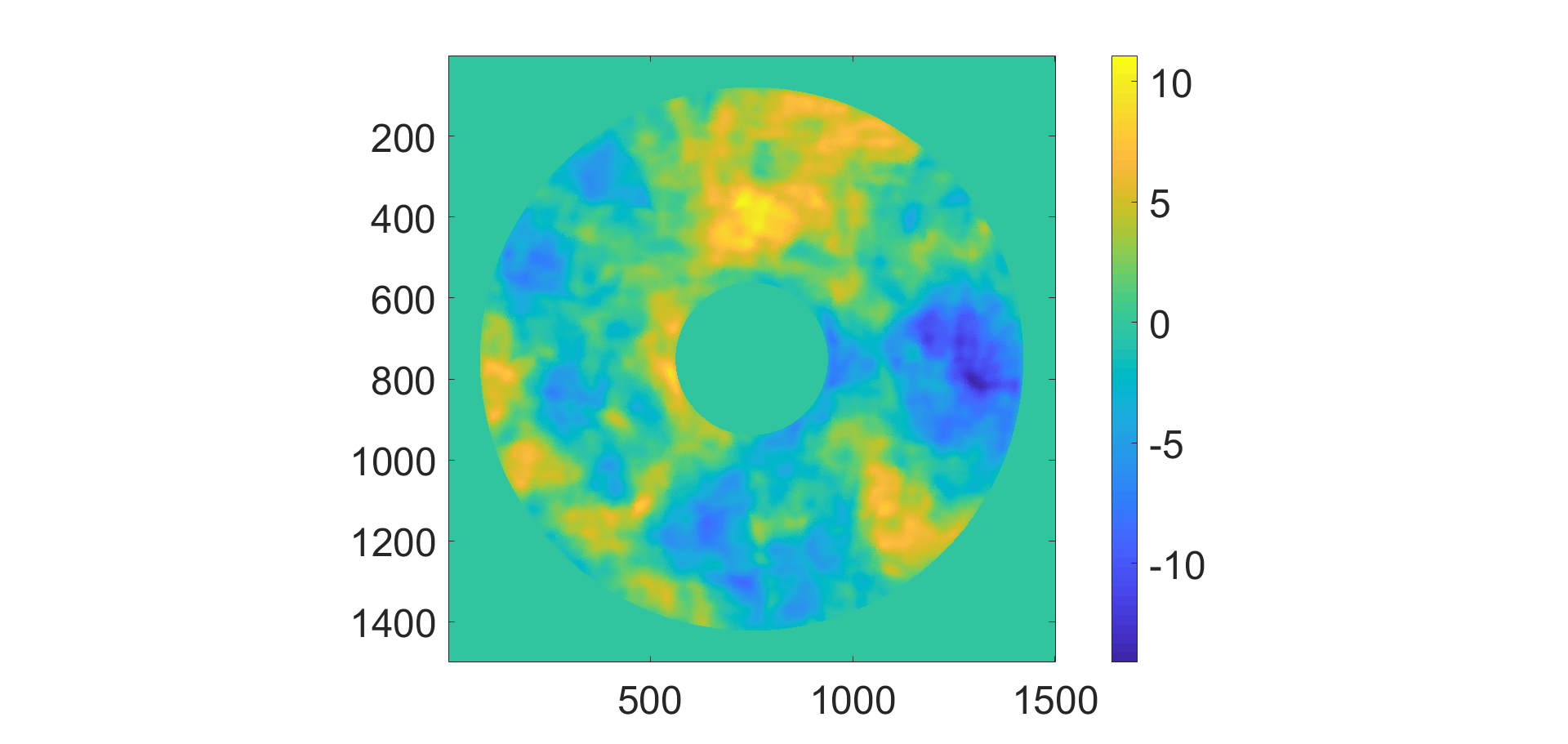} \\
        \includegraphics[width=0.4\textwidth, trim = {12cm 0.5cm 11cm 1cm}, clip=true]{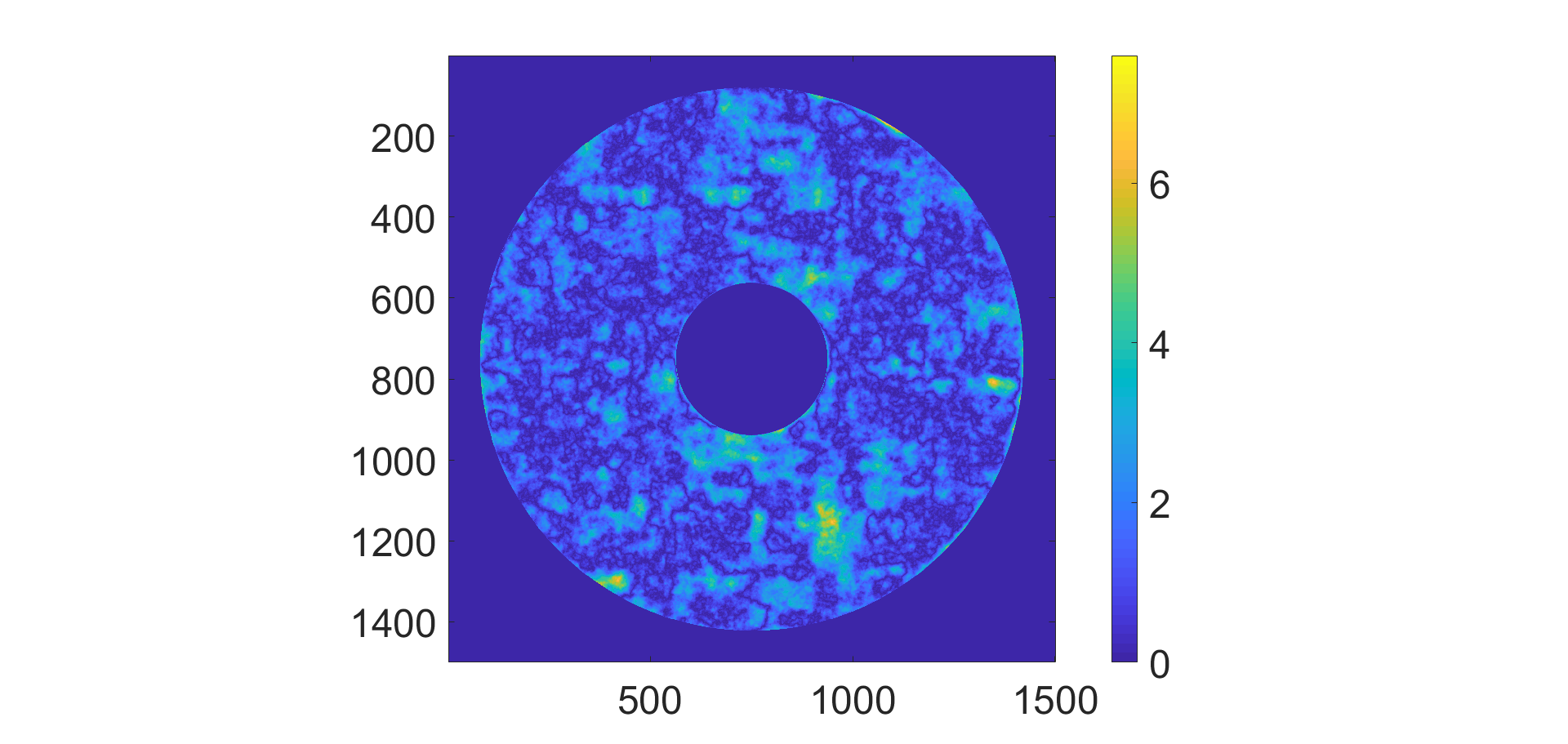} &
        \includegraphics[width=0.4\textwidth, trim = {12cm 0.5cm 11cm 1cm}, clip=true]{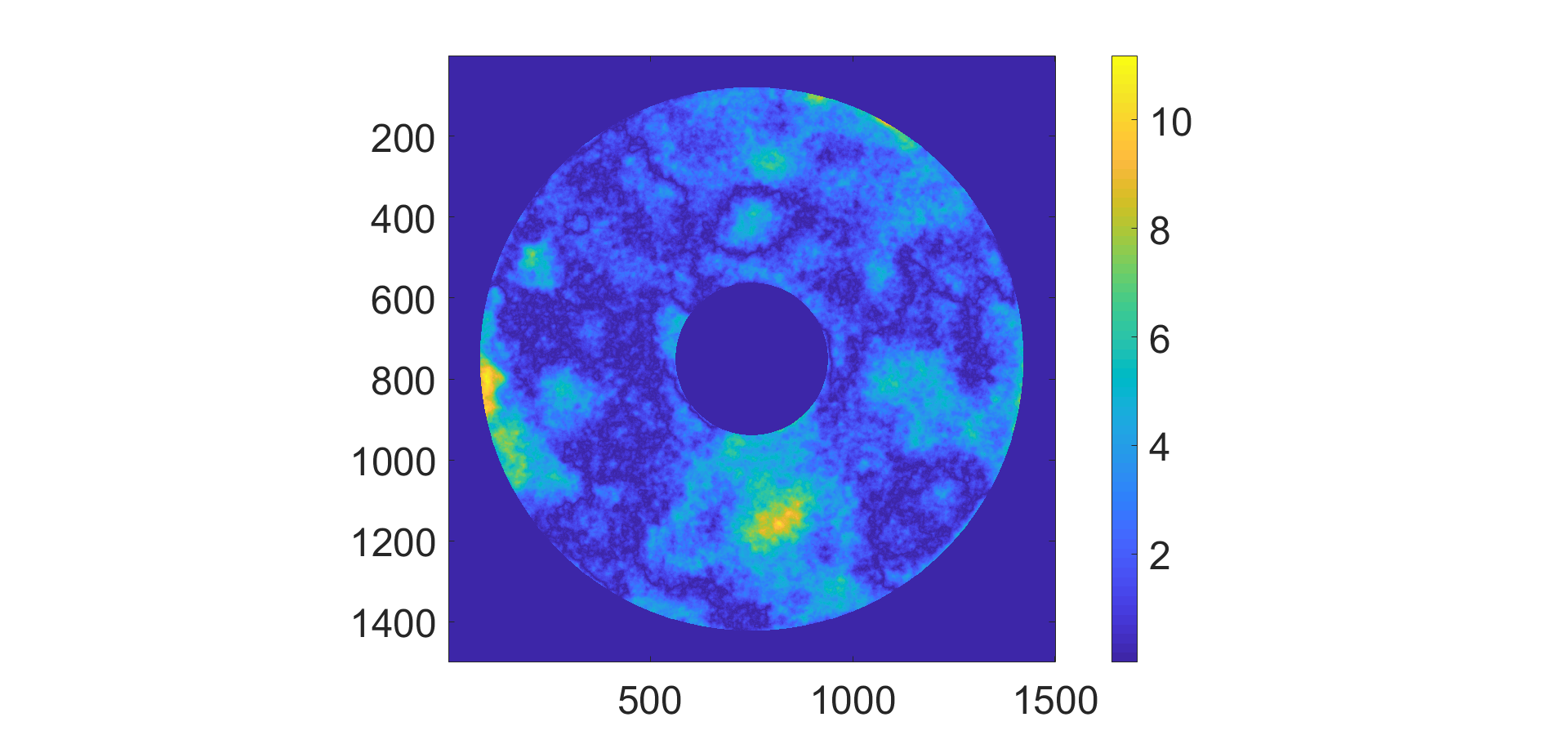} &
        \includegraphics[width=0.4\textwidth, trim = {12cm 0.5cm 11cm 1cm}, clip=true]{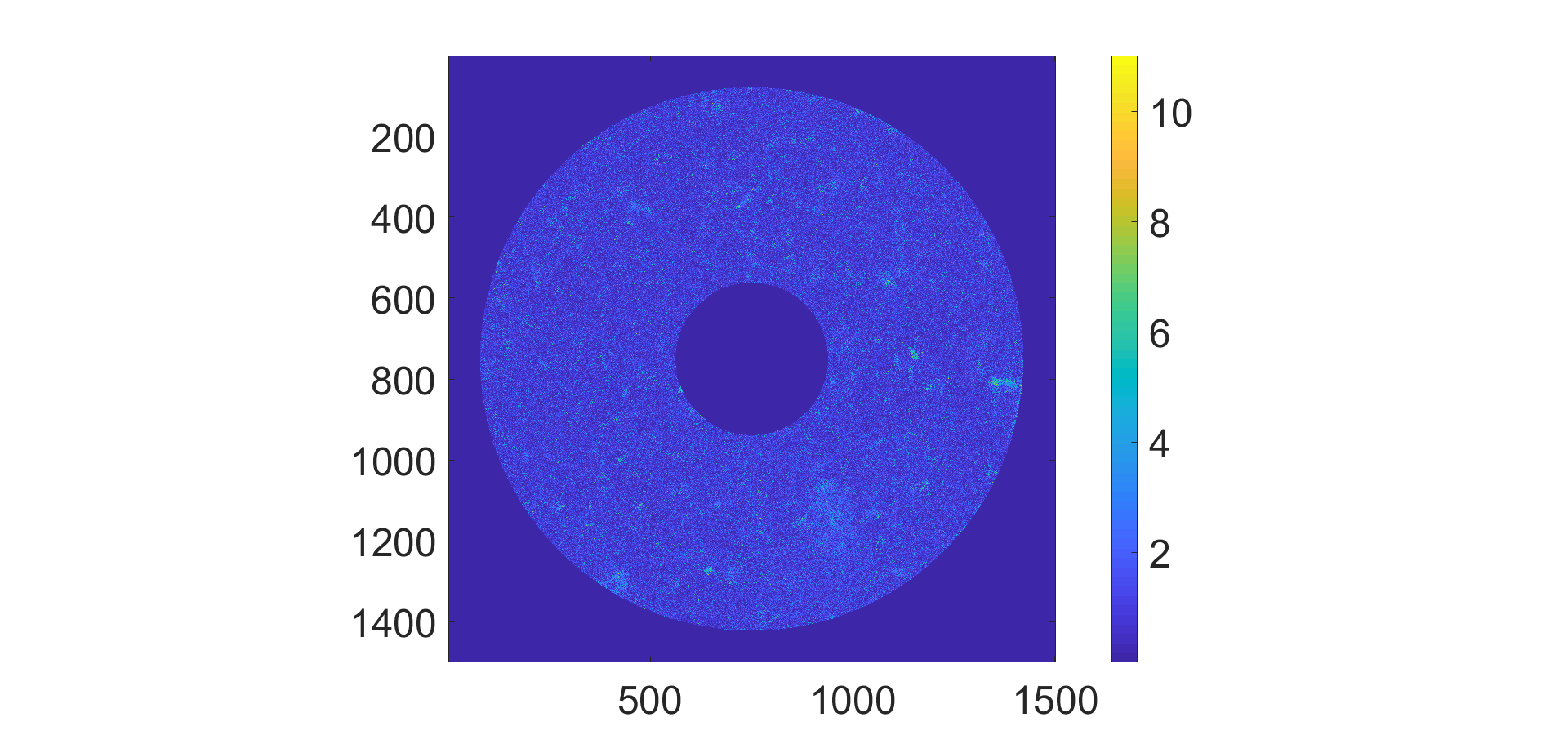} &
        \includegraphics[width=0.4\textwidth, trim = {12cm 0.5cm 11cm 1cm}, clip=true]{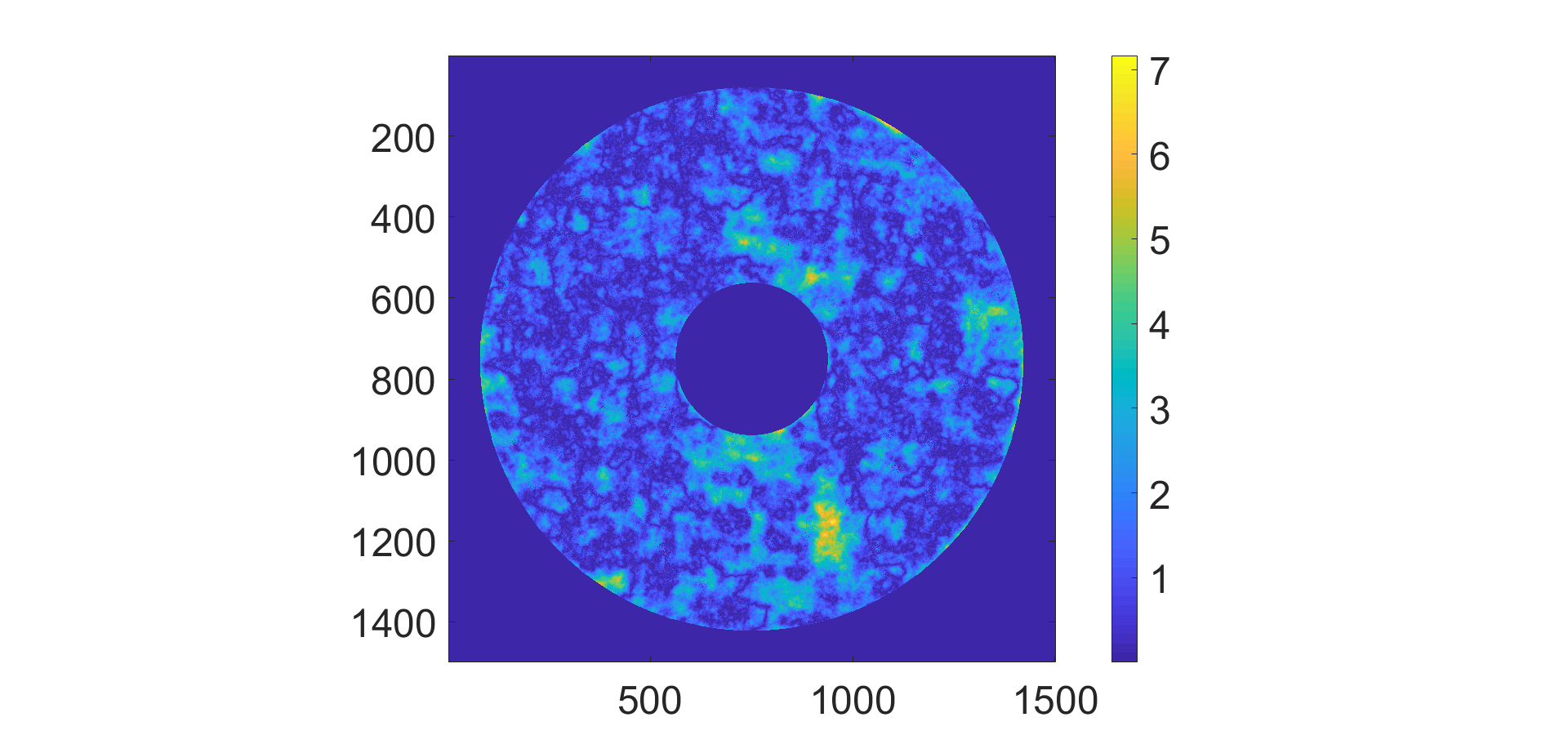} \\
    \end{tabular}
    }
    \caption{Unwrapped phases $\phir$ obtained via the different PWFS unwrapping methods of Section~\ref{subsect_digFWFS}, i.e., Algorithm~\ref{algo_dFb_WFS} \&~\ref{algo_dP_WFS}, applied to the noisy wrapped phase $\phiwd$ depicted in Figure~\ref{fig_OCTOPUS_data} (top), as well as the corresponding absolute difference to the non-wrapped phase $\phi$ (bottom). The results show Algorithm~\ref{algo_dP_WFS} (left) and Algorithm~\ref{algo_dFb_WFS} (three right columns) with zero or linear starting and smoothness index $s=11/6$ or $s=1/2$ respectively.}
    \label{fig_OCTOPUS_results_pyramid}
\end{figure}

Next, we consider the results obtained via the digital Fourier-type WFS unwrapping methods of Section~\ref{subsect_digFWFS}, in particular Algorithm~\ref{algo_dFb_WFS} and Algorithm~\ref{algo_dP_WFS}. We start with the classical 4-sided PWFS, cf.~Table~\ref{table_psi}. Figure~\ref{fig_OCTOPUS_results_pyramid} collects four different reconstructions: The left column was obtained using the linear PCuReD method, i.e. Algorithm~\ref{algo_dP_WFS}, while the other columns were obtained via the NOPE algorithm, i.e., Algorithm~\ref{algo_dFb_WFS}. As explained above, the smoothness index $s$ is an optional parameter of the NOPE, and the terms zero/linear starting correspond to different initial guesses for the iterative procedure underlying the NOPE \cite{HuNeuSha_2023}. In particular, the reconstruction shown in the left column was used as an initial guess for the NOPE reconstructions depicted in the right two columns. As can be seen from the residual plots, the best reconstruction is obtained when the NOPE is combined with linear starting and $s=11/6$, as depicted in the third column. Good results are also obtained with the linear Algorithm~\ref{algo_dP_WFS} and Algorithm~\ref{algo_dFb_WFS} with $s=1/2$, depicted in the first and fourth columns, respectively. Compared with the results shown in the second column, the benefit of linear starting over zero starting becomes apparent. Unfortunately, while the NOPE method (Algorithm~\ref{algo_dFb_WFS}) produces some of the best results overall, it is also computationally expensive, requiring several minutes of runtime compared to a couple of seconds for all other algorithms considered in this paper. As a remedy, the NOPE can be run in an abridged form using only a couple of iterations and linear starting. The obtained results are then very close to those shown in the third column, but with a significantly improved computation time.

\begin{figure}[ht!]
    \centering
    \small
\resizebox{\columnwidth}{!}{%
    \begin{tabular}{cccc}
        \textbf{\Large 3-sided} & \textbf{\Large Roof} & \textbf{\Large Cone} & \textbf{\Large iQuad} \\
        \includegraphics[width=0.4\textwidth, trim = {12cm 0.5cm 11cm 1cm}, clip=true]{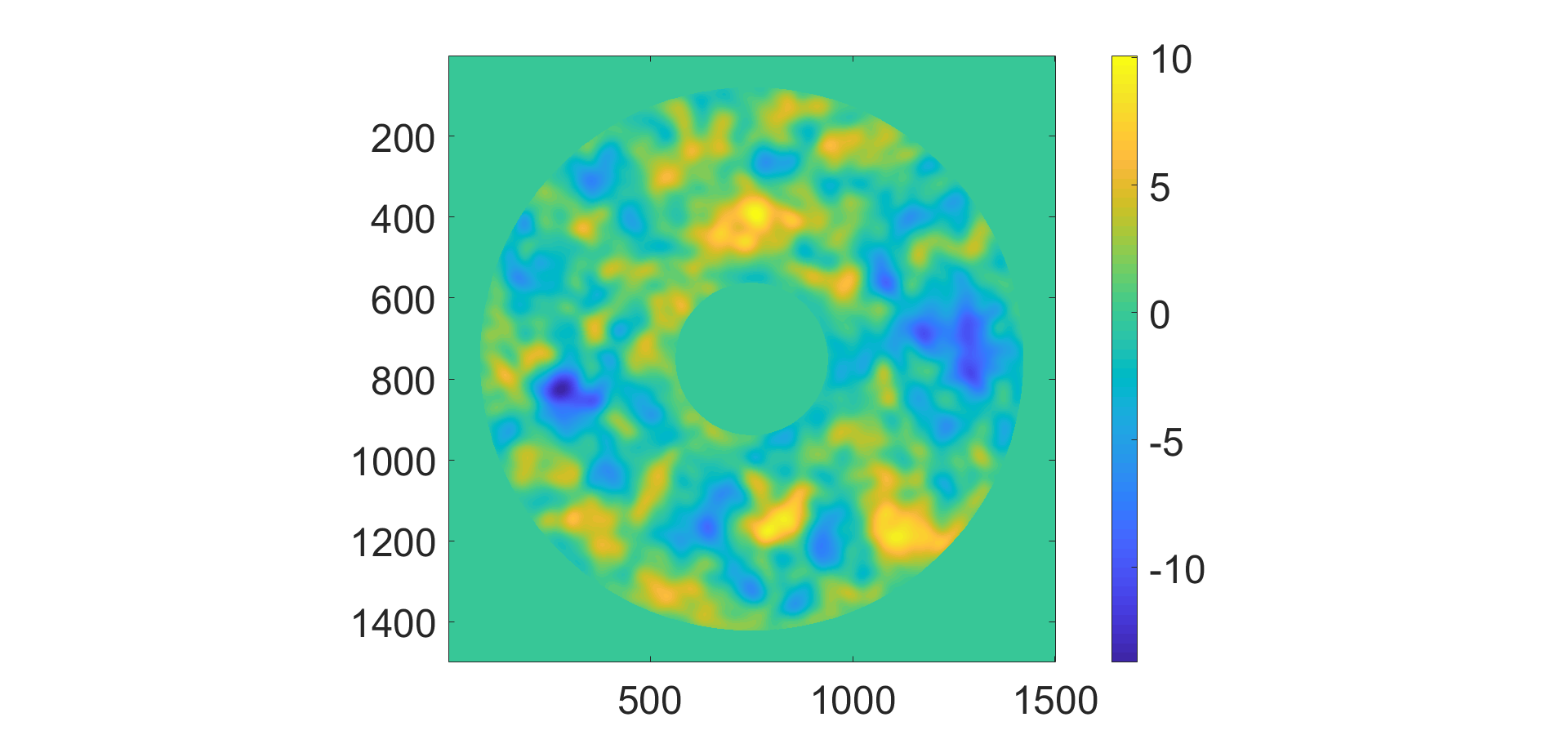} &
        \includegraphics[width=0.4\textwidth, trim = {12cm 0.5cm 11cm 1cm}, clip=true]{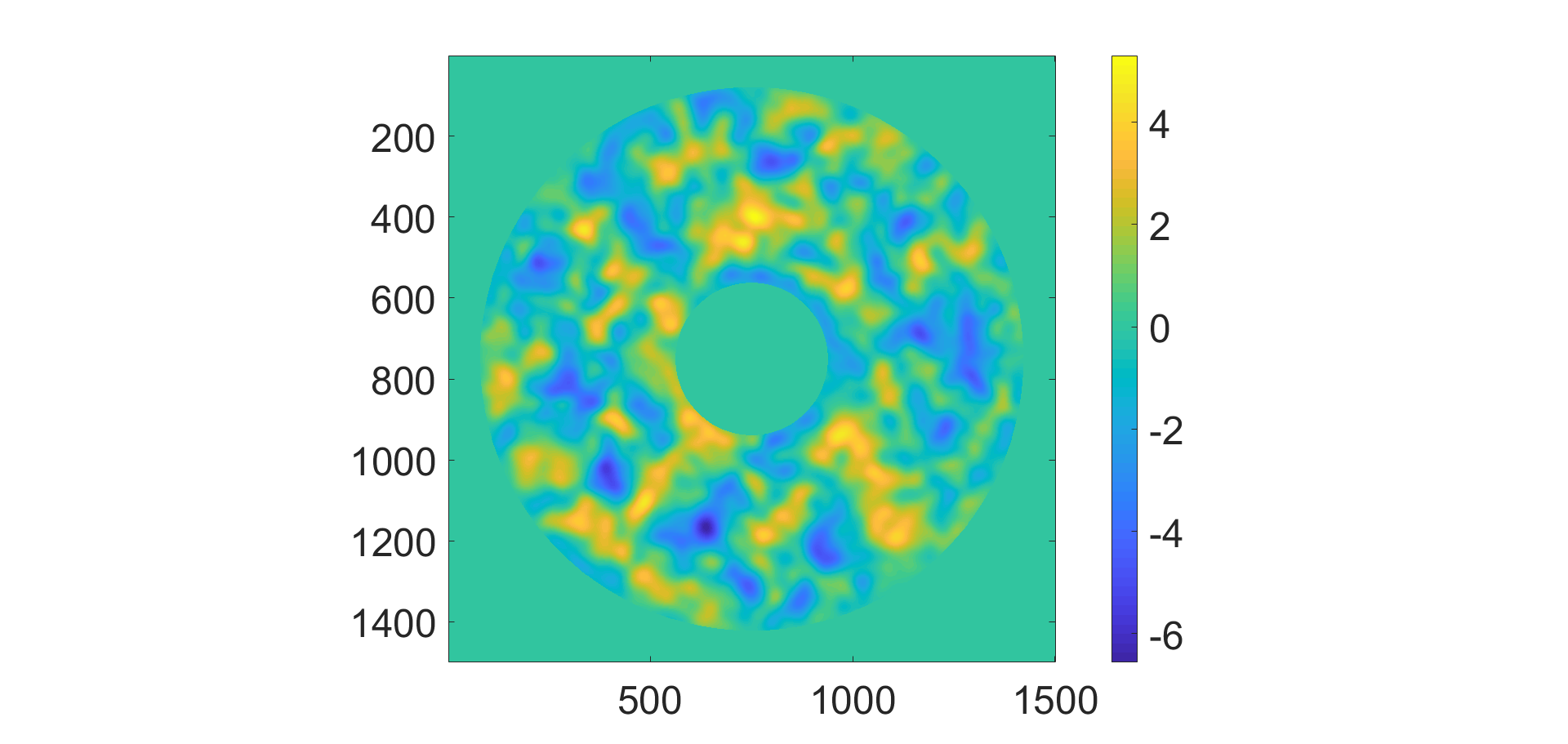} &
        \includegraphics[width=0.4\textwidth, trim = {12cm 0.5cm 11cm 1cm}, clip=true]{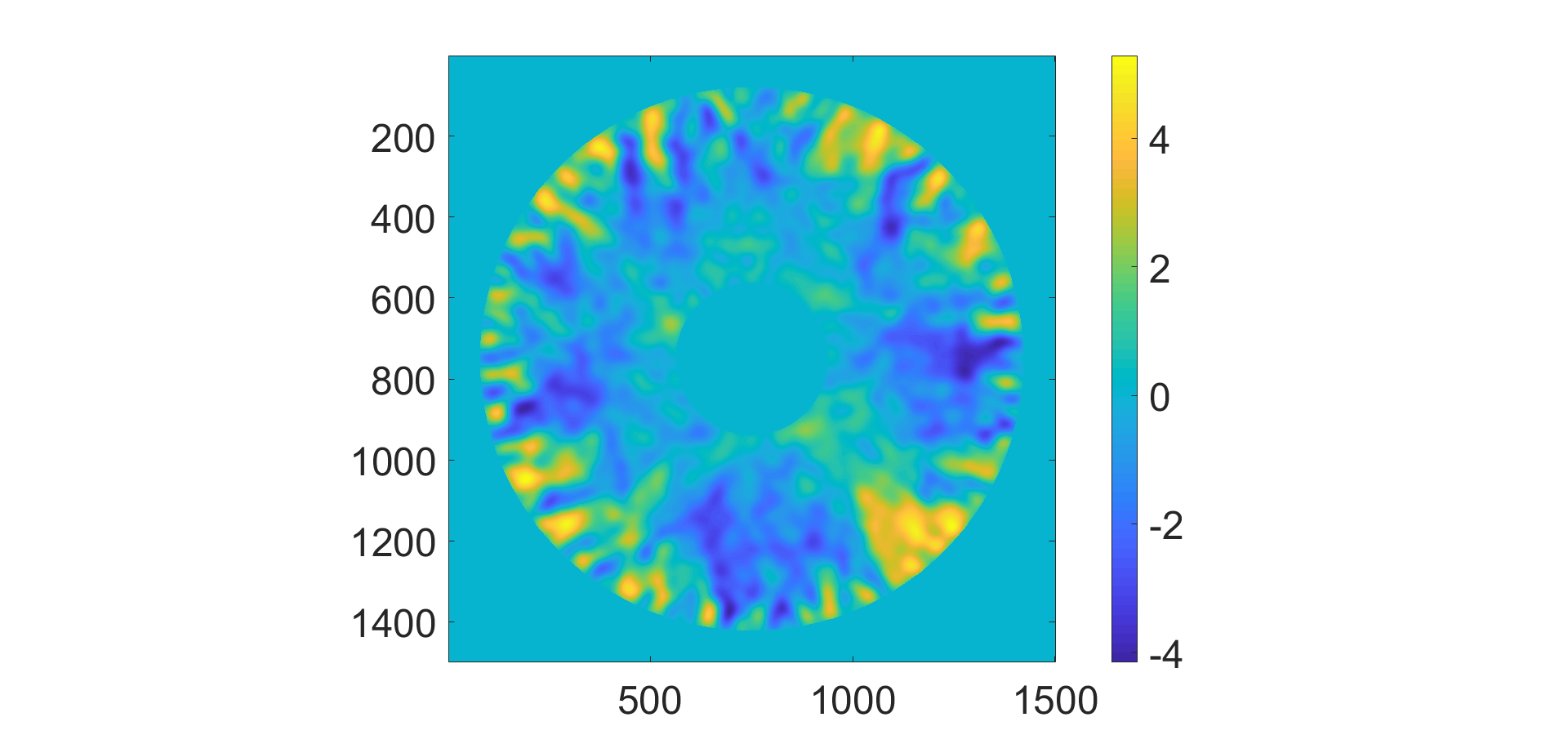} &
        \includegraphics[width=0.4\textwidth, trim = {12cm 0.5cm 11cm 1cm}, clip=true]{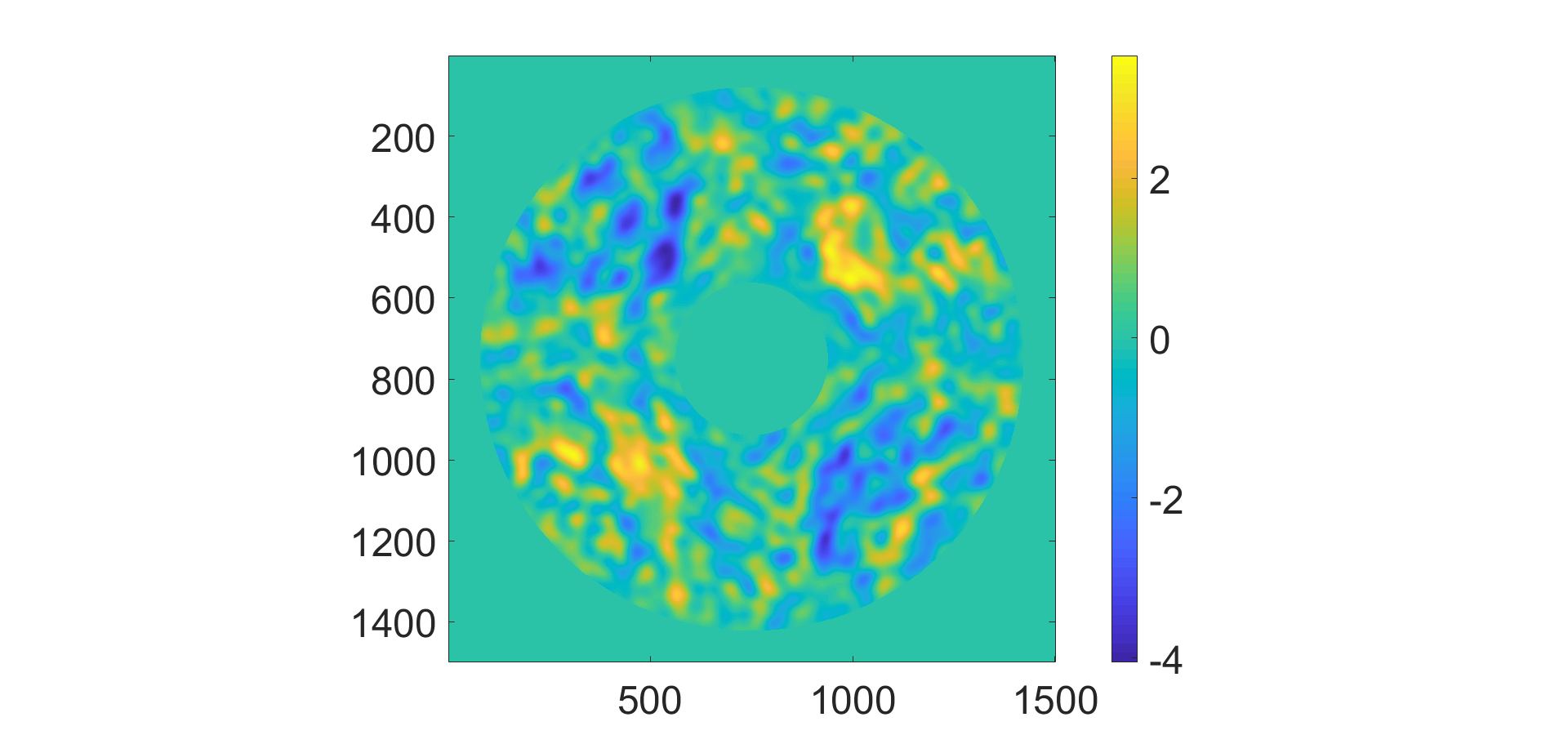} \\
        \includegraphics[width=0.4\textwidth, trim = {12cm 0.5cm 11cm 1cm}, clip=true]{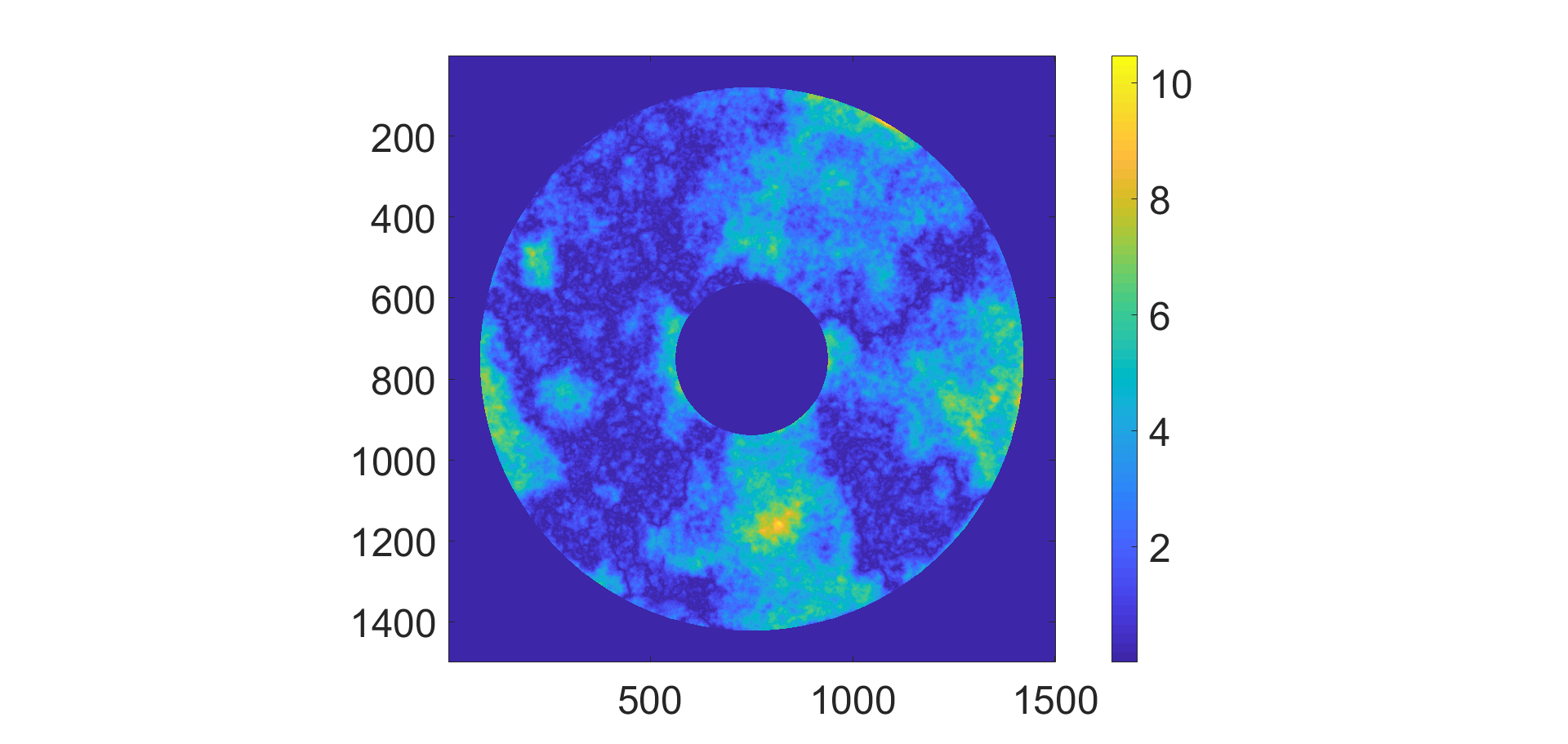} &
        \includegraphics[width=0.4\textwidth, trim = {12cm 0.5cm 11cm 1cm}, clip=true]{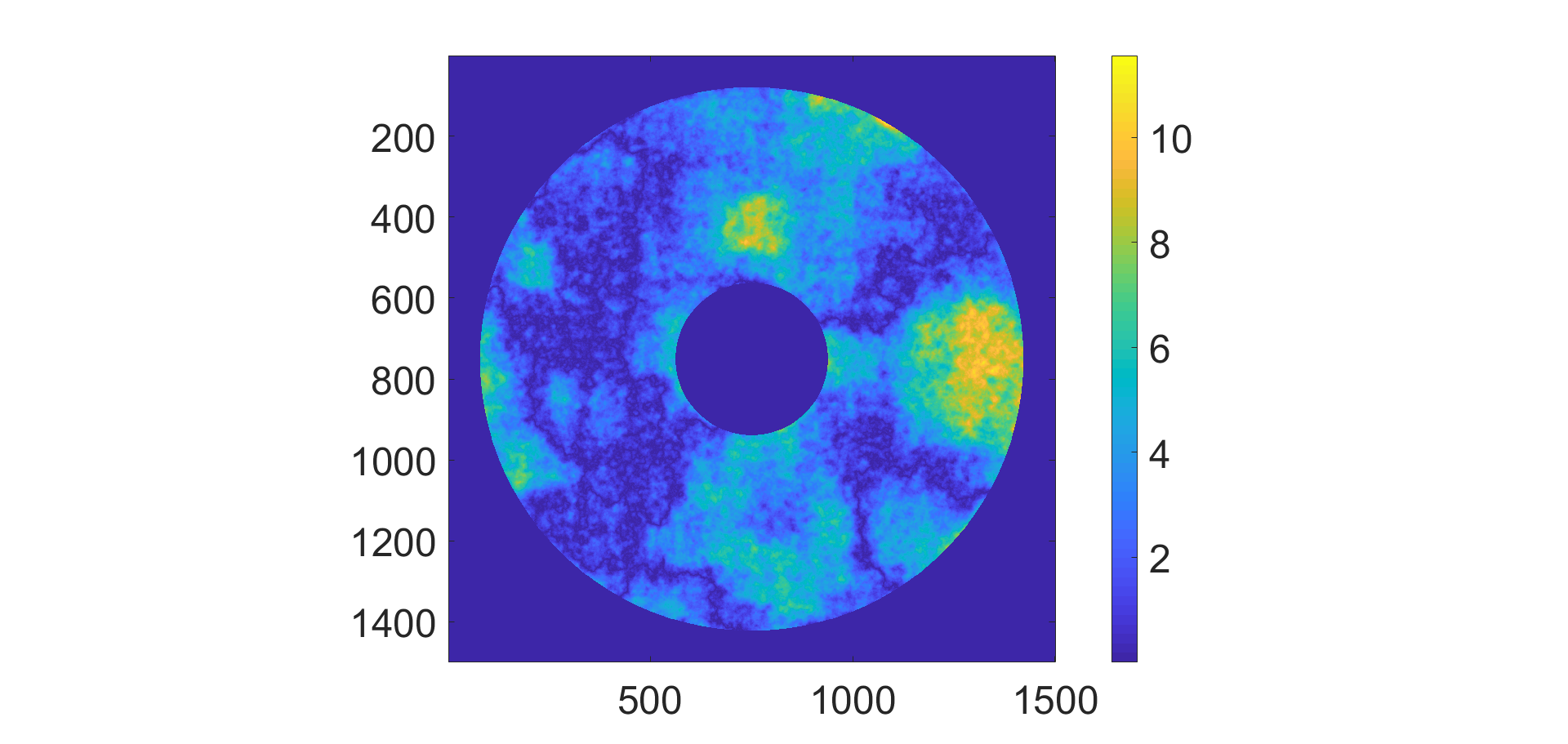} &
        \includegraphics[width=0.4\textwidth, trim = {12cm 0.5cm 11cm 1cm}, clip=true]{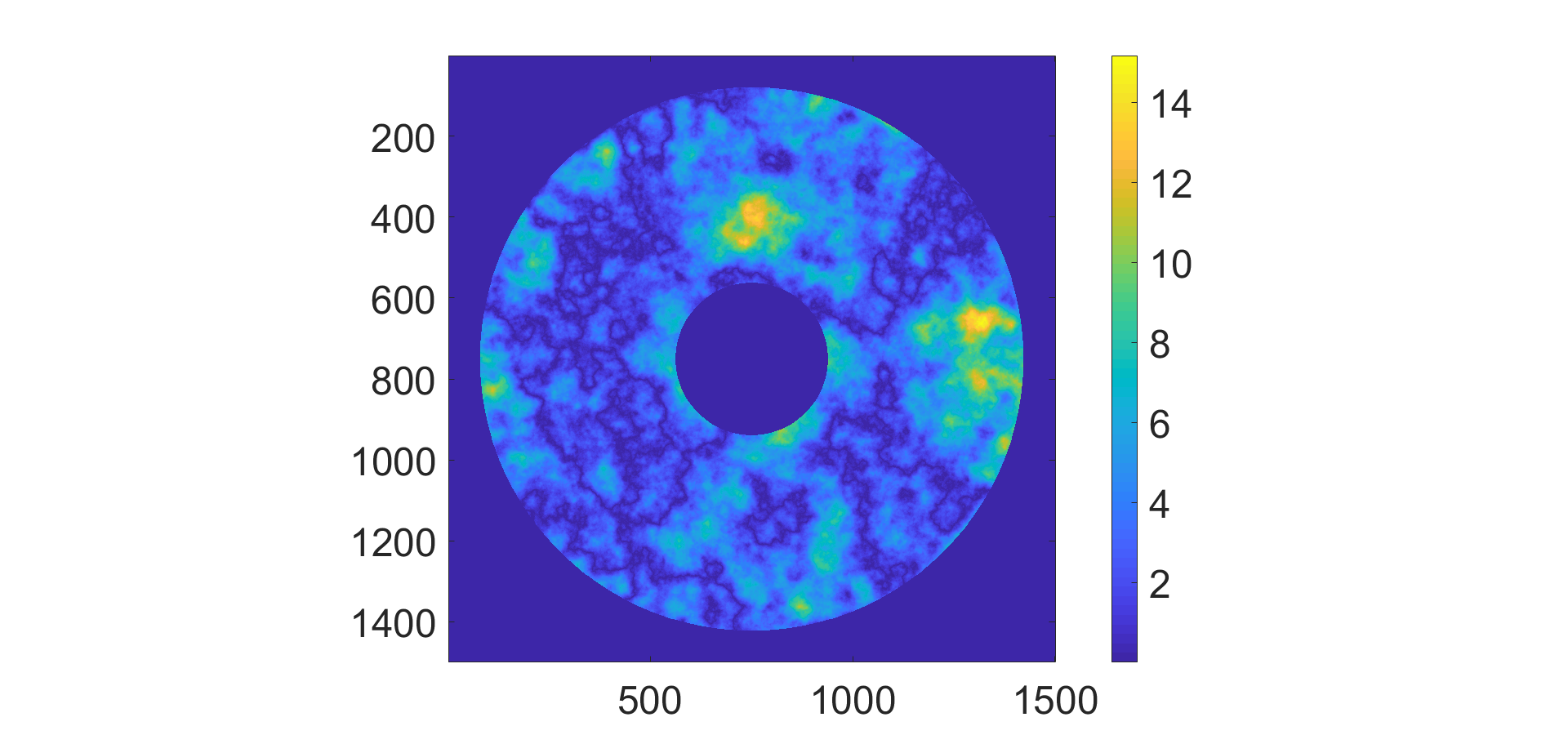} &
        \includegraphics[width=0.4\textwidth, trim = {12cm 0.5cm 11cm 1cm}, clip=true]{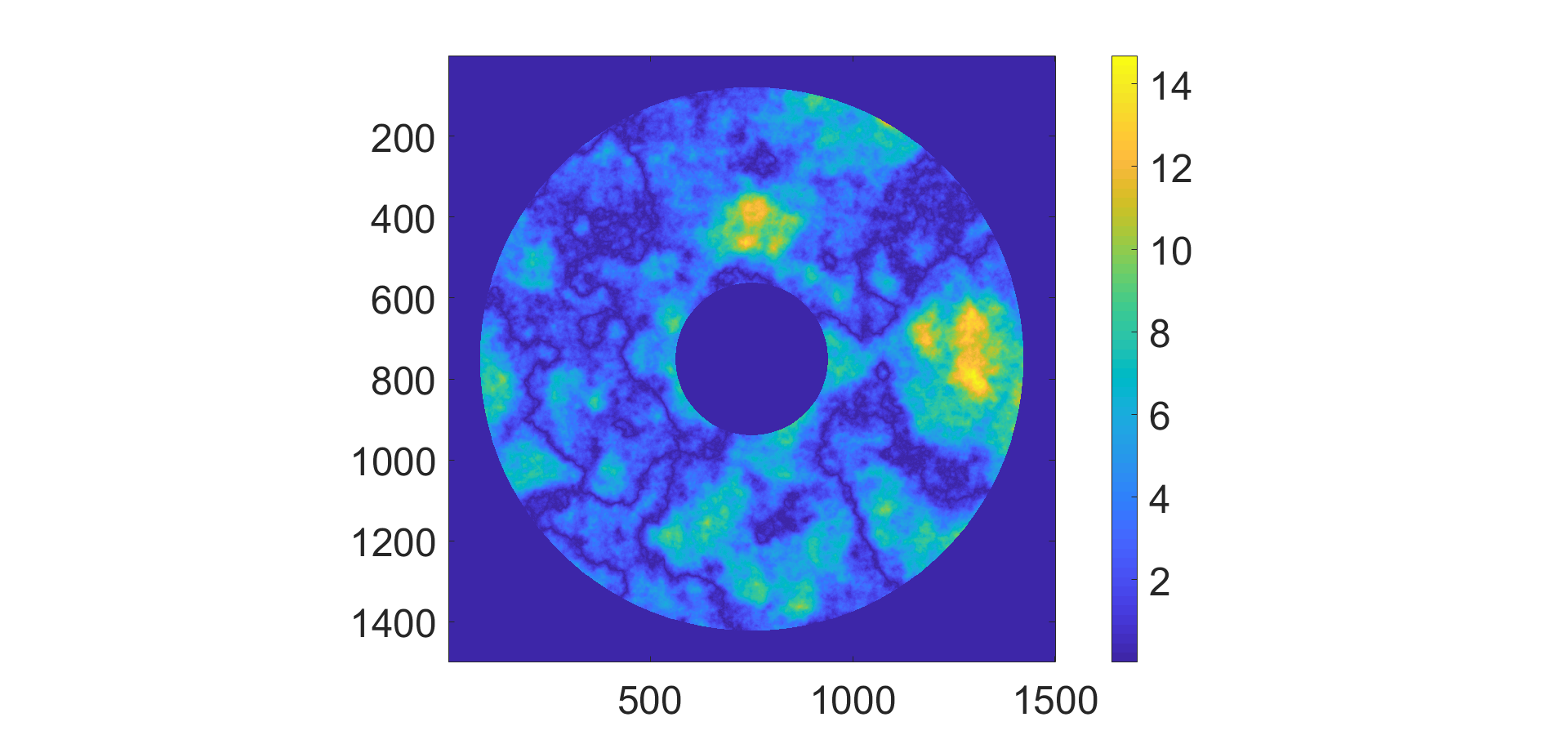} \\
    \end{tabular}
	}
    \caption{Unwrapped phases $\phir$ obtained via the different Fourier-type WFS unwrapping methods of Section~\ref{subsect_digFWFS}, i.e., Algorithm~\ref{algo_dFb_WFS} with the different functions $\psi$ listed in Table~\ref{table_psi}, applied to the noisy wrapped phase $\phiwd$ depicted in Figure~\ref{fig_OCTOPUS_data} (top), as well as the corresponding absolute difference to the non-wrapped phase $\phi$ (bottom).}
    \label{fig_OCTOPUS_results_pyramid_others}
\end{figure}

Next, we consider the effect of different shape functions $\psi$ on the performance of Algorithm~\ref{algo_dFb_WFS}. In particular, we consider $\psi$ corresponding to the 3-sided pyramid, roof, cone, and iQuad wavefront sensors as given in Table~\ref{table_psi}. The obtained unwrapped phases $\phir$ and the corresponding reconstruction errors are depicted in Figure~\ref{fig_OCTOPUS_results_pyramid_others}. Note that here ``Roof'' refers to the averaged sum of the corresponding x-Roof and y-Roof reconstructions. Among the different sensors, only the 3-sided PWFS, and to some extent also the roof WFS, still provide acceptable reconstructions. For the cone sensor, only traces of the structure of the true phase are visible, hidden behind severe artefacts, which for the iQuad sensor dominate the reconstruction entirely. These artefacts may be explained with non-uniqueness issues that have already been observed in \cite{FaHuShaRaLAM19_AO4ELTproc} for the iQuad sensor, and might be removed with further developments of the NOPE algorithm.

\begin{figure}[ht!]
\centering
\small
\resizebox{\columnwidth}{!}{%
    \begin{tabular}{cccc}
        \textbf{\Large MATLAB} & \textbf{\Large MRP} & \textbf{\Large PE} &  \textbf{\Large TIE}\\
        \includegraphics[width=0.4\textwidth,trim = {12cm 0.5cm 11cm 1cm}, clip=true]{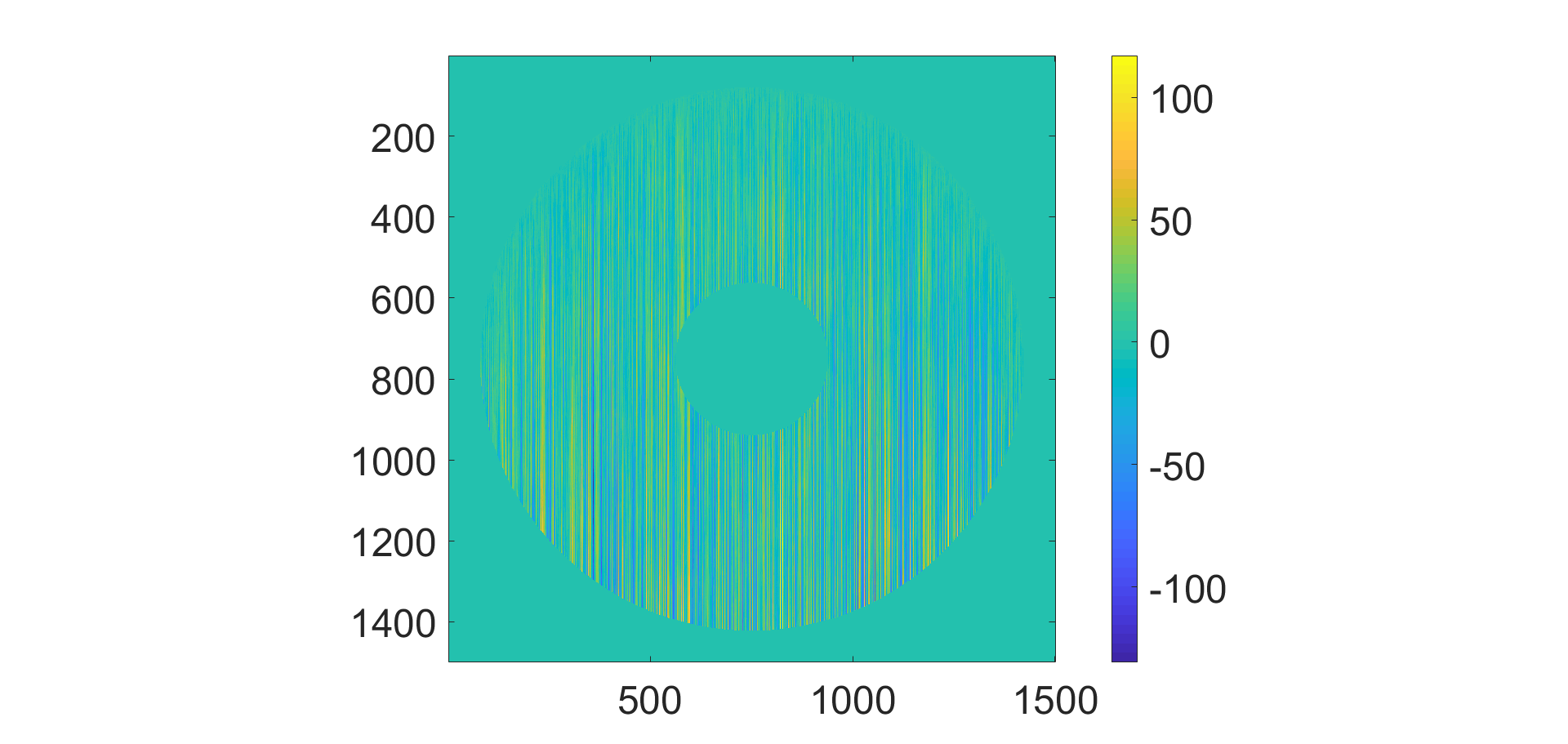} & 
        \includegraphics[width=0.4\textwidth,trim = {12cm 0.5cm 11cm 1cm}, clip=true]{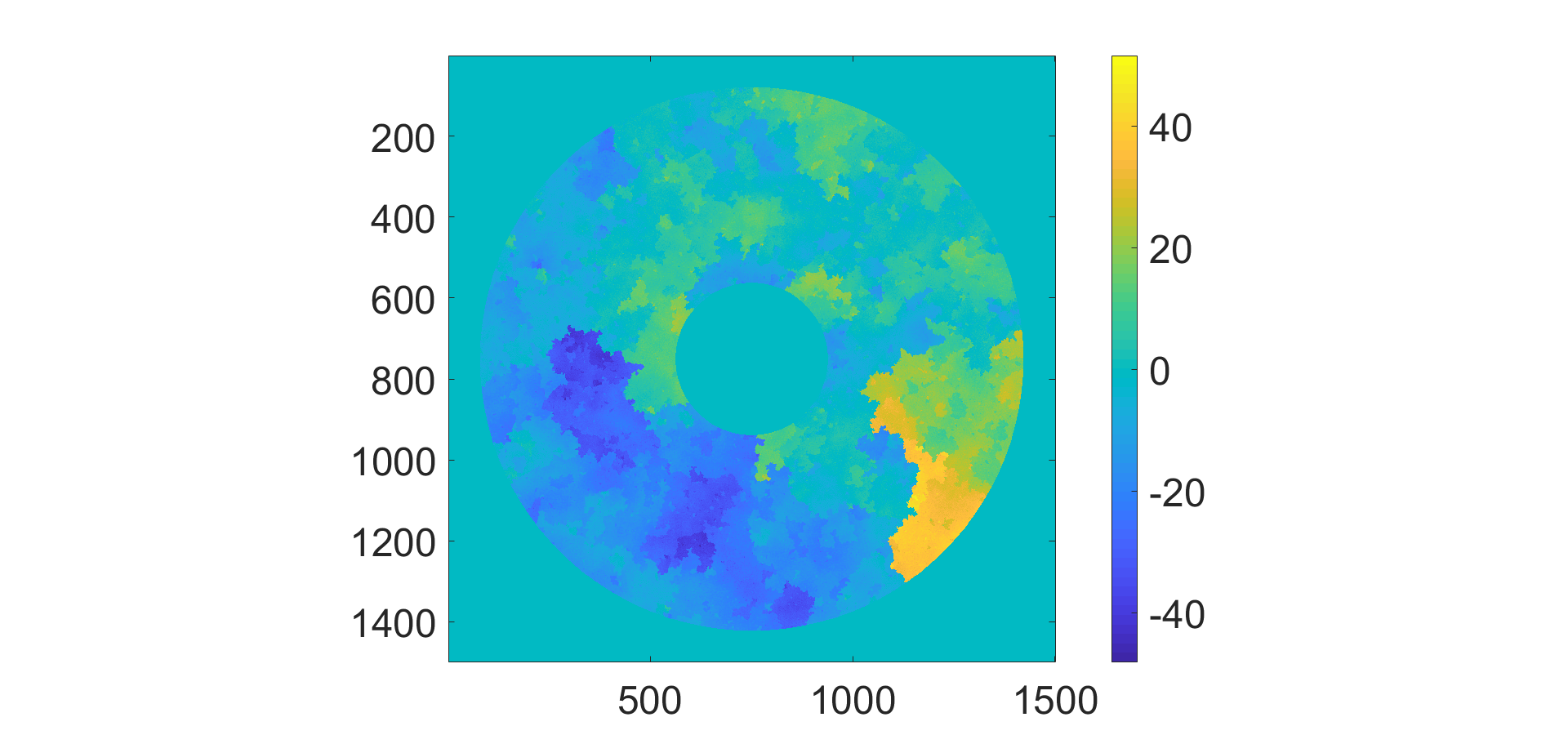} & 
        \includegraphics[width=0.4\textwidth,trim = {12cm 0.5cm 11cm 1cm}, clip=true]{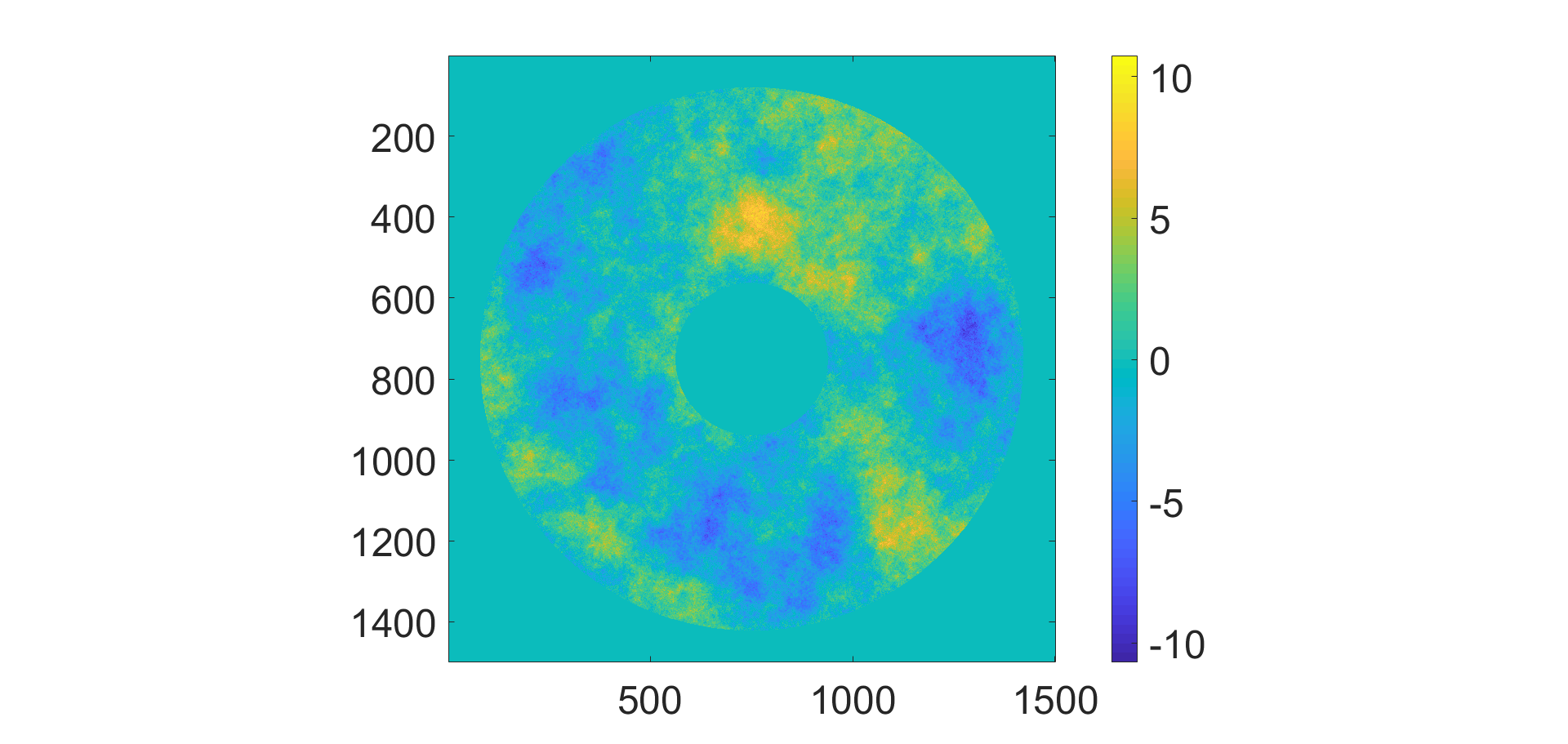} & 
        \includegraphics[width=0.4\textwidth,trim = {12cm 0.5cm 11cm 1cm}, clip=true]{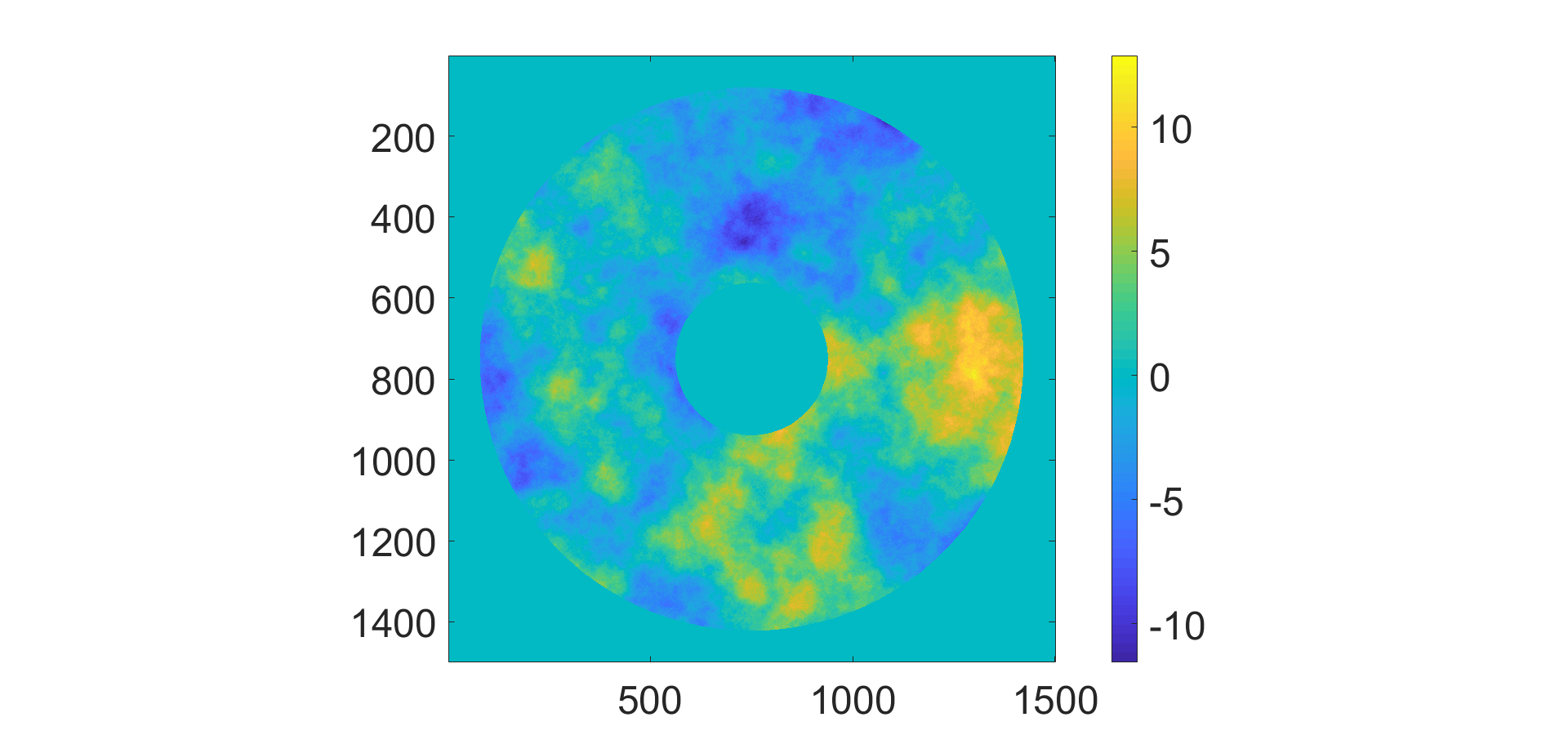}\\
        \includegraphics[width=0.4\textwidth,trim = {12cm 0.5cm 11cm 1cm}, clip=true]{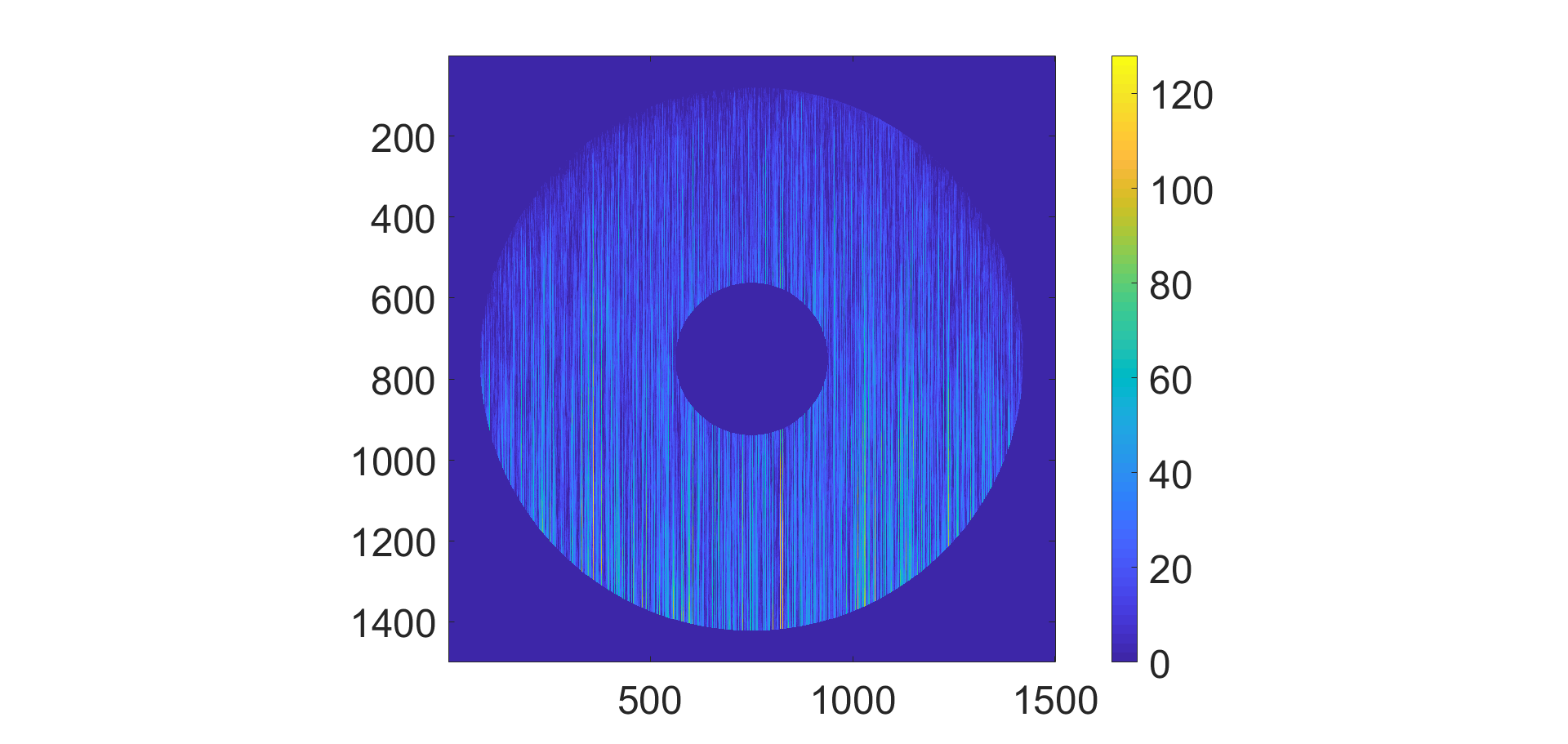} & 
        \includegraphics[width=0.4\textwidth,trim = {12cm 0.5cm 11cm 1cm}, clip=true]{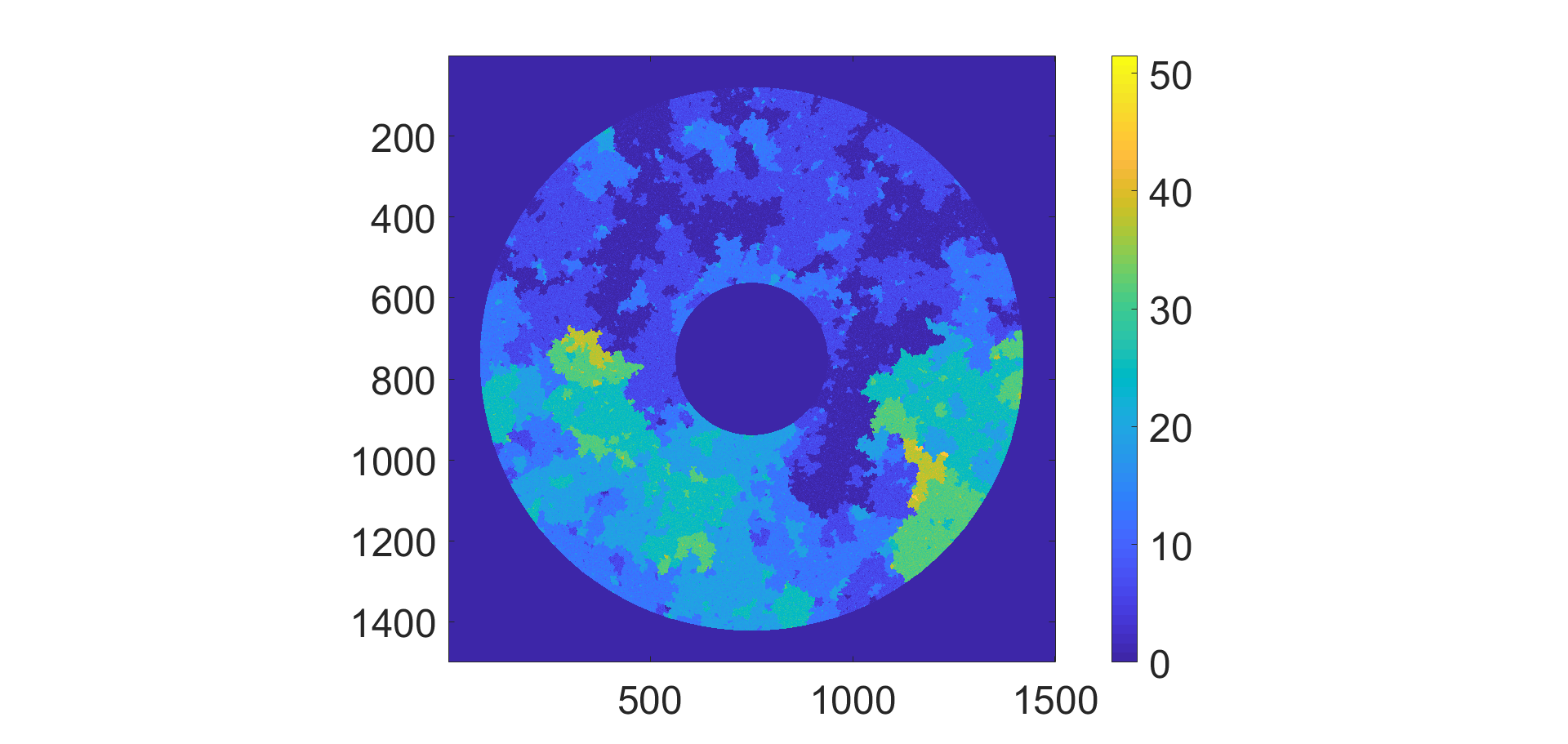} & 
        \includegraphics[width=0.4\textwidth,trim = {12cm 0.5cm 11cm 1cm}, clip=true]{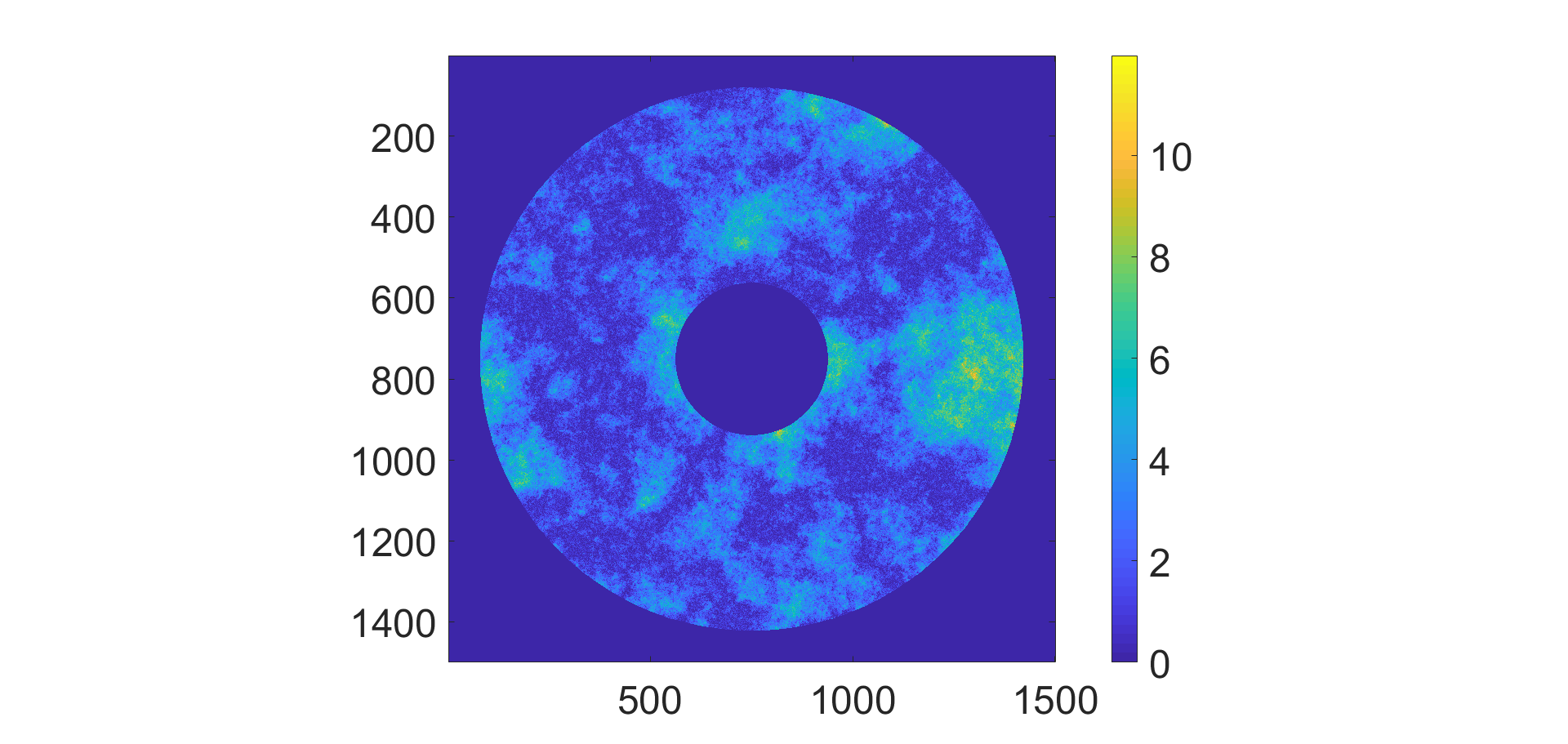} &
        \includegraphics[width=0.4\textwidth,trim = {12cm 0.5cm 11cm 1cm}, clip=true]{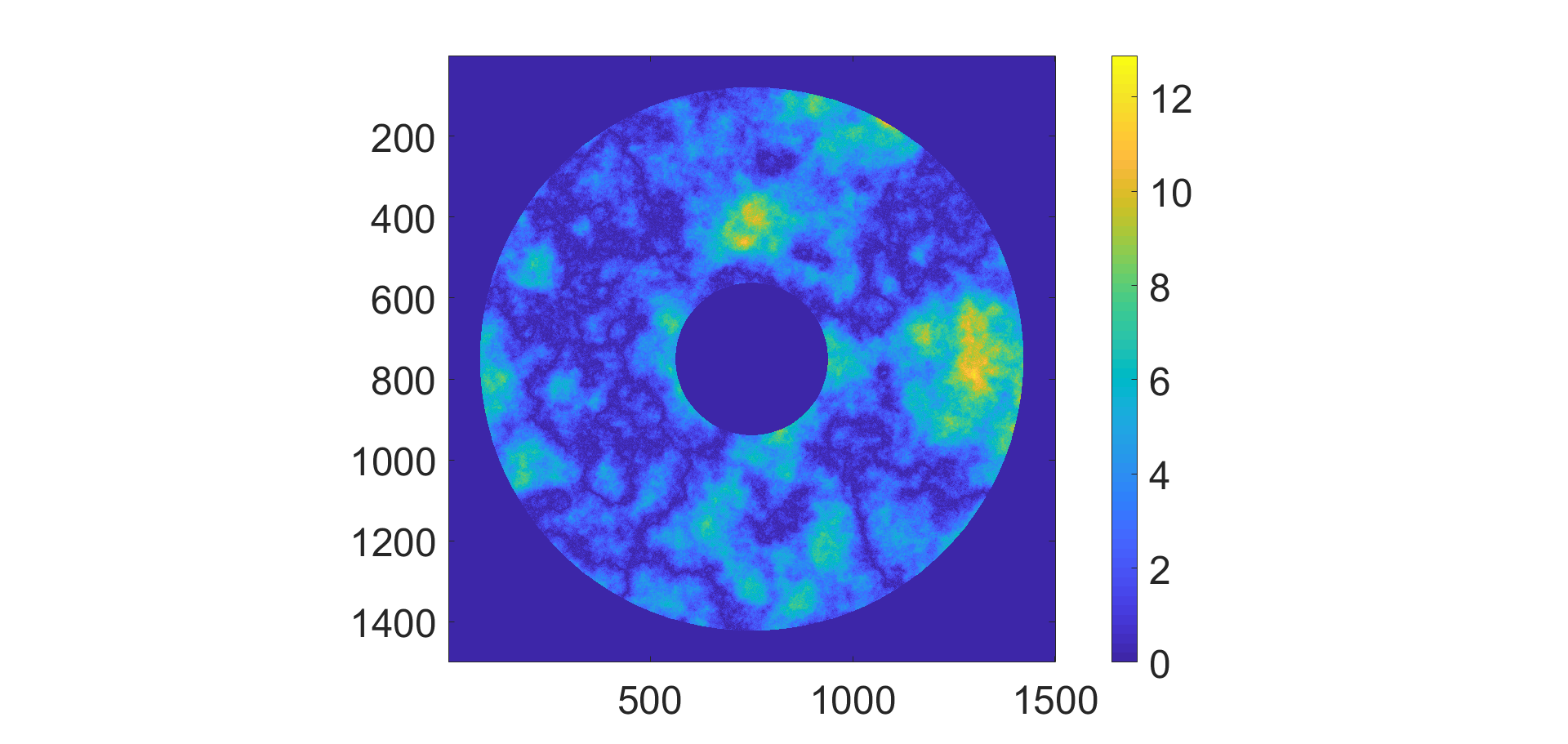}
    \end{tabular}
    }
    \caption{Unwrapped phases $\phir$ obtained via different phase unwrapping approaches (top) applied to the noisy wrapped phase $\phiwd$ depicted in Figure~\ref{fig_OCTOPUS_data}, and the corresponding absolute difference to the non-wrapped phase $\phi$ (bottom). From left to right: Matlab built-in function, unwrapping based on most reliable pixels (MRP) \cite{Herraez_2002,Firman_2023}, the Poisson equation (PE)\cite{Zhao_2020,Zhao_2023_1} and on the transport of intensity equation (TIE) \cite{Zhao_2019,Zhao_2023_2}.}
    \label{fig_OCTOPUS_results_Standard}
\end{figure}

Finally, we compare the results of our proposed digital WFS unwrapping approaches with those obtained by four different state-of-the art methods introduced above. The corresponding results are depicted in Figure~\ref{fig_OCTOPUS_results_Standard}. First, note that the built-in Matlab unwrapping function works only into one dimension. Depending on the given data, it may be sufficient to perform unwrapping in the two dimensions consecutively. However, in our case this entirely failed. The MRP method produces discontinuous, patch-like reconstructions which are also far off from the true phase. The best reconstruction is the one obtained via the PE method, which is similar to, but less detailed than that of the digital SH-WFS unwrapping. The TIE unwrapping captures many of the structural features of the ground truth phase, but some features almost appear to be inverted with respect to the ground truth. In summary, we find that our proposed digital SH-WFS and Fourier-type WFS methods perform at least as good as, and often much better than, these four standard approaches. Table~\ref{table_errortable} quantifies the reconstruction quality in terms of the relative error, structural similarity index (SSIM) \cite{Wang_Bovik_Sheikh_Simoncelli_2004} and multiscale SSIM (MS-SSIM) \cite{Wang_Simoncelli_Bovik_2003} for each of the considered approaches. Overall, we find that the digital SH-WFS with a suitable number of subapertures, and the digital 4-sided PWFS using the NOPE algorithm with linear starting produce the best results among all methods. 

\begin{table}[ht!]
    \centering
    \scalebox{0.8}{
    \begin{tabular}{|c|c|c|c|c|c|c|c|}
    \cline {2-8}
    \multicolumn{1}{c|}{} & \multicolumn{3}{c|}{Digital Shack-Hartmann WFS} &   \multicolumn{4}{c|}{Digital 4-sided Pyramid WFS} \\ \hline
    Criterion & $30 \times 30$ & $100\times 100$ & $500 \times 500$ & PCuReD & NOPE & NOPE+Lin. & NOPE+Lin.($s=0.5$)  \\ \hline \hline
    Rel.Error & 41.51\% & 41.38\% & 193.87\% & 27.92\% & 67.37\% & \textbf{11.19\%} & 30.8\%   \\ \hline 
    SSIM & 0.5208 & \textbf{0.6119} & 0.4103 & 0.5554 & 0.5224 & 0.4527 & 0.5236 \\ \hline
    MS-SSIM & 0.6576 & \textbf{0.8154} & 0.3654 & 0.7109 & 0.7168 & 0.7800 & 0.7120 \\ \hline
    \cline {2-8} \cline {2-8}
    \multicolumn{1}{c|}{} & \multicolumn{3}{c|}{Figure~\ref{fig_OCTOPUS_results_SH}} &  \multicolumn{4}{c|}{Figure~\ref{fig_OCTOPUS_results_pyramid}}  \\ \cline {2-8} 
    \end{tabular}}
    \\ \vspace{15pt} 
    \scalebox{0.8}{
    \begin{tabular}{|c|c|c|c|c|c|c|c|c|}
    \cline {2-9}
    \multicolumn{1}{c|}{} & \multicolumn{4}{c|}{Digital Fourier-type WFS} & \multicolumn{4}{c|}{Other Algorithms}  \\ \hline
    Criterion & 3-sided Pyramid & Roof & Cone & iQuad & MATLAB & MRP & PE & TIE \\ \hline \hline
    Rel.Error & 73.27\% & 86.18\% & 89.96\% & 98.75\% & 724.96\% & 417.43\% & 66.94\% & 81.17\% \\ \hline 
    SSIM & 0.5119 & 0.4629 & 0.4233 & 0.4121 & 0.4080 & 0.4171 & 0.4438 & 0.4431 \\ \hline
    MS-SSIM & 0.7018 & 0.6439 & 0.4583 & 0.4671 & 0.3771 & 0.5101 & 0.6376 & 0.5885 \\ \hline
    \cline {2-9} \cline {2-9}
    \multicolumn{1}{c|}{} & \multicolumn{4}{c|}{Figure~\ref{fig_OCTOPUS_results_pyramid_others}} & \multicolumn{4}{c|}{Figure~\ref{fig_OCTOPUS_results_Standard}}  \\ \cline {2-9} 
    \end{tabular}}
    \caption{Comparison of the relative error, structural similarity index (SSIM) \cite{Wang_Bovik_Sheikh_Simoncelli_2004} and multiscale SSIM (MS-SSIM) \cite{Wang_Simoncelli_Bovik_2003} for each of the considered unwrapping approaches.}
    \label{table_errortable}
\end{table}

% Numerical experiment 2: Unwrapping without ground truth
\subsubsection*{Numerical experiment 2: Unwrapping without ground truth}

\begin{figure}[ht!]
    \centering
    \includegraphics[width=0.35\textwidth, trim = {13cm 0.5cm 12cm 1cm}, clip=true]{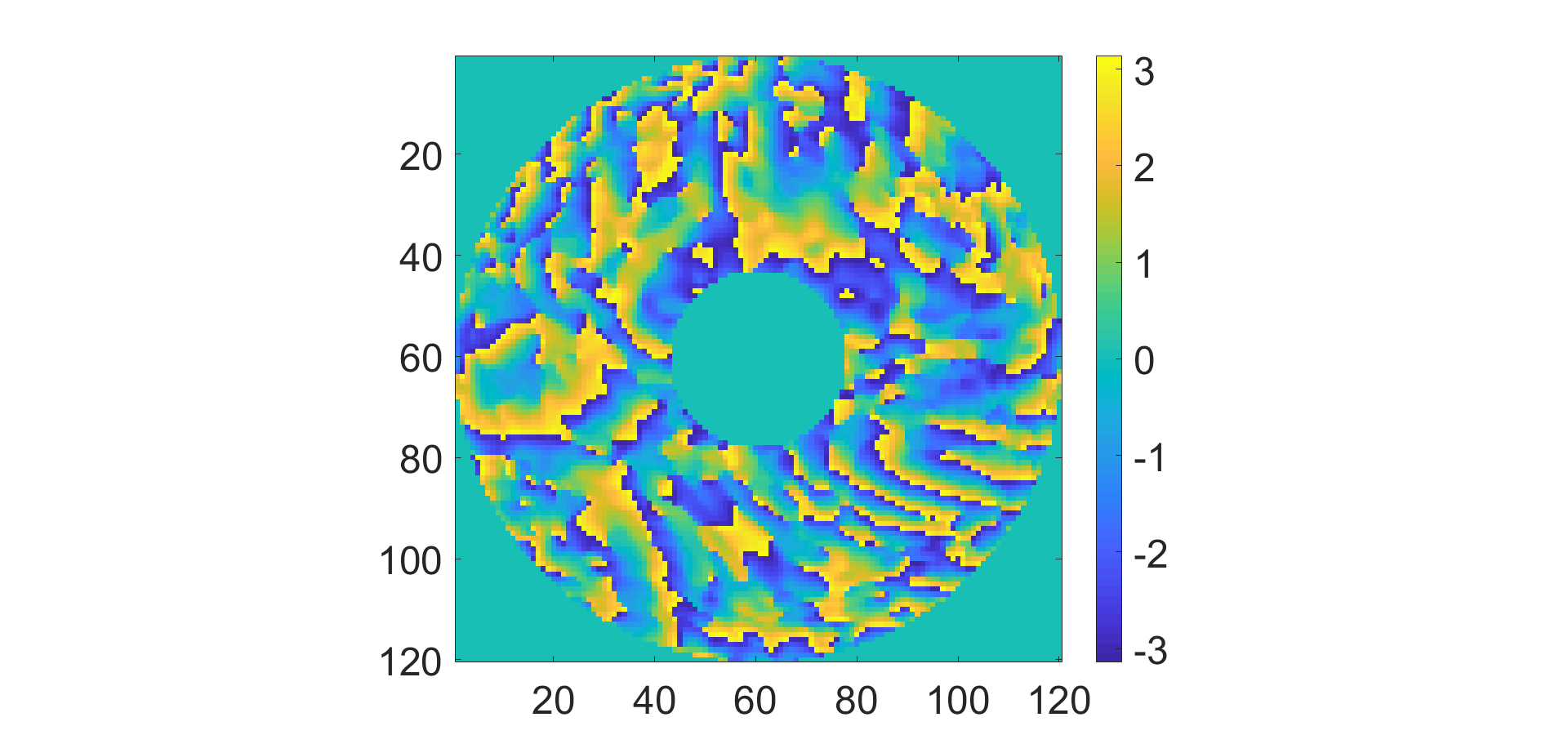}
    \quad
    \includegraphics[width=0.35\textwidth, trim = {13cm 0.5cm 12cm 1cm}, clip=true]{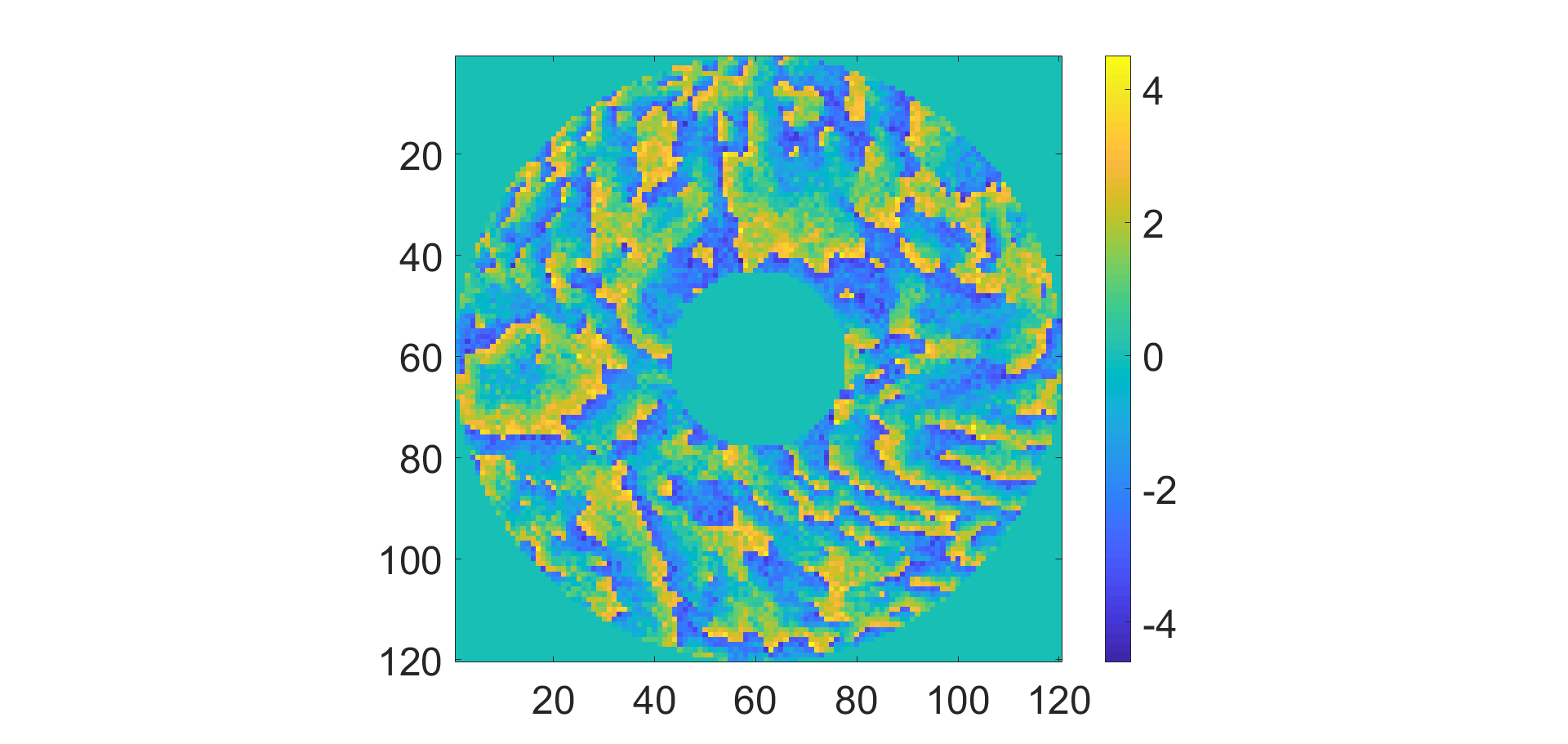}
    \caption{Realistic non wrapped phase $\phiw$ (left) occurring during an ALASCA daytime simulation and noisy wrapped phase $\phiwd$ including $20\%$ noise (right). }
    \label{fig_ALASCA_data}
\end{figure}

For the second example, we consider a phase unwrapping problem of FSOC, in particular from the Advanced Laser Guide Star Adaptive Optics for Satellite Communication Assessments (ALASCA) project \cite{ALASCA} currently under development, for which we now provide some brief background. Within the ALASCA project, an optical communication channel from the optical ground station of the European Space Agency (ESA) in Tenerife to communication satellites is developed. The concept is based on laser guide star AO and aims for 24/7 operations. In the design phase of the project, reliable simulations of the laser guide star AO system are crucial. Hence, the laser propagation in a turbulent atmosphere has to be modelled. This usually involves physical optics propagation methods and thus phase screens which contain $2\pi$ ambiguities. The simulations for ALASCA are carried out via the PyrAmid Simulator Software for Adaptive opTics Arcetri (PASSATA)\cite{Agapito2016}, which is IDL and CUDA based object oriented software for Monte-Carlo end-to-end AO simulations. Figure~\ref{fig_ALASCA_data} shows a typical wrapped phase screen occurring during daytime conditions, i.e., strong turbulence, simulated with PASSATA, for which no ground truth unwrapped phase is available.

\begin{figure}[ht!]
\centering
\small
\resizebox{\columnwidth}{!}{%
    \begin{tabular}{ccc}
        \textbf{SH-WFS} &  \textbf{PWFS + NOPE} &  \textbf{MATLAB}\\
        \includegraphics[width=0.3\textwidth,trim = {13cm 0.5cm 12cm 1cm}, clip=true]{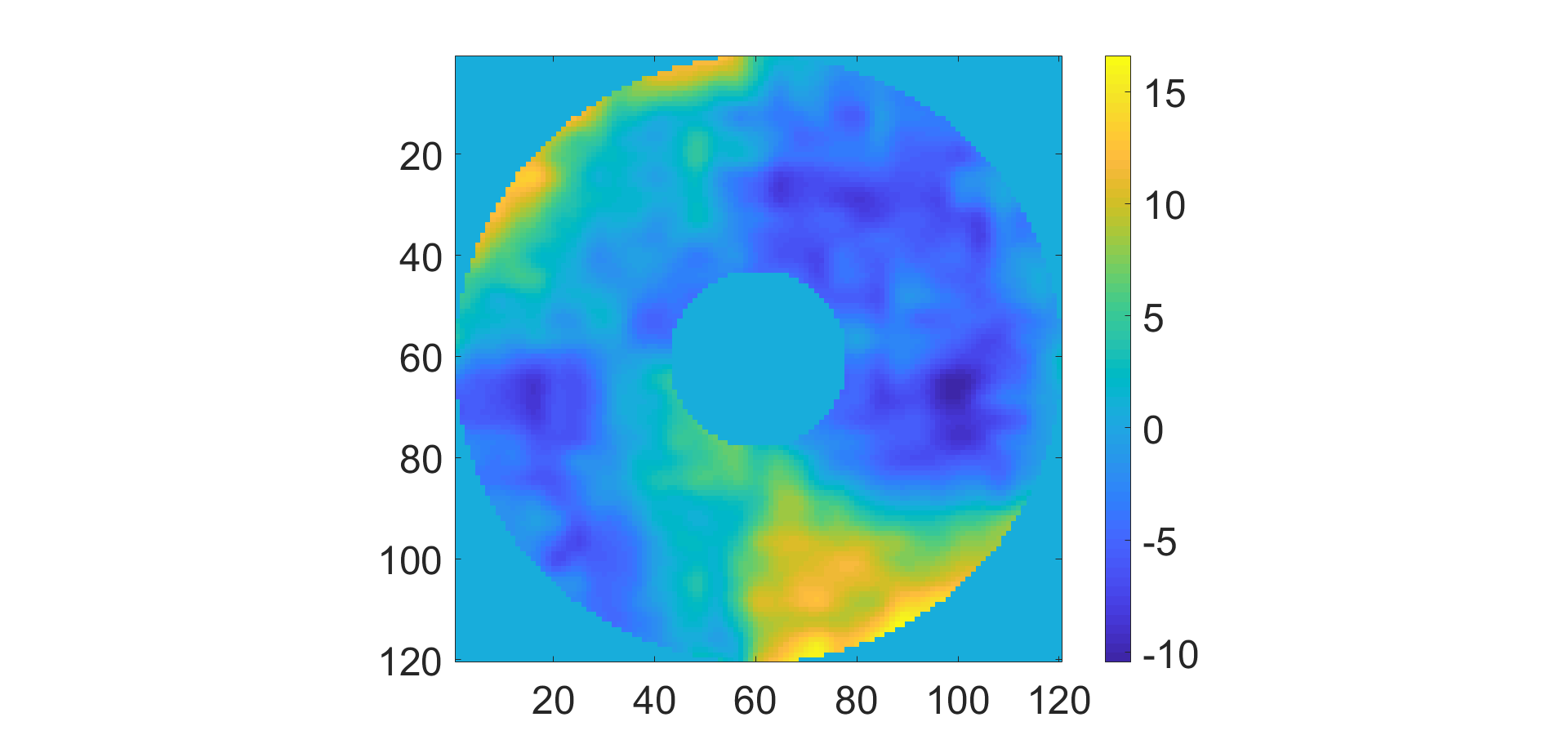} &
        \includegraphics[width=0.3\textwidth,trim = {13cm 0.5cm 12cm 1cm}, clip=true]{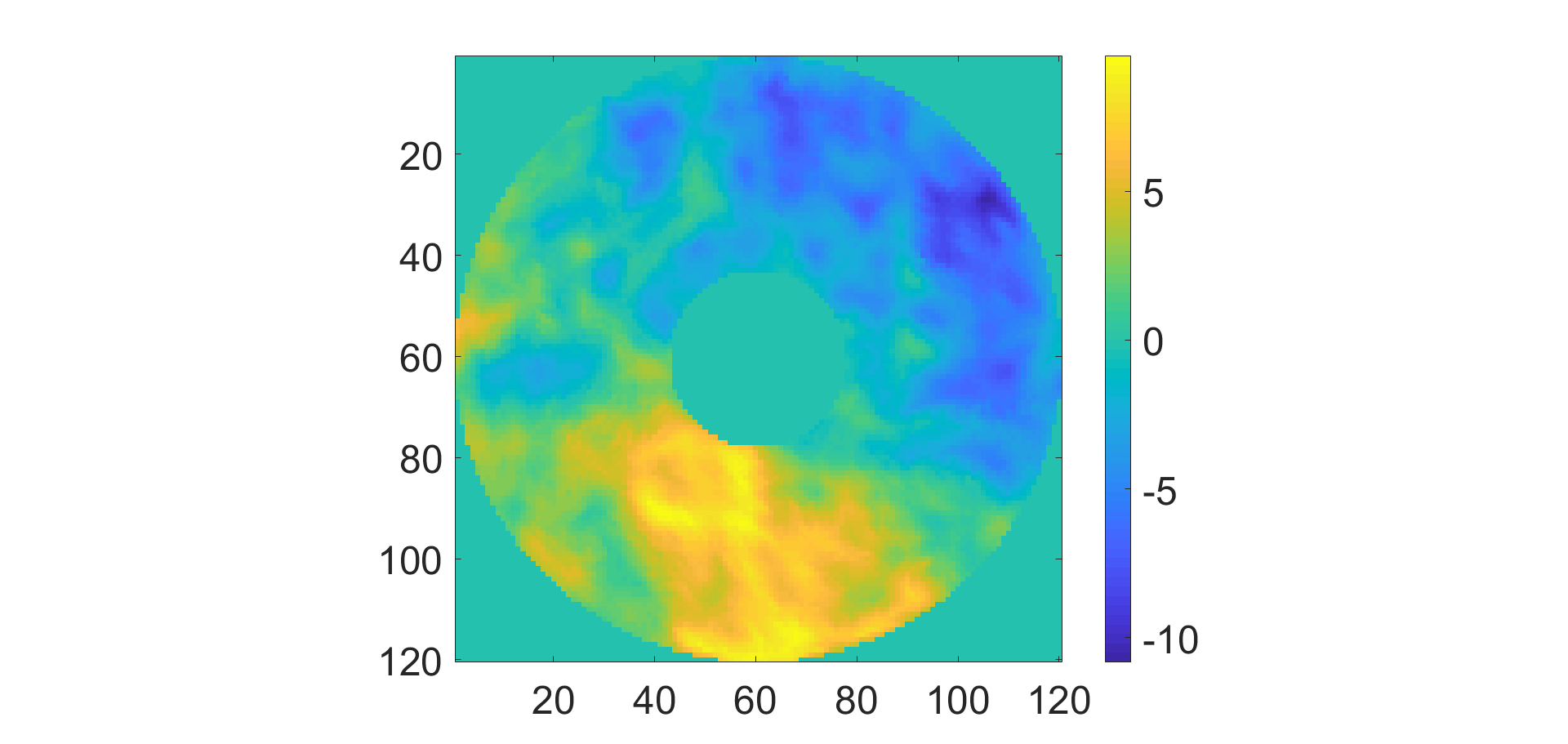} 
        &
        \includegraphics[width=0.3\textwidth,trim = {13cm 0.5cm 12cm 1cm}, clip=true]{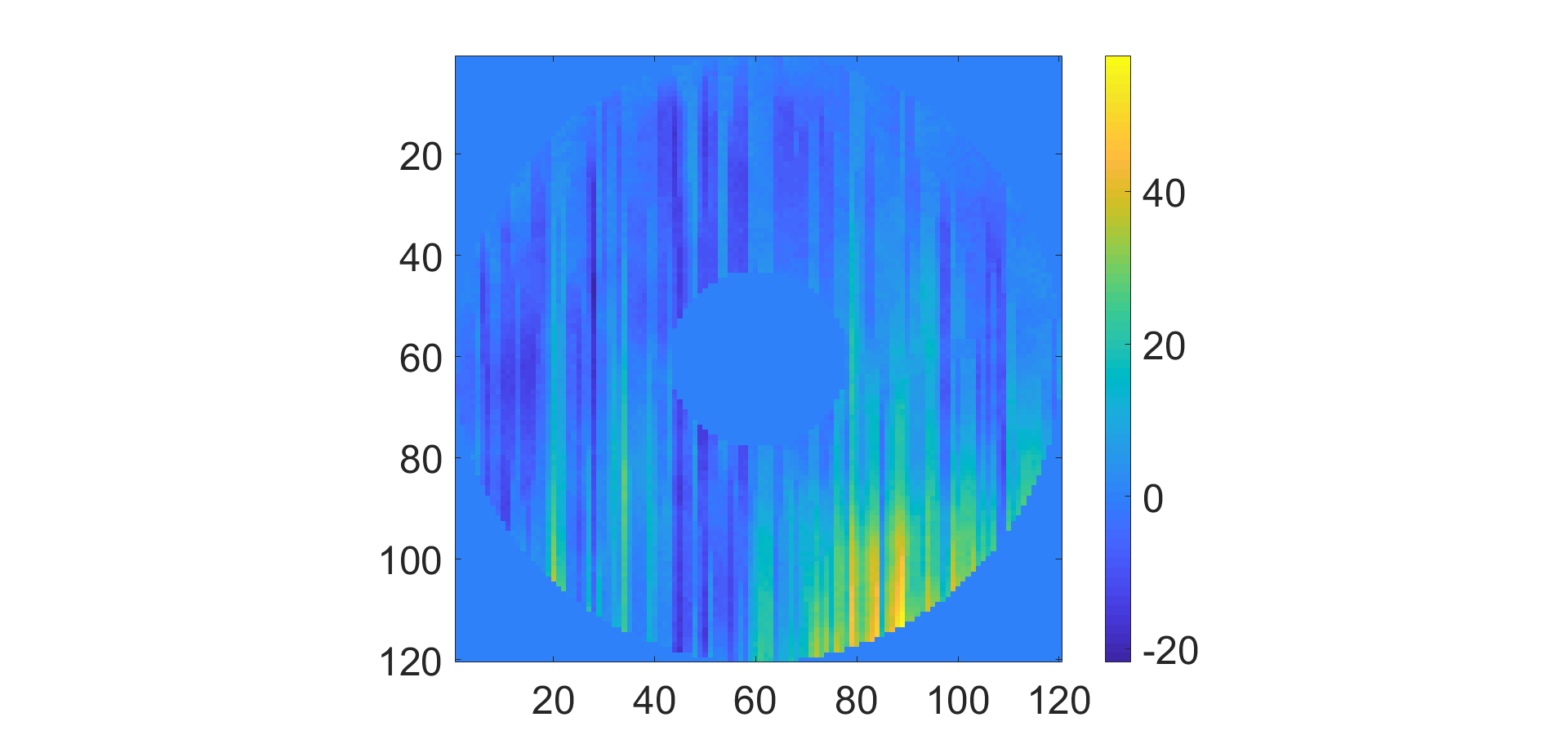} 
        \\
        \textbf{MRP} & \textbf{PE} &  \textbf{TIE}\\
        \includegraphics[width=0.3\textwidth,trim = {13cm 0.5cm 12cm 1cm}, clip=true]{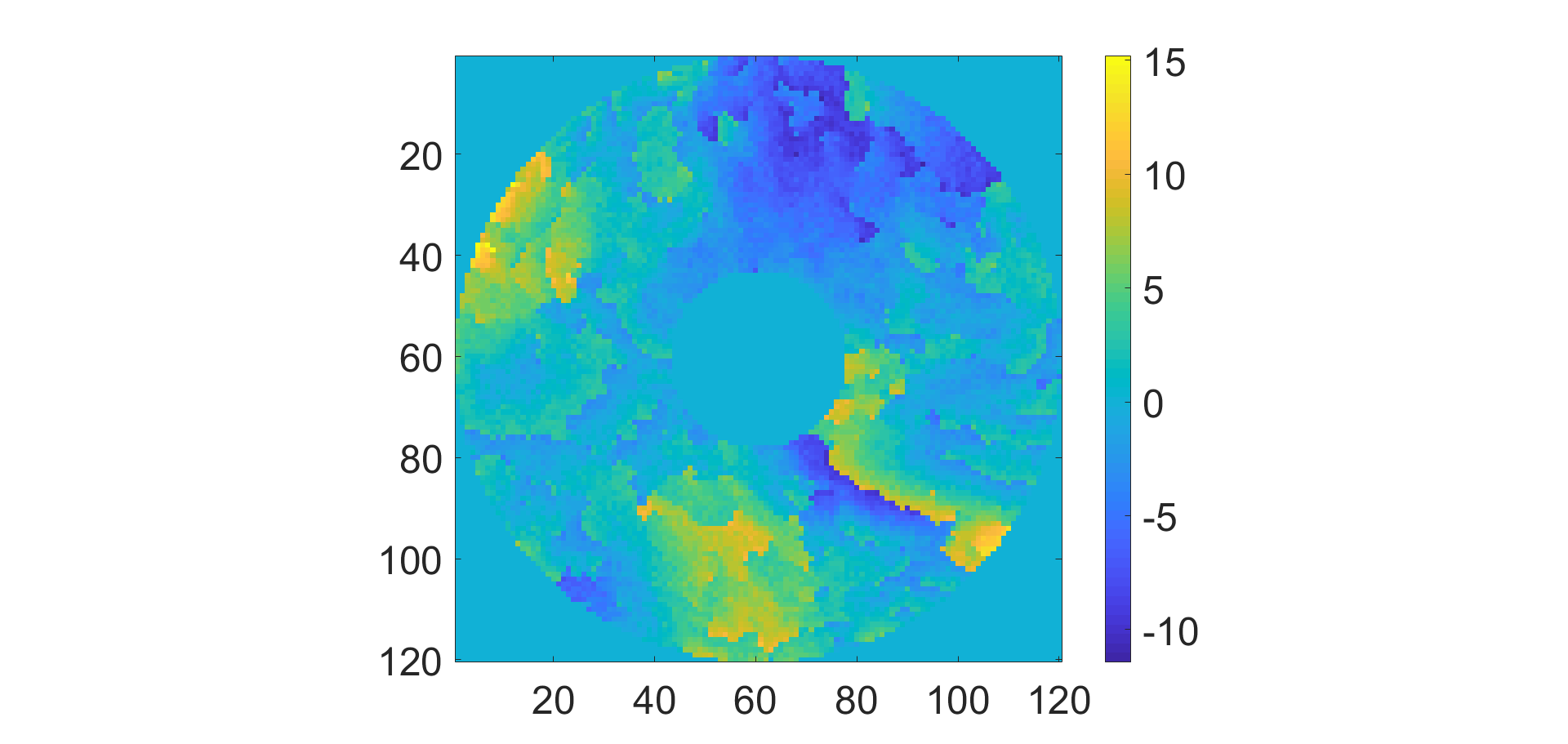} & 
        \includegraphics[width=0.3\textwidth,trim = {13cm 0.5cm 12cm 1cm}, clip=true]{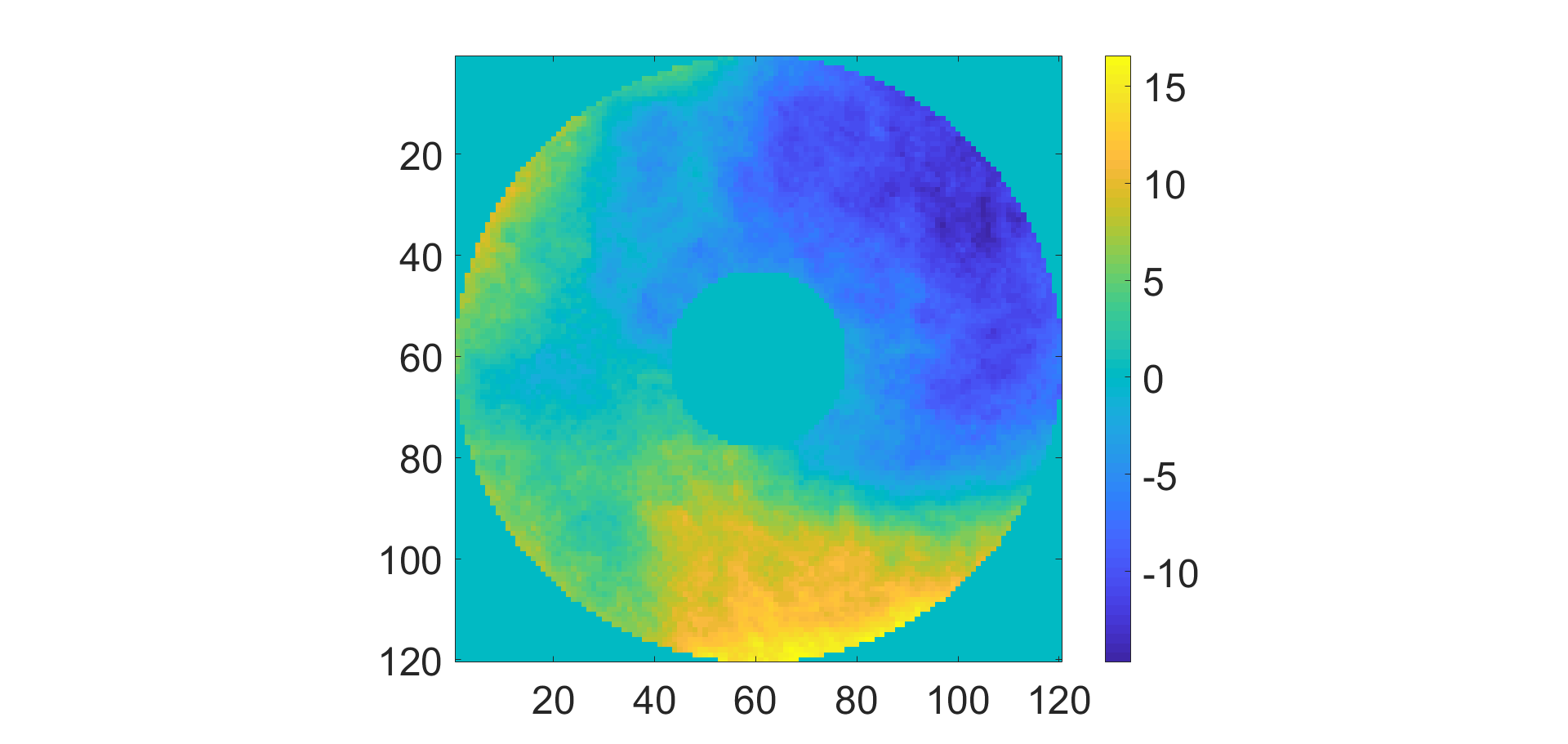} & 
        \includegraphics[width=0.3\textwidth,trim = {13cm 0.5cm 12cm 1cm}, clip=true]{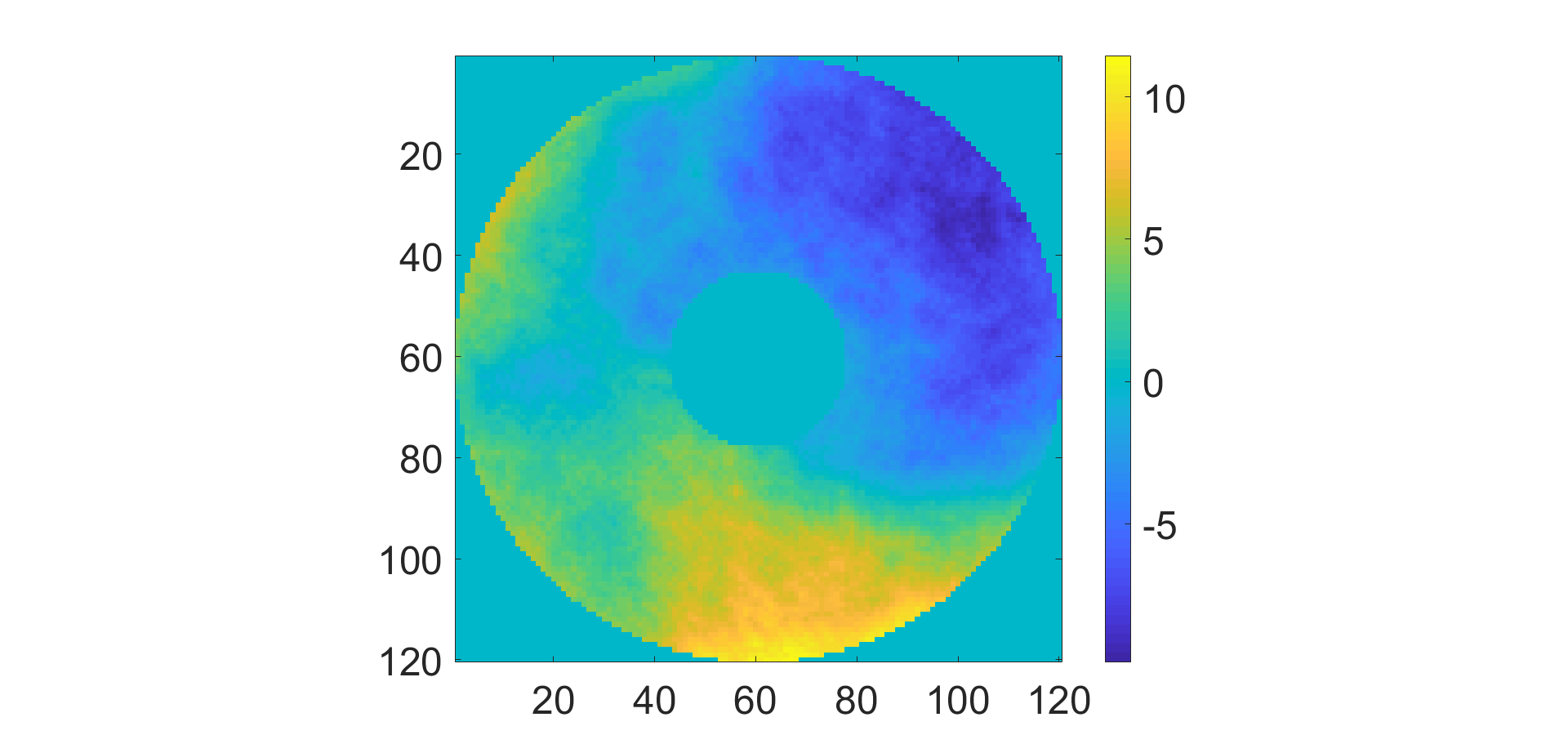}
    \end{tabular}
    }
    \caption{Unwrapped phases $\phir$ obtained via different phase unwrapping approaches applied to the noisy wrapped phase $\phiwd$ depicted in Figure~\ref{fig_ALASCA_data}. From left top to right bottom: Digital SH-WFS unwrapping, i.e., Algorithm~\ref{algo_dSH_WFS}, digital 4-sided PWFS unwrapping, i.e., Algorithm~\ref{algo_dFb_WFS}, MATLAB built-in function, unwrapping based on most reliable pixels (MRP), the Poisson equation (PE) and on the transport of intensity equation (TIE).}
    \label{fig_ALASCA_results}
\end{figure}

In Figure~\ref{fig_ALASCA_results}, we compare the unwrapped phases $\phir$ produced by the four state-of-the-art methods considered above to the two digital WFS approaches which performed best on simulated data: The digital SH-WFS, and the digital 4-sided PWFS using the NOPE and linear starting. Due to the smaller size of the wrapped phase data, the shown results for the digital SH-WFS were obtained using $30 \times 30$ subapertures. As before, the built-in MATLAB function and the MRP method produce unacceptable reconstructions, with either line or patch artefacts corrupting the phase. The other methods perform broadly similar, recovering the same general structure of the phase but with a varying amount of resolution. While due to the absence of a ground truth we are unable to say which of the unwrapping methods performs best, we nevertheless see that our digital WFS-based approaches produce unwrapped phases similar to those of the PE and TIE methods, and at a similar computational speed.

% % % % % % % % % % % % %
% Section - Conclusion  %
% % % % % % % % % % % % %
\section{Conclusion}\label{sect_concl}

In this paper, we introduced a new class of phase unwrapping algorithms based on the concept of digital WFSs. In their derivation, the key observation was that when considered as wavefront aberrations in an optical system, wrapped and unwrapped phases produce the same sensor measurements. Since reconstruction algorithms for WFSs are typically optimized to provide smooth wavefront reconstructions, they can be applied to compute an unwrapped phase from digitally computed WFS measurements. This procedure was then elaborated in detail for two classes of digital WFSs, based on both the SH and Fourier-type WFSs. Numerical examples on a real-world adaptive optics unwrapping problem demonstrate that the obtained results are comparable to, and in some cases exceed, the performance of other state-of-the-art unwrapping algorithms.

% % % % % % % % % % %
% Section - Support %
% % % % % % % % % % %
\section{Support}

This research was funded in part by the Austrian Science Fund (FWF) SFB 10.55776/F68 ``Tomography Across the Scales'', project F6805-N36 (Tomography in Astronomy) and project F6807-N36 (Tomography with Uncertainties), as well as 10.55776/P34981 ``New Inverse Problems of Super-Resolved Microscopy (NIPSUM)''. For open access purposes, the authors have applied a CC BY public copyright license to any author-accepted manuscript version arising from this submission. Furthermore, the authors were supported by the Austrian Research Promotion Agency (FFG) project number FO999888133, as well as the NVIDIA Corporation Academic Hardware Grant Program. The financial support by the Austrian Federal Ministry for Digital and Economic Affairs, the National Foundation for Research, Technology and Development and the Christian Doppler Research Association is gratefully acknowledged.

% % % % % % % % %
% Bibliography  %
%% % % % % % % % %
\bibliographystyle{plain}
{\footnotesize
\bibliography{mybib,bib_fourier_clear,bib_bernadett}
}

\end{document}